\documentclass[a4paper,10pt,reqno]{amsart}
\usepackage{array}
\usepackage{amsmath,amssymb,amsthm,color}
\usepackage[english]{babel}
\usepackage[T1]{fontenc}
\usepackage[utf8]{inputenc}
\usepackage{enumitem}
\usepackage[a4paper,left=2.6cm,right=2.6cm,top=3cm,bottom=3.5cm]{geometry}
\usepackage{mathtools}
\usepackage{stmaryrd}
\usepackage{relsize}
\usepackage{lmodern}
\usepackage{graphicx}
\usepackage{xcolor}
\usepackage{centernot}
\usepackage{tikz}
\usepackage{tikz-cd}
\usetikzlibrary{decorations.markings}
\usepackage{hyperref}
\usepackage[toc,page]{appendix}
\hypersetup{linktocpage,colorlinks,linkcolor={blue},citecolor={blue}}

\setlist{nosep}

\numberwithin{equation}{section}

\newtheorem{theorem}{Theorem}
\newtheorem{proposition}{Proposition}[section]
\newtheorem{lemma}[proposition]{Lemma}
\newtheorem{corollary}[proposition]{Corollary}

\theoremstyle{definition}
\newtheorem{definition}[proposition]{Definition}

\theoremstyle{remark}
\newtheorem{remark}[proposition]{Remark}

\newcommand{\shHom}{\mathcal{H}\!{om}}
\newcommand{\shEnd}{\mathcal{E}\!{nd}}

\title[Symplectic duality for the constant term]{Symplectic duality for the constant term of the geometric Eisenstein series}
\author{Igor Chaban}
\address{Department of Mathematics, Columbia University, New York, NY 10027, USA}
\email{ic2614@columbia.edu}
\date{}

\begin{document}
\begin{abstract}
We study the cohomology of a quasimap space that categorifies the constant term of the geometric Eisenstein series for the mirabolic parabolic subgroup of $GL$ over the function field $\mathbb{F}_q(C)$ of a smooth projective curve $C$. This cohomology carries a natural action of an algebra of correspondences whose commutative subalgebra is the ring of regular functions on the Coulomb branch, which here is the $A_{n}$-surface singularity. A choice of rank-one local system on $C$ induces an action of the \'etale fundamental group on the Coulomb branch; the scheme-theoretic fixed locus carries a natural vector bundle. Our main result identifies the cohomology of the quasimap space with the local cohomology of this vector bundle, for a generic range of parameters.
\end{abstract}
\maketitle
\setcounter{tocdepth}{1}
\tableofcontents

\section{Introduction}\label{sec:intro}

\subsection{The geometric Eisenstein series}\label{sub:intro-motivation}
Eisenstein series and related functions appear naturally in the study of automorphic forms over global fields. In the function-field case, they have a geometric counterpart given by sheaf-theoretic functors between moduli stacks of bundles. In this paper we study the constant term of a particular Eisenstein series for the general linear group. The cohomology of the corresponding sheaf carries an action of a Hecke-type modification algebra. Remarkably, we show that it is isomorphic to an algebra of differential operators on a singular supervariety. Moreover, the constant term is encoded in the local cohomology of its resolution.

Let \(C\) be a smooth projective curve over \(\mathbb F_q\), and let
\(G=GL_{n+1}\). We consider the mirabolic parabolic \(P\subset G\), whose
Levi quotient is \(GL_n\). Laumon's compactification of
\(\operatorname{Bun}_P\) gives the geometric Eisenstein series correspondence \cite{Laumon90, Kuznetsov1996TheLR, BravermanGaitsgory02}
\[
\operatorname{Bun}_{M}
\xleftarrow{}
\overline{\operatorname{Bun}}_{P}
\xrightarrow{}
\operatorname{Bun}_{G}.
\]
Fibers of the map to $\mathrm{Bun}_{G}$ are also referred to as twisted quasimaps \cite{CIOCANFONTANINE20103022, Ciocan-Fontanine:2011gqf, KimBumsig} from $C$ to the projective space $\mathbb{P}^{n}$.
The constant term~\cite{DrinfeldGaitsgory16} along the Borel is obtained by a similar correspondence
\[
\operatorname{Bun}_{G}
\xleftarrow{}
\operatorname{Bun}_{B}
\xrightarrow{}
\operatorname{Bun}_{T}.
\]
The value
of the constant term of the Eisenstein series is then given by the trace of Frobenius on the cohomology of a fiber of the map
\[
\overline{\operatorname{Bun}}_P
\times_{\operatorname{Bun}_G}
\operatorname{Bun}_B \rightarrow \mathrm{Bun}_{T}.
\]
The fiber over a \(T\)-bundle
\[
\vec{\mathcal M}=(\mathcal M_1,\ldots,\mathcal M_{n+1}) \in \mathrm{Bun}_{T}
\]
is the stack \(QM\vec N_{\vec{\mathcal M}}\)
of triples
\[
(\mathcal V^\bullet,\mathcal L,s),
\]
where
\[
0=\mathcal V^{(0)}\subset\mathcal V^{(1)}\subset\cdots
\subset\mathcal V^{(n+1)}=\mathcal V
\]
is a flag of vector bundles with
\[
\mathcal V^{(k)}/\mathcal V^{(k-1)}\cong\mathcal M_k,
\]
and
\[
s:\mathcal L\hookrightarrow \mathcal V
\]
is an injective sheaf map. We regard
\[
 H_c^*(QM\vec N_{\vec{\mathcal M}})
\]
as the categorified value of the constant term at \(\vec{\mathcal M}\).
The stack \(QM\vec N_{\vec{\mathcal{M}}}\) carries a Hecke-type modification correspondence: one replaces
the line bundle \(\mathcal L\) by \(\mathcal L(-x)\) at a point
\(x\in C\). Pullback and pushforward along this correspondence give \emph{modification/monopole operators} \cite{Finkelberg1997GlobalIC, FFNR11, Bullimore:2016hdc, KazOko23, Hilburn:2020aau}
\[
e\langle \alpha\rangle,\qquad f\langle \beta\rangle, \qquad \alpha, \, \beta \in H^{*}_{c}(C)
\]
on \( H_c^*(QM\vec N_{\vec{\mathcal M}})\) which generate an algebra \(\mathcal A\). We show that the subalgebra of
\(\mathcal A\) corresponding to $H^{2}(C)$ is the algebra of regular functions on the \(A_n\)-surface:
\[
\mathcal A^{(2)}
\cong
\overline{\mathbb Q}_l[x,y,\eta]/(xy-\eta^{n+1}).
\]
The degree-one operators form a spin module over this
surface, and the degree-zero operators act as first-order differential
operators on it.

\subsection{Main results}
\subsubsection{Symplectic dual description}
Let
\[
X_0^\vee=\operatorname{Spec}\mathcal{A}^{(2)}
\]
and let
\[
\rho:X^\vee\to X_0^\vee
\]
be its minimal resolution. $X_0^\vee$ is the \emph{Coulomb branch} of a 3-dimensional $\mathcal{N} = 4$ SUSY gauge theory \cite{Nakajima15, BFN:2016wma, BFN2017}, which is the \emph{symplectic/3d-mirror} dual to the \emph{Higgs branch} $X = T^{*}\mathbb{P}^{n}$ \cite{KamnitzerJoel22, webster20233dimensionalmirrorsymmetry}; one can view the zero section $\mathbb{P}^{n} \subset X$ as the target of the quasimaps. The surface $X^{\vee}$ carries an action of a two-dimensional torus, coming from the cohomological and degree gradings on $H^{*}_{c}(QM\vec{N})$. This action preserves a natural symplectic form on $X^{\vee}$, so the attracting set $L\subset X^{\vee}$ is Lagrangian.

A choice of Lagrangian subspace in $H_{c}^{1}(C)$ gives rise to a supercommutative locally free $\mathcal{O}_{X^\vee}$-algebra $\hat{\mathcal O}_{\mathrm{vir}}$.
The operators in $\mathcal{A}$ coming from $H^{0}_{c}(C)$ can be identified with differential operators on the twisted sheaf
\[
\hat{\mathcal O}_{\mathrm{vir}} \otimes \mathcal O(\vec{\mathcal M});
\]
here \(\mathcal O(\vec{\mathcal M})\) is a certain equivariant line bundle depending on the tuple \(\vec{\mathcal M}\), whose restrictions to the exceptional curves have degree \(\deg \mathcal{M}_k-\deg \mathcal{M}_{k+1}\). We give an interpretation of the cohomology of quasimaps as the local cohomology of $\hat{\mathcal O}_{\mathrm{vir}} \otimes \mathcal O(\vec{\mathcal M})$ with support in the Lagrangian subvariety $L$. To make the following isomorphism torus-equivariant, we introduce a suitable shift~\eqref{Hcent} of the cohomology of quasimaps, denoted by $\hat H^{*}_{c}(QM\vec{N})$.

\begin{theorem}[Dual description for cohomology of quasimaps, trivial character]\label{thm:main-intro}
There is a torus-equivariant isomorphism of \(\mathcal A \)-modules
\[
 \hat{H}_c^*(QM\vec N_{\vec{\mathcal M}})
\cong
H_L^*
\left(
X^\vee,
\hat{\mathcal O}_{\mathrm{vir}}\otimes
\mathcal O(\vec{\mathcal M})
\right).
\]
\end{theorem}
As a consequence, the constant term can be expressed as the equivariant Euler characteristic of $\hat{\mathcal O}_{\mathrm{vir}}\otimes\mathcal O(\vec{\mathcal M})$ on $X^{\vee}$. This generalizes the analysis of \cite{KazOko23}, where the case of $G = GL(2)$ is considered.
The proof of Theorem~\ref{thm:main-intro} proceeds in two steps. First, we establish filtrations on both sides and explicitly compute the associated graded; we then compare the extensions. The filtration on the right-hand side comes from the decomposition of $L$ into the attractors of the $n+1$ fixed points. The filtration on $\hat{H}^{*}_{c}(QM\vec{N})$ is provided by a Bia{\l}ynicki-Birula decomposition of $QM\vec{N}$ into attractor orbits $\mathrm{Attr}(k)$, $k = 1, \ldots, n+1$. Importantly, to set up the filtration, we use the purity of the cohomology, which follows from the smoothness of the attractors. Note that the entire stack $QM\vec{N}$ need not be smooth.

\subsubsection{Nontrivial character}
The construction generalizes to cohomology of quasimaps with coefficients in nontrivial local systems. We consider rank-one local systems $\chi$ on $QM\vec{N}_{\vec{\mathcal{M}}}$ pulled back from $\mathrm{Pic}(C)$ by the natural map~\eqref{Ns}. The local systems on $\mathrm{Pic}(C)$ are prescribed by characters of $\mathrm{Pic}(C)(\mathbb{F}_{q})$ and correspond, by geometric class field theory, to local systems $\chi$ on the curve. On the dual side this gives an action of the geometric fundamental group $\pi_{1}^{g}$ on the $A_{n}$-surface and its resolution. The scheme-theoretic fixed locus of the singular surface is the nonreduced thick point
\[
\operatorname{Spec}\overline{\mathbb Q}_l[\eta]/(\eta^{n+1}),
\]
whereas the fixed locus $(X^{\vee})^{\pi_{1}^{g}}$ of the resolved surface consists of $n+1$ reduced points. Similarly to the trivial-character case, one can define a sheaf $\hat{\mathcal O}_{\mathrm{vir}, \chi}$ on the fixed locus, a centered cohomology $\hat H^{*}_{c}(QM\vec{N}, \chi)$, and an algebra $\mathcal A_{\chi}$ of modification operators acting on the cohomology.

Let $\mathcal{K}_{C}$ denote the canonical bundle of $C$, and write $\chi_C(\mathcal F)= h^0(C,\mathcal F)-h^1(C,\mathcal F)$ for the Euler characteristic of a coherent sheaf $\mathcal F$ on $C$.

\begin{theorem}[Dual description for cohomology of quasimaps, nontrivial character]\label{thm:main}
There is an isomorphism of $\mathcal{A}_{\chi}$-modules
\begin{align*}
    \hat{H}^{*}_{c}(QM\vec{N}, \chi) \cong H^{*}_{L}((X^{\vee})^{\pi_{1}^{g}}, \hat{\mathcal{O}}_{\mathrm{vir}}\otimes \mathcal{O}(\vec{\mathcal{M}})),
\end{align*}
unless $\chi_{0}^{n+1} \neq 1$ and there exist indices
$1 \leq k < s \leq n+1$ with
\begin{align}\label{exceptionalcond}
\begin{cases}
    \mathcal{M}_{k} \cong \mathcal{M}_{s},\\
    \chi_C\left( (\mathcal{V}^{(s)}/\mathcal{V}^{(k)}) \otimes \mathcal{M}_{k}^{-1} \otimes \mathcal{K}_{C} \right) = 1,\\
    \mathrm{Hom}(\mathcal{M}_{l}, \mathcal{M}_{k}) = 0, \; l < k,\\
    \mathrm{Ext}^{1}(\mathcal{M}_{l}, \mathcal{M}_{s}) = 0, \; k < l < s.
\end{cases}
\end{align}
\end{theorem}

In any case, the associated graded of the Bia{\l}ynicki-Birula filtration on $\hat H^{*}_{c}(QM\vec N,\chi)$ matches $H^{*}_{L}((X^{\vee})^{\pi_{1}^{g}}, \hat{\mathcal O}_{\mathrm{vir},\chi}\otimes \mathcal O(\vec{\mathcal M}))$; the conditions~\eqref{exceptionalcond} are exactly the obstructions to splitting this filtration. 

\subsection{Organization}

Section~\ref{sec:quasimaps} defines the derived stack \(QM\vec N_{\vec{\mathcal M}}\) and computes its tangent complex. We then introduce the cohomology of $QM\vec{N}$ with coefficients in a local system.

Section~\ref{sec:mod-operators} constructs the modification
correspondence and the operators \(e\langle\alpha\rangle\),
\(f\langle\beta\rangle\). We prove the commutator and Casimir
relations satisfied in the algebra \(\mathcal A_\chi\) for an arbitrary character $\chi$.

Section~\ref{sec:geom-interp} interprets \(\mathcal A\) geometrically:
the degree-two part is the algebra of regular functions on the \(A_n\)-surface, the degree-one part is
a Clifford algebra, and the degree-zero part is described by first-order
differential operators.

Section~\ref{sec:resolution} constructs the resolution \(X^\vee\) by gluing its
affine charts. We describe the chart algebras in terms of the fixed
components of the quasimap stack.

Section~\ref{sec:cohomology-Xvee} computes the local cohomology of
\(X^\vee\) with coefficients in
\(\hat{\mathcal O}_{\mathrm{vir}}\otimes\mathcal O(\vec{\mathcal M})\).

Section~\ref{sec:cohomology-QMN} compares this local cohomology with
\( H_c^*(QM\vec N_{\vec{\mathcal M}},\chi)\) for both the trivial and nontrivial cases. We first match the associated graded
pieces and then analyze the extension problem.

\subsection*{Acknowledgements}

I am grateful to Andrei Okounkov for suggesting this problem and for many illuminating discussions. I also thank Tommaso Botta, Ivan Danilenko, Andr\'es N\'u\~nez, Che Shen, and Irina Sizova for exciting conversations. I also wish to warmly thank Maria Ib\'a\~nez.

\section{The stack of quasimaps}\label{sec:quasimaps}

In this section we introduce the moduli stack $QM\vec{N}$ of quasimaps. We define its derived structure and compute the tangent complex, then introduce a local system and define a shift for its cohomology.

\subsection{Notation and conventions}\label{sub:setup}

Let $G = GL(n+1)$ and let $T \subset G$ be the diagonal torus. Fix a
geometrically smooth projective curve $C$ over $\mathbb{F}_{q}$ of
genus~$g$, an $\mathbb{F}_{q}$-point $p\in C$, and a tuple of line
bundles
\[
\vec{\mathcal{M}} = (\mathcal{M}_{1}, \ldots, \mathcal{M}_{n+1})
\in \mathrm{Bun}_{T}(C),
\qquad c_{1}(\mathcal{M}_{k}) = \mu_k, \qquad \deg(\mathcal{M}_{k}) = m_{k}.
\]
The point $p$ gives an identification
$\mathrm{Pic}_{d}(C) \cong \mathrm{Pic}_{0}(C)$, used throughout. Fix an embedding $\mathbb{F}_{q} \subset \overline{\mathbb{F}}_{q}$,
and write $C_{\overline{\mathbb{F}}_{q}}$ for the corresponding base change
and $\bar{p}$ for the induced geometric point. For a scheme or stack
$X$ over $\mathbb{F}_{q}$ we write $H^{*}(X)$ and $H^{*}_{c}(X)$
(with coefficients in an understood local system) for the $l$-adic
\'etale cohomology
$H^{*}_{\mathrm{et}}(X_{\overline{\mathbb{F}}_{q}})$ and compactly supported
cohomology $H^{*}_{\mathrm{et},c}(X_{\overline{\mathbb{F}}_{q}})$ respectively. We fix a square root of $q$, $q^{1/2} \in \overline{\mathbb{Q}}_{l}$, and use it to define the Tate twist $\boldsymbol{(}\frac{1}{2}\boldsymbol{)}$ on Frobenius modules. 

Let $\mathbf{p} = p_{*}1 \in H^{2}(C)$ denote the class of the point $p$. Fix a basis
\begin{align*}
    \gamma_{1}, \ldots, \gamma_{2g} \in H^{1}_{c}(C) = H^{1}_{c}(C, \overline{\mathbb{Q}}_{l}),
\end{align*}
satisfying
\begin{align*}
    \gamma_{g + i} = \gamma_{i}^{\vee}, \qquad i = 1, \ldots, g,
\end{align*}
where the dual basis $\{ \gamma_{i}^{\vee} \}_{i = 1}^{2g}$ to
$\{ \gamma_{i} \}_{i = 1}^{2g}$ is defined by
\begin{align*}
    \gamma_{i} \cup \gamma_{j}^{\vee} = \delta_{ij}\mathbf{p}.
\end{align*}
For a rank~$1$ local system $\chi_{0}$ on $C$ we also fix a basis
\begin{align*}
    \gamma_{1}^{\chi}, \ldots, \gamma_{2g-2}^{\chi}
    \in H^{1}_{c}(C, \chi^{n+1}_{0})
\end{align*}
and denote the dual basis by
\begin{align*}
     \gamma_{1}^{\chi^{-1}}, \ldots, \gamma_{2g-2}^{\chi^{-1}}
     \in H^{1}_{c}(C, \chi^{-n-1}_{0}).
\end{align*}
We write
\begin{equation*}
    \Theta = \sum_{i = 1}^{g} \gamma_{i}\, \gamma_{i}^{\vee}
    \in H^{2}(\mathrm{Pic}_{0}(C))
\end{equation*}
for the theta divisor, and
\begin{align}
    \tau = c_{1}(TC) \in H^{2}(C)
\end{align}
for the first Chern class of the tangent bundle of the curve.

Consider the projection to the Picard variety
\begin{equation}\label{piLmap}
   \pi_{\mathcal{L}}\colon QM\vec{N}_{d} \to \mathrm{Pic}_{-d}(C) \cong \mathrm{Pic}_{0}(C),
   \qquad (\mathcal{V}^{\bullet}, \mathcal{L}, s) \mapsto \mathcal{L}.
\end{equation}
With a slight abuse of notation, we use the same letter $\mathcal{L}$
for the universal line bundle (the pullback of the universal line bundle from
$C \times \mathrm{Pic}(C)$ under $\mathrm{id}_{C} \times \pi_{\mathcal{L}}$),
and $\mathcal{V}$ for the universal vector bundle on
$C \times QM\vec{N}$.

\begin{lemma}\label{lem:lambda}
The first Chern class of $\mathcal{L}$ is given by
\begin{equation*}
    \lambda \overset{\mathrm{def}}{=} c_{1}(\mathcal{L})
    = \deg \mathcal{L} \, \mathbf{p} \otimes 1
    + \sum_{i = 1}^{2g} \gamma_{i} \otimes \gamma_{i}^{\vee},
\end{equation*}
where $\gamma_{i}^{\vee}$ is identified with its image under
$H^{1}(C) \cong H^{1}(\mathrm{Pic}_{0}(C))
\xrightarrow{\pi_{\mathcal{L}}^{*}} H^{1}(QM\vec{N})$.
\end{lemma}

\subsection{The stack $QM\vec{N}$}\label{sub:QMN}

\begin{definition}\label{defn:quasimap}
For $\vec{\mathcal{M}} \in \mathrm{Pic}(C)^{n+1}$, an \emph{$\vec{\mathcal{M}}$-twisted quasimap} to $\mathbb{P}^{n}$ is a triple
$(\mathcal{V}^{\bullet}, \mathcal{L}, s)$ consisting of a flag of vector bundles on $C$
\begin{align*}
    \mathcal{V}^{\bullet}
    = \left( 0 = \mathcal{V}^{(0)} \subset \mathcal{V}^{(1)}
    \subset \cdots \subset \mathcal{V}^{(n)}
    \subset \mathcal{V}^{(n+1)}
    \overset{\mathrm{def}}{=}\mathcal{V} \right)
\end{align*}
realised by successive extensions
\begin{align}\label{VviaM}
    0 \rightarrow \mathcal{V}^{(k-1)} \rightarrow \mathcal{V}^{(k)}
    \rightarrow \mathcal{M}_{k} \rightarrow 0,
    \qquad k = 1, \ldots, n+1;
\end{align}
a line bundle $\mathcal{L}$ on $C$; and a sheaf inclusion
\begin{equation*}
       s\colon \mathcal{L} \rightarrow \mathcal{V}.
\end{equation*}
Two quasimaps $(\mathcal{V}^{\bullet}, \mathcal{L}, s)$ and
$((\mathcal{V}^{\bullet})', \mathcal{L}', s')$ are isomorphic
if there exist isomorphisms
\begin{align*}
    &\varphi_{\mathcal{L}}\colon \mathcal{L} \rightarrow \mathcal{L}',\\
    &\varphi_{\mathcal{V}^{\bullet}}\colon \mathcal{V}
    \rightarrow \mathcal{V}', \nonumber
\end{align*}
with $\varphi_{\mathcal{V}^{\bullet}}$ preserving the grading and
inducing the identity on the graded pieces $\mathcal{M}_{l}$, such
that $s' = \varphi_{\mathcal{V}^{\bullet}} \circ s
\circ \varphi_{\mathcal{L}}^{-1}$. In particular, for any nonzero
scalar $\lambda$ the triples $(\mathcal{V}^{\bullet}, \mathcal{L}, s)$
and $(\mathcal{V}^{\bullet}, \mathcal{L}, \lambda s)$ are isomorphic.
\end{definition}

We let
\[
    QM\vec{N} = QM\vec{N}_{\vec{\mathcal{M}}}(C)
    = \{ (\mathcal{V}^{\bullet}, \mathcal{L}, s)\}
\]
denote the moduli stack of quasimaps. The \emph{degree} of a quasimap
$(\mathcal{V}^{\bullet}, \mathcal{L}, s)$ is $d \overset{\mathrm{def}}{=} -\deg \mathcal{L}$;
$QM\vec{N}$ decomposes into connected components
$QM\vec{N}_{d}$, $d \in \mathbb{Z}$. The stack $QM\vec{N}$ is algebraic with unipotent stabilizers. Indeed, the projection onto the first two components maps $QM\vec{N}$ to the scheme $\{\mathcal{V}^{\bullet}\} \times \mathrm{Pic}(C)$, where $\{\mathcal{V}^{\bullet}\}$ is the affine space of successive extensions. A fiber of this map over $(\mathcal{V}^{\bullet}, \mathcal{L})$, if nonempty, is the quotient of $H^{0}(\mathcal{V} \otimes \mathcal{L}^{-1}) \setminus \{0\}$ by the automorphism group $\mathrm{Aut}(\mathcal{V}^{\bullet}) \times \mathrm{Aut}(\mathcal{L})$, which acts on $s$ by unipotent transformations and scalars, respectively. In other words, the fiber is the quotient of $\mathbb{P}\left(H^{0}(\mathcal{V} \otimes \mathcal{L}^{-1})\right)$ by the unipotent group $\mathrm{Aut}(\mathcal{V}^{\bullet})$, so any stabilizer, being a subgroup, is unipotent.

\begin{definition}\label{defn:Otaut}
For any quasimap $(\mathcal{V}^{\bullet}, \mathcal{L}, s)$ the section
$s$ determines a line in
$\mathrm{Hom}_{\mathcal{O}_{C}}(\mathcal{L}, \mathcal{V})
\cong H^{0}(C, \mathcal{V} \otimes \mathcal{L}^{-1})$.
These lines assemble into a line bundle
$\mathcal{O}_{QM\vec{N}}(-1)$ on $QM\vec{N}$, whose dual we denote
$\mathcal{O}_{QM\vec{N}}(1)$. Let
\begin{align}
    \eta = c_{1}\left( \mathcal{O}_{QM\vec{N}}(1) \right) \in H^{2}(QM\vec{N})
\end{align}
denote its first Chern class.
\end{definition}

\subsection{Derived structure}
\label{sub:pot}
The algebraic stack $QM\vec{N}$ admits a natural quasi-smooth derived enhancement. In this subsection we define the derived structure and compute the tangent complex.

For the quiver $Q$
\begin{equation}\label{quiver}
\vcenter{\hbox{\tikzset{vert/.style={draw, circle, minimum size=7.5mm, inner sep=0pt, font=\footnotesize}, midarrow/.style={postaction={decorate, decoration={markings, mark=at position 0.6 with {\arrow{>}}}}}}\begin{tikzcd}[column sep=small, row sep=small]
        |[vert]| 1 & |[vert]| 2 & {} & {} & |[vert]| n{+}1 \\
        &&&& |[vert]| 1
        \arrow[from=1-1, to=1-2, no head, midarrow]
        \arrow[from=1-2, to=1-3, no head, midarrow]
        \arrow[dotted, no head, from=1-3, to=1-4]
        \arrow[from=1-4, to=1-5, no head, midarrow]
        \arrow[from=2-5, to=1-5, no head, midarrow]
    \end{tikzcd}}}
\end{equation}
with the specified dimension vector $\mathbf{d}$, consider the moduli stack of $Q$-representations over $\mathbb{F}_{q}$
\begin{align}\label{repmoduli}
    \mathfrak{M}_{\mathbf{d}} = \mathrm{Rep}_{Q}(\mathbf{d})/GL_{\mathbf{d}},
\end{align}
where 
\begin{align*}
    \mathrm{Rep}_{Q}(\mathbf{d}) = \mathrm{Hom}( \mathbb{A}^{1}_{\mathbb{F}_{q}}  ,\mathbb{A}^{n+1}_{\mathbb{F}_{q}} ) \times \prod\limits_{k =1}^{n}\mathrm{Hom}(\mathbb{A}^{k}_{\mathbb{F}_{q}} , \mathbb{A}^{k+1}_{\mathbb{F}_{q}} )
\end{align*}
is the affine space of all representations and
\begin{align*}
    GL_{\mathbf{d}} = \left( GL_{1} \times \prod\limits_{k = 1}^{n+1} GL_{k} \right)/\mathbb{F}_{q}
\end{align*}
is the \emph{gauge group}. There is an open substack
\begin{align*}
    \mathfrak{M}_{\mathbf{d}}^{\zeta} \subset \mathfrak{M}_{\mathbf{d}}
\end{align*}
defined by the stability condition $\zeta$, which requires all maps in $ \mathrm{Rep}_{Q}(\mathbf{d})$ to be injective. The mapping stack
\begin{align*}
    \mathrm{Map}(C, \mathfrak{M}_{\mathbf{d}})
\end{align*}
parametrizes the data of vector bundles $(\mathcal{V}^{(1)}, ..., \mathcal{V}^{(n+1)}, \mathcal{L}), \; \mathrm{rk} \, \mathcal{V}^{(k)} = k,\;\;  \mathrm{rk} \, \mathcal{L} = 1$ with sheaf maps $\mathcal{V}^{(k)} \overset{i_{k}}{\rightarrow} \mathcal{V}^{(k+1)}, \; \mathcal{L} \overset{s}{\rightarrow} \mathcal{V}^{(n+1)} = \mathcal{V}$. Consider an open substack of the mapping stack
\begin{align*}
    \mathrm{Map}(C, \mathfrak{M}_{\mathbf{d}})^{\zeta} \subset \mathrm{Map}(C, \mathfrak{M}_{\mathbf{d}})
\end{align*}
which is defined by the condition that the generic point maps to $\mathfrak{M}_{\mathbf{d}}^{\zeta}$. Equivalently, this condition means that the sheaf maps $i_{k},\, k = 1, ..., n$ and $s$ are injective. We can then consider a smaller open substack
\begin{align*}
    \mathrm{Map}(C, \mathfrak{M}_{\mathbf{d}})^{\zeta, \circ} \subset  \mathrm{Map}(C, \mathfrak{M}_{\mathbf{d}})^{\zeta}
\end{align*}
defined by the condition that the inclusions $\mathcal{V}^{(k)} \rightarrow \mathcal{V}^{(k+1)}$ have torsion-free quotients. Taking these quotients gives a map
\begin{align*}
    \pi_{T} \colon \mathrm{Map}(C, \mathfrak{M}_{\mathbf{d}})^{\zeta, \circ} \rightarrow \mathrm{Pic}(C)^{n+1}.
\end{align*}
Then the stack $QM\vec{N}_{\vec{\mathcal{M}}}$ is the fiber of this map over the point $\vec{\mathcal{M}} \in \mathrm{Pic}(C)^{n+1}$. Therefore, to define a derived structure on $QM\vec{N}$ one can consider the natural derived structure on the mapping stack and on its open substack, and then take the derived fiber. Let us compute the corresponding tangent complex. With a slight abuse of notation, we use $\mathcal{V}^{\bullet}, \mathcal{L}, s$ for the universal successive extensions, line bundle and sheaf map, respectively, on
\begin{align*}
    C \times \mathrm{Map}(C, \mathfrak{M}_{\mathbf{d}})^{\zeta, \circ}.
\end{align*}
By the quotient representation~\eqref{repmoduli}, the tangent complex of the moduli stack of $Q$-representations is
\begin{align*}
    &\mathbb{T}\mathfrak{M}_{\mathbf{d}} = \left[ \underline{\mathrm{Lie}(GL_{\mathbf{d}})} \rightarrow \underline{\mathrm{Rep}_{Q}(\mathbf{d})}\right][1]\\
    &(u, u_1 , ..., u_{n+1}) \mapsto  (u_{n+1} a - a u, u_2 a_1 - a_{1} u_1 , ..., u_{n+1}a_{n} - a_n u_{n}),
\end{align*}
where $(a, a_{1}, ..., a_{n})$ is the universal section of the trivial bundle $\underline{\mathrm{Rep}_{Q}(\mathbf{d})}$.  
Let
\begin{equation*}
     \mathrm{ev} \colon C \times \mathrm{Map}(C, \mathfrak{M}_{\mathbf{d}})^{\zeta, \circ} \rightarrow \mathfrak{M}_{\mathbf{d}}
\end{equation*}
be the evaluation map and let $p$ be the projection onto the second factor. Then the tangent complex of the mapping stack is
\begin{align*}
    \mathbb{T} \mathrm{Map}(C, \mathfrak{M}_{\mathbf{d}})^{\zeta, \circ} &= p_{*}\mathrm{ev}^{*} \mathbb{T}\mathfrak{M}_{\mathbf{d}} \nonumber\\ &= p_{*}\left[\shEnd(\mathcal{L}) \oplus \bigoplus_{k = 1}^{n+1} \shEnd(\mathcal{V}^{(k)}) \rightarrow \shHom(\mathcal{L}, \mathcal{V}) \oplus \bigoplus\limits_{k = 1}^{n} \shHom(\mathcal{V}^{(k)}, \mathcal{V}^{(k+1)})\right][1]
\end{align*}

Similarly, the pullback of the tangent complex of $\mathrm{Pic}(C)^{n+1}$ under $\pi_{T}$ is given by
\begin{align}\label{tpic}
    \pi_{T}^{*}\mathbb{T}\mathrm{Pic}(C)^{n+1} &= p_{*} \left[ \bigoplus_{k = 1}^{n+1}\shEnd(\mathcal{V}^{(k)}/\mathcal{V}^{(k-1)}) \right][1] \nonumber \\
    &\overset{\mathrm{qis}}{\cong} p_{*}\left[ \bigoplus_{k = 1}^{n+1} \shHom(\mathcal{V}^{(k)}, \mathcal{V}^{(k)}/\mathcal{V}^{(k-1)}) \rightarrow \bigoplus_{k = 2}^{n+1} \shHom(\mathcal{V}^{(k-1)}, \mathcal{V}^{(k)}/\mathcal{V}^{(k-1)})\right][1].
\end{align}

We have a distinguished triangle
\begin{equation*}
    \mathbb{T}_{\pi_{T}} \rightarrow  \mathbb{T} \mathrm{Map}(C, \mathfrak{M}_{\mathbf{d}})^{\zeta, \circ} \rightarrow  \pi_{T}^{*}\mathbb{T}\mathrm{Pic}(C)^{n+1} \rightarrow \mathbb{T}_{\pi_{T}}[1],
\end{equation*}
which in representation~\eqref{tpic} is given by the natural projection. Since this projection is component-wise epimorphic, the tangent complex of $\pi_{T}$ can be computed simply by taking kernels:
\begin{align*}
     \mathbb{T}_{\pi_{T}} = p_{*}\left[ \shEnd(\mathcal{L}) \oplus \bigoplus_{k = 1}^{n+1} \shHom(\mathcal{V}^{(k)}, \mathcal{V}^{(k-1)}) \rightarrow \shHom(\mathcal{L}, \mathcal{V}) \oplus \bigoplus_{k = 1}^{n}\shHom(\mathcal{V}^{(k)}, \mathcal{V}^{(k)}) \right] [1].
\end{align*}
Consider the sheaf
\begin{align*}
\shEnd(\mathcal{V}^{\bullet})
    = \{\psi \in \shEnd(\mathcal{V}):
    \psi(\mathcal{V}^{(k)}) \subset \mathcal{V}^{(k-1)}\}.
\end{align*}
It is easy to see that it fits into the short exact sequence
\begin{align*}
    0 \rightarrow \shEnd(\mathcal{V}^{\bullet}) \rightarrow \bigoplus_{k = 1}^{n+1} \shHom(\mathcal{V}^{(k)}, \mathcal{V}^{(k-1)}) &\rightarrow \bigoplus_{k = 1}^{n}\shHom(\mathcal{V}^{(k)}, \mathcal{V}^{(k)})\rightarrow 0\\
     (\psi_{k})_{k = 1}^{n+1} &\mapsto (\psi_{k+1}i_{k} - i_{k}\psi_{k})_{k=1}^{n+1},
\end{align*}
where the first map is given by restriction to the subbundles. Using it, we can rewrite the tangent complex of $\pi_{T}$ as
\begin{equation}\label{eq:TpiT}
     \mathbb{T}_{\pi_{T}} = p_{*}\left[ \shEnd(\mathcal{L}) \oplus \shEnd(\mathcal{V}^{\bullet}) \rightarrow \shHom(\mathcal{L}, \mathcal{V}) \right] [1].
\end{equation}
Then the tangent complex of $QM\vec{N}$ is the restriction of~\eqref{eq:TpiT} to the fiber over $\vec{\mathcal{M}}$.
We see that it is a perfect complex of Tor-amplitude $[-1, 1]$, hence $QM\vec{N}$
is \emph{quasi-smooth} \cite{Khan19} over $\mathbb{F}_{q}$. 
\subsection{Virtual tangent space}
It follows from~\eqref{eq:TpiT} that for a point $(\mathcal{V}^{\bullet}, \mathcal{L}, s)$ in $QM\vec{N}$, the tangent complex is given by the hypercohomology groups of the complex
\begin{align*}
    C^{\bullet} = \Big[ \overset{-1}{\shEnd(\mathcal{L})
    \oplus \shEnd(\mathcal{V}^{\bullet})}
    &\overset{\Delta_{s}}{\longrightarrow} \overset{0}{\shHom(\mathcal{L}, \mathcal{V})} \Big] \\
    (\varphi, \psi) &\mapsto \psi \circ s - s \circ \varphi.
\end{align*}
The hypercohomology groups $\mathbb{H}^{-1}, \mathbb{H}^{0}, \mathbb{H}^{1}$ are referred to as infinitesimal automorphisms, deformations, and obstructions. One can easily compute the infinitesimal automorphisms
\begin{align}
    \mathbb{H}^{-1}(C^{\bullet})
    &= \mathrm{Ker}\left( \mathrm{End}(\mathcal{L})
    \oplus \mathrm{End}(\mathcal{V}^{\bullet})
    \overset{\Delta_{s}}{\longrightarrow}
    \mathrm{Hom}(\mathcal{L}, \mathcal{V})\right),    
\end{align}
and, using Serre duality, the obstructions
\begin{align}
    \mathbb{H}^{1}(C^{\bullet})
    &= \mathrm{Coker}\left(
    H^{1}\left(\shEnd(\mathcal{L})\right)
    \oplus H^{1}\left(\shEnd(\mathcal{V}^{\bullet})\right)
    \overset{\Delta_{s}}{\longrightarrow}
    \mathrm{Ext}^{1}(\mathcal{L}, \mathcal{V})\right) \nonumber\\
    &=\mathrm{Ker}\left(
    H^{0}(\mathcal{V}^{\vee} \otimes \mathcal{L} \otimes \mathcal{K}_{C})
    \overset{\Delta_{s}^{\vee}}{\longrightarrow}
    H^{0}\left(\shEnd(\mathcal{L})\otimes \mathcal{K}_{C}\right)
    \oplus H^{0}\left(\shEnd(\mathcal{V}^{\bullet})^{\vee}
    \otimes \mathcal{K}_{C} \right)\right)^{\vee}, \label{obstructions}
\end{align}
where
\begin{align}\label{Deltaadj}
    \Delta_{s}^{\vee}(\alpha) = (-\alpha \circ s,\; s \circ \alpha).
\end{align}
We will compute later that at the point $(\mathcal{V}^{\bullet}_{\mathrm{fix}}, \mathcal{M}_{k}(-D), s)$:
\begin{align*}
    &\mathbb{H}^{-1}(C^{\bullet}) = \bigoplus_{l < r \neq k} H^{0}(\mathcal{M}_{l} \otimes \mathcal{M}_{r}^{-1}),\\
    &\mathbb{H}^{1}(C^{\bullet}) =  \bigoplus_{l > k} H^{1}(\mathcal{M}_{l}\otimes \mathcal{M}_{k}^{-1}(D)).
\end{align*}
The virtual tangent space at a point $(\mathcal{V}^{\bullet}, \mathcal{L}, s)\in QM\vec{N}$ is given by the $K$-theoretic difference
\begin{align}
    T^{\mathrm{vir}}_{(\mathcal{V}^{\bullet}, \mathcal{L}, s)} QM\vec{N}
    &= \mathbb{H}^{0}(C^{\bullet}) -
\mathbb{H}^{-1}(C^{\bullet}) - \mathbb{H}^{1}(C^{\bullet}) = H^{*}\left(\shHom(\mathcal{L}, \mathcal{V})\right) - H^{*} \bigl( \shEnd(\mathcal{L})
    \oplus \shEnd(\mathcal{V}^{\bullet}) \bigr)\nonumber
    \\
    &=\sum\limits_{l = 1}^{n+1}\mathrm{Ext}^{\bullet}(\mathcal{L},
    \mathcal{M}_{l})\otimes \mathcal{O}_{QM\vec{N}}(1)
    \big|_{(\mathcal{V}^{\bullet}, \mathcal{L}, s)}
    -\mathrm{Ext}^{\bullet}(\mathcal{L}, \mathcal{L})
    - \sum\limits_{l < r}\mathrm{Ext}^{\bullet}(\mathcal{M}_{r},
    \mathcal{M}_{l}). \label{TvirQMN}
\end{align}

As a corollary, we obtain the virtual
dimension of $QM\vec{N}_{d}$:
\begin{align*}
    \mathrm{virdim}\, QM\vec{N}_{d} = - \sum\limits_{l < r}(m_{l} - m_{r}+1-g) + \sum\limits_{l = 1}^{n+1}(m_l + d + 1 - g) + g-1.
\end{align*}

\subsection{The centered cohomology}\label{sub:fra-act}

Consider the map
\begin{align}\label{Ns}
    \det\colon QM\vec{N} &\rightarrow \mathrm{Pic}(C),\\
    (\mathcal{V}^{\bullet}, \mathcal{L}, s)
    &\mapsto \det(\mathcal{V}/\mathcal{L} \otimes \mathcal{L}^{-1})
    = \mathcal{M}_{1} \otimes \cdots \otimes \mathcal{M}_{n+1}
    \otimes \mathcal{L}^{-n-1}. \nonumber
\end{align}
Let $\chi_{0}$ be a one-dimensional local system on $C$ with a fixed
trivialisation $\chi_{0}\big|_{\bar{p}} \cong \overline{\mathbb{Q}}_{l}$.
The action of the geometric fundamental group
$\pi^{g}_{1} = \pi_{1, \mathrm{et}}(C_{\overline{\mathbb{F}}_{q}}, \bar{p})$ on the
fiber of $\chi_{0}$ over $\bar{p}$ is given by a character
\begin{align*}
   \chi_{0}\colon \pi^{g}_{1} \rightarrow \overline{\mathbb{Q}}_{l}^{\times}.
\end{align*}
Let $A^{\vee} \overset{a}{\cong} \overline{\mathbb{Q}}_{l}^{\times}$ be a
one-dimensional torus with coordinate $a$. We denote by $\chi$ the $A^{\vee}$-equivariant local system on $C$
obtained from $\chi_{0}$ by the fiberwise $A^{\vee}$-action given by
multiplication by $a^{-1}$.
Accordingly, the group $\pi_{1}^{g} \times A^{\vee}$ acts on the fiber of $\chi$ over
$\bar{p}$ by the character
\begin{align*}
    \chi = \chi_{0}\,a^{-1}\colon
    \pi_{1}^{g} \times A^{\vee} \rightarrow \overline{\mathbb{Q}}_{l}^{\times}.
\end{align*}

The local system $\chi_{0}$ on $C$ also gives rise to a
one-dimensional local system (still denoted $\chi_{0}$) on
$\mathrm{Pic}_{0}(C)$. Pulling it back under the map
 \begin{align*}
     \mathrm{Pic}(C) &\rightarrow \mathrm{Pic}_{0}(C),\\
     x &\mapsto x \otimes \mathcal{O}(-p)^{\otimes \deg x}, \nonumber
 \end{align*}
we obtain a local system on $\mathrm{Pic}(C)$, which we still call
$\chi$. We make $\chi$ on $\mathrm{Pic}$ into an $A^{\vee}$-equivariant
local system by declaring $A^{\vee}$ to act on the fibers over
$\mathrm{Pic}_{d}(C)$ with multiplication by $a^{-d}$.
Pulling back under~\eqref{Ns} yields a local system on $QM\vec{N}$,
which we also denote $\chi$.
The space $QM\vec{N}$ and its cohomology carry a
geometric Frobenius action; we denote by $\mathrm{Fr} \cong \mathbb{Z}$ the group
generated by this operator. 

In order to make the isomorphism of Theorem~\ref{thm:main-intro} cohomologically graded and $\mathrm{Fr}\times A^{\vee}$-equivariant, we introduce a cohomological shift and $\mathrm{Fr} \times A^{\vee}$-twist
\begin{align}\label{Hcent}
    \hat{H}^{*}_{c}(QM\vec{N}_{d}, \chi) = H^{*}_{c}(QM\vec{N}_{d}, \chi)[[\mathrm{virdim} \,QM\vec{N}_{d}]] \otimes a^{(n+1)(g-1)},
\end{align}
where
\begin{align}
    [[i]] = [i]\boldsymbol{(}i/2\boldsymbol{)}
\end{align}
denotes the pure combination of the cohomological shift and Tate twist, and $a^{(n+1)(g-1)}$ stands for the one-dimensional $A^{\vee}$-representation given by the indicated character.

A quasimap of degree
$d=-\deg\mathcal{L}$ is sent by~\eqref{Ns} to
$\det(\mathcal{V}/\mathcal{L} \otimes \mathcal{L}^{-1})$, which has degree
$\sum_{l = 1}^{n+1}(d + m_{l})$. We will show later that the cohomology
$H^{*}_{c}(QM\vec{N}, \chi)$ is pure, so the
Frobenius eigenvalues on $H^{i}_{c}$ are equal to $\pm q^{i/2}$. Therefore, up to sign, $\mathrm{Fr} \times A^{\vee}$ acts on the cohomology of quasimaps by
\begin{align}\label{A-action}
    \mathrm{Fr} \times A^{\vee}\colon &\hat{H}^{i}_{c}(QM\vec{N}_{d}, \chi)
    \rightarrow \hat{H}^{i}_{c}(QM\vec{N}_{d}, \chi),\\
    &\alpha \mapsto \# q^{i/2}(q^{1/2}a)^{ - \sum\limits_{l = 1}^{n+1} m_l + d + 1 - g}\,
    \alpha, \nonumber
\end{align}
where the constant $\#$ is equal to
\begin{equation*}
    \# = (q^{1/2})^{1 - g + \sum\limits_{l < r}(m_l - m_r + 1 - g)}.
\end{equation*}
 The factor 
 \begin{align}
     q^{\displaystyle -\rho_{P}(\vec{\mathcal{M}}\otimes\mathcal{L}^{-1})},  \;\;\; \rho_{P}(\vec{\mathcal{M}}\otimes\mathcal{L}^{-1}) \coloneqq \frac{1}{2} \sum\limits_{l = 1}^{n+1} \chi_C(\mathcal{M}_{l} \otimes \mathcal{L}^{-1})
 \end{align}
 in~\eqref{A-action} corresponds to a shift by the half-sum of roots in the unipotent radical of $P$. Such a shift is a general feature in the theory of Eisenstein series. The centering~\eqref{Hcent} can be viewed as a geometric way to introduce it.

\section{Algebra of correspondences}\label{sec:mod-operators}

The space $\hat{H}_{c}^{*}(QM\vec{N}, \chi)$ carries a family of natural operators
labelled by cohomology classes on~$C$. In this section we construct them
from a single closed-immersion correspondence on $QM\vec{N}$, compute
the resulting pull--push identities in terms of tautological Chern classes, and
derive relations among the corresponding
operators on cohomology.

\subsection{The modification correspondence}
\label{sub:modif-corr}
We introduce a correspondence $\iota$ on $C\times QM\vec{N}$, which is a special case of the action of a Coulomb branch~\cite{Finkelberg1997GlobalIC, Nakajima15, BFN:2016wma, BFN2017}. Then we compute relations in the algebra generated by pull-push operators along $\iota$.

\begin{definition}\label{defn:iota}
The \emph{modification correspondence} is the closed immersion
\begin{equation*}
\begin{aligned}
    \iota\colon C \times QM\vec{N}_{d} &\to C \times QM\vec{N}_{d+1},\\
    (x, \mathcal{V}^{\bullet}, \mathcal{L}, s)
    &\mapsto (x, \mathcal{V}^{\bullet}, \mathcal{L}(-x), s \otimes 1),
\end{aligned}
\end{equation*}
where $1\colon \mathcal{O}_{C}(-x) \hookrightarrow \mathcal{O}_{C}$ is
the canonical inclusion.
\end{definition}

The map $\iota$ is not in general a regular embedding, but the pushforward is still defined.

\begin{lemma}\label{lem:iota-push}
The correspondence $\iota$ defines a pushforward map on the cohomology of $C \times QM\vec{N}$.
\end{lemma}
\begin{proof}
Consider a vector bundle
    \begin{equation*}
    \mathcal{N} = \mathcal{V}\otimes\mathcal{L}^{-1}\otimes \mathcal{O}_{QM\vec{N}}(1)
\end{equation*}
on $C \times QM\vec{N}$. It carries a canonical section $s(x)$ whose
vanishing locus is the image of $\iota$. Thus $\iota$ is a quasi-smooth map of derived stacks, so the Gysin pullback on Borel--Moore homology is defined \cite{Khan19}, and the pushforward on compactly supported cohomology is the adjoint of the Gysin pullback.
\end{proof}
We have the $A^{\vee}$-equivariant pullback and pushforward maps (note the unusual cohomological shift due to~\eqref{Hcent}):
\begin{align*}
    \iota^{*}\colon\;&
    \hat{H}^{i}_{c}(C \times QM\vec{N}_{d+1}, \chi^{-n-1} \boxtimes \chi)
    \to \hat{H}^{i + n + 1}_{c}(C \times QM\vec{N}_{d}, \overline{\mathbb{Q}}_{l} \boxtimes \chi),\\
    &\hat{H}^{i}_{c}(C \times QM\vec{N}_{d+1}, \overline{\mathbb{Q}}_{l} \boxtimes \chi)
    \to \hat{H}^{i + n + 1}_{c}(C \times QM\vec{N}_{d}, \chi^{n+1}\boxtimes \chi),\\
    \iota_{*}\colon\;&
    \hat{H}_{c}^{i}(C \times QM\vec{N}_{d}, \chi^{n+1} \boxtimes \chi)
    \to \hat{H}^{i+n+1}_{c}(C \times QM\vec{N}_{d+1}, \overline{\mathbb{Q}}_{l} \boxtimes \chi),\\
    &\hat{H}_{c}^{i}(C \times QM\vec{N}_{d}, \overline{\mathbb{Q}}_{l} \boxtimes \chi)
    \to \hat{H}^{i+n+1}_{c}(C \times QM\vec{N}_{d+1}, \chi^{-n-1} \boxtimes \chi).
\end{align*}

\begin{lemma}\label{lem:istar-iupper}
There is an identity of operators on $\hat{H}_{c}^{*}(C \times QM\vec{N}_{d}, \overline{\mathbb{Q}}_{l} \boxtimes \chi)$:
\begin{equation*}
    \iota_{*} \iota^{*} = c_{n+1}(\mathcal{N}) \cup,
\end{equation*}
where $c_{n+1}(\mathcal{N}) \in H^{2n+2}(C \times QM\vec{N})$ is the
Euler class of~$\mathcal{N}$.
\end{lemma}

\begin{proof}
Follows from the fact that the image of $\iota$ is the zero locus of the section $s(x)$ (see \cite{Khan19, ARANHA2025110434}).
\end{proof}

\begin{proposition}\label{pr:virt-normal}
The virtual normal bundle $N_{\iota}$
of $\iota$ equals $\iota^{*}\mathcal{N}$.
\end{proposition}

\begin{proof}
By definition, the virtual normal bundle $N_{\iota}$ is the $K$-theory class of the dual of the $1$-shifted tangent complex $\mathbb{T}_{\iota}[1]$ of $\iota$. The exact triangle (see for example \cite{Khan26lec})
\begin{align*}
\mathbb{T}_{\iota} \rightarrow \mathbb{T}_{QM\vec{N}} \rightarrow  \iota^{*}\mathbb{T}_{QM\vec{N}} \rightarrow \mathbb{T}_{\iota}[1]
\end{align*}
gives
\begin{align*}
    N_{\iota} = \iota^{*}T^{\mathrm{vir}}QM\vec{N} - T^{\mathrm{vir}}QM\vec{N},
\end{align*}
hence by~\eqref{TvirQMN},
\begin{equation}\label{eq:iotanorm}
    N_{\iota}
    = \bigl(\iota^{*}-\mathrm{id}\bigr)\bigl[H^{*}(\mathcal{V} \otimes \mathcal{L}^{-1}) \otimes \mathcal{O}_{QM\vec{N}}(1)\bigr].
\end{equation}
For $x \in C$, the short exact sequence
$0 \to \mathcal{O}_{C} \to \mathcal{O}_{C}(x) \to \mathcal{O}_{\{x\}} \otimes TC \to 0$,
tensored by $\mathcal{V}\otimes \mathcal{L}^{-1} \otimes \mathcal{O}_{QM\vec{N}}(1)$,
gives on cohomology
\begin{equation*}
\begin{aligned}
    &H^{*}(C, \mathcal{V} \otimes \mathcal{L}^{-1}(x)) \otimes \mathcal{O}_{QM\vec{N}}(1)
    - H^{*}( \mathcal{V} \otimes \mathcal{L}^{-1}) \otimes \mathcal{O}_{QM\vec{N}}(1)\\
    &\qquad \cong \bigl(\mathcal{V} \otimes \mathcal{L}^{-1} \otimes TC\bigr)\big|_{x} \otimes \mathcal{O}_{QM\vec{N}}(1).
\end{aligned}
\end{equation*}
Since $\iota$ preserves $\mathcal{V}$, $TC$, and
$\mathcal{O}_{QM\vec{N}}(1)$, and
\begin{equation}\label{ionL}
    \iota^{*}\mathcal{L} = \mathcal{L} \otimes \Omega^{1}_{C},
\end{equation}
we conclude $N_{\iota} = \iota^{*}\mathcal{N}$.
\end{proof}

\begin{corollary}\label{cor:iupper-istar}
As operators on $H_{c}^{*}(C \times QM\vec{N}_{d}, \overline{\mathbb{Q}}_{l} \boxtimes \chi)$,
\begin{equation*}
    \iota^{*} \iota_{*} = c_{n+1}(\iota^{*} \mathcal{N}) \cup.
\end{equation*}
\end{corollary}
For the following computations we need a formula for the Euler class of $\mathcal{N}$ and its pullback under $\iota$ in terms of tautological classes. Applying the splitting principle to the filtration~\eqref{VviaM} of
$\mathcal{V}$, using
$\mathcal{N} = \mathcal{V} \otimes \mathcal{L}^{-1} \otimes \mathcal{O}_{QM\vec{N}}(1)$
and~\eqref{ionL}, we get
\begin{align}\label{eq:chern-N}
    c_{n+1}(\mathcal{N})
    &= \prod_{l=1}^{n+1}\bigl(\mu_{l}\otimes 1 - \lambda + 1 \otimes \eta\bigr),\\
    c_{n+1}(\iota^{*}\mathcal{N})
    &= \prod_{l=1}^{n+1}\bigl(\tau \otimes 1 + \mu_{l}\otimes 1 - \lambda + 1 \otimes \eta\bigr). \nonumber
\end{align}

\begin{lemma}\label{lem:pullbacks}
When $\chi_{0}^{n+1} = 1$,
\begin{equation}
    \iota^{*}(1 \otimes \eta) = \mathbf{p} \otimes 1 + 1 \otimes \eta,
    \qquad
    \iota^{*}(1 \otimes \alpha) = \alpha \otimes 1 + 1 \otimes \alpha, \;\;\; \alpha \in H^{1}_{c}(C)
\end{equation}
\end{lemma}

\begin{definition}\label{defn:ef}
The correspondence operators are
\begin{equation}\label{efdef}
\begin{aligned}
    e &= \int_{C} \iota_{*}\colon
    \hat{H}^{i}_{c}(C \times QM\vec{N}_{d}, \chi^{n+1} \boxtimes \chi)
    \to \hat{H}^{i+n-1}_{c}(QM\vec{N}_{d+1}, \chi),\\
    f &= -\int_{C} \iota^{*}\colon
    \hat{H}^{i}_{c}(C \times QM\vec{N}_{d+1}, \chi^{-n-1} \boxtimes \chi)
    \to \hat{H}^{i+n-1}_{c}(QM\vec{N}_{d}, \chi),
\end{aligned}
\end{equation}
where $\int_{C}$ denotes the pushforward along the first factor.
\end{definition}

\subsection{Commutator and Casimir relations}\label{sub:commutator}

In the following we compute the commutator of $e$ and
$f$ explicitly in terms of tautological classes.

Let $\Delta_{C}\colon C \to C\times C$ be the diagonal map.

\begin{proposition}\label{pr:comm-ef}
As operators on
$\hat{H}^{*}_{c}\bigl(C \times C \times QM\vec{N}_{d},
\chi^{n+1} \boxtimes \chi^{-n-1} \boxtimes \chi\bigr)$,
\begin{equation}\label{commans}
    ef_{2} - f e_{1} = \int_{C} \psi \cup \Delta_{C}^{*},
\end{equation}
for a class $\psi \in H^{2n}(C \times QM\vec{N}_{d})$; the
subscripts on $e_{1}, f_{2}$ indicate the $C$-factor being acted on.
\end{proposition}

\begin{proof}
This is a common identity for raising and lowering-type operators defined
by pull--push along a correspondence; see for example \cite{MarianNegut23} for a similar situation.
\end{proof}

\begin{proposition}\label{pr:psi}
The class $\psi$ is given by
\begin{equation*}
\begin{aligned}
    \psi
    &= (n+1)\otimes \eta^{n}
    + n\,\mathbf{p} \otimes 1 \,\Bigl( \bigl(-(n+1)\deg \mathcal{L} + \sum_{l=1}^{n+1}m_{l} + (n+1)(1 - g)\bigr)\,1 \otimes \eta^{n-1}\\
    &\qquad {}-(n-1)(n+1)\,1 \otimes \eta^{n-2} \Theta \Bigr)
    - n(n+1) \sum_{i=1}^{2g} \gamma_{i} \otimes \eta^{n-1} \gamma_{i}^{\vee}.
\end{aligned}
\end{equation*}
\end{proposition}

\begin{proof}
Composing~\eqref{commans} with $\Delta_{C*}$ and using
Lemma~\ref{lem:istar-iupper}, Corollary~\ref{cor:iupper-istar}, and the
formulas~\eqref{eq:chern-N},
\begin{equation}\label{commutef}
\begin{aligned}
    (ef_{2} - f e_{1})\Delta_{C*}
    &= \int_{C}\bigl(\iota^{*}\iota_{*} - \iota_{*}\iota^{*}\bigr)
    = \int_{C}\bigl(c_{n+1}(\iota^{*}\mathcal{N}) - c_{n+1}(\mathcal{N})\bigr)\cup\\
    &= \int_{C} \left( \prod_{l=1}^{n+1}\bigl( \tau \otimes 1 + \mu_{l}\otimes 1 - \lambda + 1 \otimes \eta \bigr)
    - \prod_{l=1}^{n+1}\bigl( \mu_{l}\otimes 1 - \lambda + 1 \otimes \eta \bigr) \right)\cup.
\end{aligned}
\end{equation}
On the other hand, \eqref{commans} gives
$(ef_{2} - f e_{1})\Delta_{C*} = \int_{C} \psi \cup (\tau \otimes 1) \cup$.
Dividing by $\tau\otimes 1$ (see the proof of Proposition $5.1$ in \cite{KazOko23}), we get
\begin{equation*}
    \psi
    = \frac{1}{\tau \otimes 1}
    \left( \prod_{l=1}^{n+1}\bigl( \tau \otimes 1 + \mu_{l}\otimes 1 - \lambda + 1 \otimes \eta \bigr)
    - \prod_{l=1}^{n+1}\bigl( \mu_{l}\otimes 1 - \lambda + 1 \otimes \eta \bigr) \right).
\end{equation*}
Expanding using Lemma~\ref{lem:lambda} and the identity
$\bigl(\sum_{i=1}^{2g}\gamma_{i}\otimes\gamma_{i}^{\vee}\bigr)^{2} = -2\,\mathbf{p}\otimes\Theta$
yields the stated formula.
\end{proof}

\begin{proposition}\label{pr:anticomm}
\begin{equation*}
\begin{aligned}
    \tfrac{1}{2}(ef_{2} + f e_{1})\Delta_{C*}
    &= -\int_{C} 1 \otimes \eta^{n+1}\\
    &\quad + \mathbf{p} \otimes 1 \,\Bigl( \bigl(-(n+1)\deg\mathcal{L} + \sum_{i=1}^{n+1}m_i + (n+1)(1 - g)\bigr)\,1 \otimes \eta^{n}\\
    &\qquad - n(n + 1)\,1\otimes \eta^{n-1}\Theta \Bigr)
    - (n+1) \sum_{i=1}^{2g} \gamma_{i} \otimes \eta^{n}\gamma_{i}^{\vee}.
\end{aligned}
\end{equation*}
\end{proposition}

\begin{proof}
Analogous to \eqref{commutef}, replacing
$\iota^{*}\iota_{*} - \iota_{*}\iota^{*}$ by
$\iota^{*}\iota_{*} + \iota_{*}\iota^{*}$.
\end{proof}

\subsection{The algebra of correspondences}\label{sub:comm-casimirs}

For any cohomology class on~$C$, the operators $e, f$ specialise to
honest endomorphisms of $H^{*}_{c}(QM\vec{N}, \chi)$. We record the relations in the algebra generated by these operators.

\begin{definition}\label{defn:ealpha}
For $\alpha \in H^{*}_{c}(C, \chi^{n+1})$ and
$\beta \in H^{*}_{c}(C, \chi^{-n-1})$, their \emph{descendants} are operators on
$\hat{H}^{*}_{c}(QM\vec{N}, \chi)$ defined by
\begin{equation*}
    e\langle \alpha \rangle = e(\alpha \otimes \cdot),
    \qquad
    f\langle \beta \rangle = f(\beta \otimes \cdot).
\end{equation*}
\end{definition}
\begin{definition}\label{corralg}
    The algebra of correspondences
\begin{align*}
    \mathcal{A}_{\chi} \subset \mathrm{End}(\hat{H}^{*}_{c}(QM\vec{N}, \chi))
\end{align*}
is the unital algebra generated by descendants
$e\langle \alpha \rangle$ for $\alpha \in H^{*}_{c}(C, \chi^{n+1})$,
$f \langle \beta \rangle$ for $\beta \in H^{*}_{c}(C, \chi^{-n-1})$, and by operators of multiplication by $\eta$, $\gamma_{i}$ for $i = 1, \ldots, 2g$, and $\chi_C(\mathcal{V} \otimes \mathcal{L}^{-1})$; when $\chi_{0}^{n+1} = 1$
we write $\mathcal{A}$ in place of $\mathcal{A}_{\chi}$. We denote by $\mathcal{A}^{2}_{\chi}$ the subalgebra generated by $\eta$ and descendants of degree two, and by $\mathcal{A}^{\geq 1}_{\chi}$ the subalgebra generated by $\gamma_{i}$ and descendants of degrees one and two.
\end{definition}
We use the notation $[\;,\;]$ for the supercommutator
\begin{equation*}
    [e\langle \alpha \rangle, f \langle \beta \rangle]
    = e\langle \alpha \rangle f \langle \beta \rangle
    - (-1)^{\deg \alpha \cdot \deg \beta}
    f \langle \beta \rangle e\langle \alpha \rangle.
\end{equation*}
By Definition~\ref{defn:ealpha} and~\eqref{commans}, it is given by
\begin{equation}\label{compsup}
[e\langle\alpha\rangle, f\langle\beta\rangle] = \int_{C}\psi\cup\alpha\cup\beta.
\end{equation}
We collect the supercommutator relations in Appendix~\ref{relations}.

\begin{remark}
    We will see in Remark~\ref{rem:gammavanish} that $\gamma_{i}$ in fact vanishes on $H^{*}_{c}(QM\vec{N}, \chi)$ for nontrivial $\chi_{0}^{n+1}$.
\end{remark}
\section{The symplectic dual variety}\label{sec:geom-interp}

The aim of this section is to construct the action of a certain Clifford algebra over an affine singular surface $X_{0}^{\vee}$ on the cohomology of quasimaps. This is an example of the action of a \emph{Coulomb branch}.

\subsection{Action of the $A_{n}$-surface coordinate ring}
Let $R$ be the commutative ring
\begin{align}
    R = \overline{\mathbb{Q}}_{l}[x, y, \eta]/(xy - \eta^{n+1}).
\end{align}
Here, with a slight abuse of notation, $\eta$ denotes a formal variable. The spectrum 
\begin{equation*}
    X_{0}^{\vee} = \mathrm{Spec} \, R
\end{equation*}
is a singular surface of $A_{n}$-type, which is the symplectic dual \cite{KamnitzerJoel22, webster20233dimensionalmirrorsymmetry} of $X = T^{*}\mathbb{P}^{n}$. In this situation $X$ is referred to as the \emph{Higgs branch}, and $X_{0}^{\vee}$ is the \emph{Coulomb branch}. The following proposition is an example of a Coulomb branch action.
\begin{proposition}\label{pr:A2-triv}
$R$ acts on the cohomology of quasimaps by
\begin{align}\label{RtoAchi}
   \Phi_{\chi}^{2} \colon R &\rightarrow \mathcal{A}^2_{\chi} \qquad x \mapsto e\langle \mathbf{p} \rangle, \qquad y \mapsto -f\langle \mathbf{p} \rangle,
\end{align}
with the $\eta$'s being identified.
\end{proposition}

\begin{proof}
Commutativity follows from \eqref{commutef0} and the
vanishing of $[e\langle\mathbf{p}\rangle,\eta]$, $[f\langle\mathbf{p}\rangle,\eta]$,
which holds because $\iota^{*}\mathcal{O}_{QM\vec{N}}(1) = \mathcal{O}_{QM\vec{N}}(1)$. The identity $xy = \eta^{n+1}$ is the Casimir relation \eqref{cas0}.
\end{proof}
\begin{remark}
    We will show later that the map~\eqref{RtoAchi} is an isomorphism for $\chi_{0}^{n+1}  = 1$. In this case we omit the subscript $\chi$.
\end{remark}
The $\mathrm{Fr} \times A^{\vee}$-action~\eqref{A-action} on the cohomology of quasimaps induces an action on $\mathcal{A}_{\chi}$. In order for the map~\eqref{RtoAchi} to be equivariant, $\mathrm{Fr} \times A^{\vee}$ has to act on $R$ by:
\begin{align*}
     x \mapsto (q^{1/2})^{n+1}a^{-n-1}x, \qquad
    y \mapsto (q^{1/2})^{n+1}a^{n+1}y, \qquad \eta &\mapsto \eta. \nonumber
\end{align*}
The variety $X_{0}^{\vee}$ has a Poisson structure given by an $\mathrm{Fr}\times A^{\vee}$-equivariant map
\begin{align}\label{Poissonbraket}
    \{\cdot , \cdot \} : q^{-1} \otimes \bigwedge\nolimits^{2} \Omega_{R} \rightarrow R \qquad \{x, y \} = (n+1)\eta^{n}, \qquad \{x , \eta\} = x, \qquad \{y, \eta\} = -y,
\end{align}
where $\Omega_{R}$ is the module of K\"ahler differentials of $R$ over $\overline{\mathbb{Q}}_{l}$, and $q^{-1}$ denotes the one-dimensional $\mathrm{Fr}$-representation with the action by the corresponding character.

In the nontrivial
local system case the action of $R$ simplifies significantly.

\begin{proposition}\label{pr:A2chi}
Assume $\chi_{0}^{n+1} \neq 1$. Then the action~\eqref{RtoAchi} factors through  
\begin{align*}
    R \rightarrow R_{\mathrm{fix}} \coloneqq \overline{\mathbb{Q}}_{l}[\eta]/(\eta^{n+1}).
\end{align*}
\end{proposition}
It has an interpretation in terms of the action of $\pi_{1}^{g}$ on $X_{0}^{\vee}$.
\begin{remark}\label{def:mellin-fibre}
Let \(\pi_1^g \) act on $R$ by
\begin{align}\label{A2chi}
    x \mapsto \chi^{n+1}_{0} x,\qquad y \mapsto \chi^{-n-1}_{0} y, \qquad \eta \mapsto \eta. 
\end{align}
Then $R_{\mathrm{fix}} $ is identified with the algebra of coinvariants
\[
R_{\mathrm{fix}} \cong R/(x - \chi_{0}^{n+1}x, \; y - \chi_{0}^{-n-1}y).
\]
Accordingly, its spectrum is the scheme-theoretic fixed locus
\[
\operatorname{Spec} \, R_{\mathrm{fix}} = (X_{0}^{\vee})^{\pi_{1}^{g}}
\]
of the $\pi_{1}^{g}$-action.
\end{remark}

\subsection{Action of a Clifford algebra}
Descendants of cohomology classes of degrees two and one generate a noncommutative algebra $\mathcal{A}^{\geq 1}$. The homomorphism~\eqref{RtoAchi} makes it an $R$-algebra. In this subsection we consider a natural Clifford algebra over $R$ and realise it inside $\mathcal{A}^{\geq 1}$.

Consider the $R$-module
\begin{equation}\label{eq:RmoduleE}
    \mathcal{E} = q^{-1}\otimes \Omega_{R} \otimes H^{1}_{c}(C).
\end{equation}
Let us introduce the notation for the generators of $\mathcal{E}$:
\begin{equation*}
    \theta^{x}_{i} =  d x \otimes \gamma_{i}, \qquad \theta^{y}_{i} = - d y \otimes \gamma_{i}, \qquad i = 1, ..., 2g,
\end{equation*}
and use $\gamma_i$ for the generator $\gamma_{i} = d \eta \otimes \gamma_{i}$.

Consider a $\mathrm{Fr}\times A^{\vee}$-equivariant symmetric pairing
\begin{align}\label{clifpair}
    \langle \cdot, \cdot \rangle \colon S^{2}\left( \mathcal{E} \right) \rightarrow R,
\end{align}
which is given by the tensor product of the Poisson bracket~\eqref{Poissonbraket} with the symplectic intersection form on $H^{1}_{c}(C)$. Let $Cl$ be a Clifford algebra
\begin{align}
    Cl = \mathrm{Cliff}_{R}(\mathcal{E}, \langle \cdot, \cdot \rangle ).
\end{align}
It acts on the cohomology of quasimaps.
\begin{proposition}\label{pr:clifford-triv}
The map 
\begin{align}\label{CltoA}
    \Phi^{\geq 1} \colon Cl \rightarrow \mathcal{A}^{\geq 1} \qquad   \theta^{x}_{i} \mapsto e\langle \gamma_{i} \rangle, \qquad \theta^{y}_{i} \mapsto - f\langle \gamma_{i} \rangle, \qquad i = 1, \ldots, 2g
\end{align}
with the $\gamma_i$'s identified is a homomorphism of $R$-algebras.
\end{proposition}

\begin{proof}
The anticommutators~\eqref{[egfg]} and~\eqref{[egg]} define a pairing on $\mathcal{A}^{\geq 1}$ compatible with~\eqref{clifpair} under the map. The only remaining relation is the Casimir equation~\eqref{cas1}.
\end{proof}
\begin{remark}
    We will show later that the map~\eqref{CltoA} is an isomorphism.
\end{remark}
In the nontrivial case $\chi_{0}^{n+1} \neq 1$ we define an
$R_{\mathrm{fix}}$-module
\begin{align*}
    \mathcal{E}_{\chi} = H^{1}\left(\pi_{1}^{g}, q^{-1}\Omega_{R}\otimes R_{\mathrm{fix}}\right),
\end{align*}
where the $\pi_{1}^{g}\times A^{\vee}$-action on $\Omega_{R}$
and its restriction to the fixed locus are induced from the
action~\eqref{A2chi} on $X_{0}^{\vee}$. Let us introduce the notation for the generators of $\mathcal{E}_{\chi}$:
\begin{align}
    \theta_{\chi, i}^{x} = dx \otimes \gamma_{i}^{\chi}, \qquad \theta_{\chi, i}^{y} = - dy \otimes \gamma_{i}^{\chi^{-1}}, \qquad i = 1, ..., 2g-2,
\end{align}
and again use $\gamma_{i}$ for the generator $\gamma_{i}\otimes \eta$.
There is an analogous pairing
\begin{align*}
    \langle \cdot, \cdot \rangle \colon \mathrm{Sym}^{2}\left( \mathcal{E}_{\chi} \right) \rightarrow R_{\mathrm{fix}},
\end{align*}
which allows us to define a Clifford algebra
\begin{align}
    Cl_{\chi} = \mathrm{Cliff}(\mathcal{E}_{\chi}, \langle \cdot, \cdot\rangle).
\end{align}
The algebra $Cl_{\chi}$ plays the role of $Cl$ in the case of a nontrivial local system.
\begin{proposition}\label{pr:clifford-chi}
Let $\chi_{0}^{n+1} \neq 1$. Then the map 
\begin{align}\label{CltoAchi}
     \Phi_{\chi}^{\geq 1} \colon Cl_{\chi} \rightarrow \mathcal{A}^{\geq 1} \qquad   \theta^{x}_{\chi, i} \mapsto e\langle \gamma_{i}^{\chi} \rangle, \qquad \theta^{y}_{\chi,i}\mapsto - f\langle \gamma_{i}^{\chi^{-1}} \rangle, \qquad i = 1, \ldots, 2g - 2
\end{align}
with the $\gamma_{i}$'s identified is a homomorphism of $R_{\mathrm{fix}}$-algebras.
\end{proposition}

\begin{proof}
This follows from the relations~\eqref{commutinAchin} in the
same way as Proposition~\ref{pr:clifford-triv}.
\end{proof}
\begin{remark}
    We will show that, depending on the values of the parameters $g, m_{k}$, the map~\eqref{CltoAchi} may or may not be an isomorphism.
\end{remark}
In Appendix~\ref{Degreezerodescendants} we also identify degree-zero descendants with differential operators on $X_{0}^{\vee}$.

\section{The resolution of the singular surface}\label{sec:resolution}

In this section we use the torus action on $QM\vec{N}$
to produce a compatible family of maps
\begin{align}\label{A2ktoA2}
    \mathcal{A}^{2} \rightarrow \mathcal{A}^{(k)2}, \;\;\; k = 1, \ldots, n+1,
\end{align}
to the polynomial algebras in two variables $\mathcal{A}^{(k)2} \cong \overline{\mathbb{Q}}_{l}[x_k , y_k ]$. This allows us to prove that the map~\eqref{RtoAchi} is an isomorphism. The homomorphisms~\eqref{A2ktoA2} then correspond to birational maps
\begin{equation*}
  \mathrm{Spec}\, \mathcal{A}^{(k)2} \rightarrow X_{0}^{\vee}.
\end{equation*}
Gluing them along the transition maps produces a smooth symplectic resolution
\begin{align*}
    X^{\vee} \rightarrow X^{\vee}_{0}.
\end{align*}
Analogously to $\mathcal{A}^2$, the algebras $\mathcal{A}^{(k)2}$ have natural modules $\mathcal{E}_{k}$ equipped with symmetric pairings. The Clifford algebras of $\mathcal{E}_{k}$ glue to a locally free sheaf of Clifford algebras on $X^{\vee}$, whose global sections contain $Cl$ as a subalgebra. 

\subsection{Torus fixed loci and attracting strata}

We begin by introducing a torus $T$ acting on $QM\vec{N}$ through
rescaling of the graded pieces $\mathcal{M}_{k}$, identifying its fixed
loci, and computing the virtual normal bundles to the attracting strata.
A Bia{\l}ynicki-Birula-type stratification then yields a filtration on
$\hat{H}^{*}_{c}(QM\vec{N})$ whose graded pieces are the cohomologies of the
fixed loci.

The torus
$T \cong \mathrm{Aut}(\vec{\mathcal{M}}) \cong \mathrm{Aut}(\mathcal{M}_{1}) \times \cdots \times \mathrm{Aut}(\mathcal{M}_{n+1})$
acts on the space of filtrations $\{\mathcal{V}^{\bullet}\}$ by
\begin{align*}
    t(\mathcal{V}^{(k-1)} \hookrightarrow \mathcal{V}^{(k)} \twoheadrightarrow \mathcal{M}_{k}) = (\mathcal{V}^{(k-1)} \hookrightarrow \mathcal{V}^{(k)} \twoheadrightarrow \mathcal{M}_{k} \overset{t_{k}}{\rightarrow} \mathcal{M}_{k}),
\end{align*}
where $t_{k}$ is the coordinate on $\mathrm{Aut}(\mathcal{M}_{k})$.
The associated natural isomorphism
\begin{align*}
    \rho_{t}\colon \mathcal{V} \rightarrow t\mathcal{V}
\end{align*}
lifts the $T$-action to $QM\vec{N}$ via
\begin{align}\label{Taction}
    &T \ni t \colon QM\vec{N} \rightarrow QM\vec{N},\\
    &(\mathcal{V}^{\bullet}, \mathcal{L}, s) \mapsto (t\mathcal{V}^{\bullet}, \mathcal{L}, \rho_{t} \circ s). \nonumber
\end{align}
The set of fixed points decomposes into $n + 1$ components
\begin{align*}
    QM\vec{N}^{T} = \bigsqcup\limits_{k = 1}^{n+1}QM\vec{N}(k),
\end{align*}
where
\begin{align*}
    QM\vec{N}(k) = \{\big( &\mathcal{V}^{\bullet}_{\mathrm{fix}} = (\mathcal{V}^{(l)} = \mathcal{M}_{1} \oplus \cdots \oplus \mathcal{M}_{l})_{l = 1}^{n+1}, \mathcal{L} = \mathcal{M}_{k}(- D), \\
    &s\colon \mathcal{M}_{k}(- D) \rightarrow \mathcal{M}_{k} \subset \mathcal{V} \big) : D \in \mathrm{Div}_{+}(C) \cong S^{\bullet}C\}. \nonumber
\end{align*}
Choose the one-dimensional subtorus
\begin{align}\label{1torus}
    \mathbb{G}_{m} &\hookrightarrow T\\
    &z \mapsto \left(t_{k} = z^{n+1 - k} \right)_{k = 1}^{n+1}. \nonumber
\end{align}
The set of fixed points for $\mathbb{G}_{m}$ is the same as for $T$,
\begin{align*}
   QM\vec{N}^{\mathbb{G}_{m}} \cong QM\vec{N}^{T},
\end{align*}
so we obtain a Bia{\l}ynicki-Birula-type decomposition
\begin{align}\label{attrdecomp}
    QM\vec{N} = \bigsqcup\limits_{k = 1}^{n+1} \mathrm{Attr}(k)
\end{align}
with attractor cells
\begin{align*}
    \mathrm{Attr}(k) = \underset{z \rightarrow 0}{\mathrm{Attr}}(QM\vec{N}(k)) = \{(\mathcal{V}^{\bullet}, \mathcal{L}, s): s(\mathcal{L}) \subset \mathcal{V}^{(k)}\} \setminus
    \{(\mathcal{V}^{\bullet}, \mathcal{L}, s): s(\mathcal{L}) \subset \mathcal{V}^{(k-1)}\}.
\end{align*}

\begin{proposition}\label{pr:attr-smooth}
Let
\begin{align*}
    N(k) = N_{\mathrm{Attr}(k)/QM\vec{N}(k)}
\end{align*}
denote the virtual normal bundle to $QM\vec{N}(k)$ in $\mathrm{Attr}(k)$.
At a fixed point
\begin{align}\label{fixedpt}
(\mathcal{V}^{\bullet}_{\mathrm{fix}}, \mathcal{M}_{k}(-D), s) \in QM\vec{N}(k),
\end{align}
the virtual bundle $N(k)$ has rank
\begin{align}\label{rkNk}
    \mathrm{rk}\,N(k) = - \sum\limits_{l < r} (m_{l} - m_{r} + 1 - g) + \sum\limits_{l = 1}^{k-1} (m_{l} - m_{k} + \deg\, D + 1 - g).
\end{align}
\end{proposition}

\begin{proof}
From~\eqref{TvirQMN} the tangent space to $QM\vec{N}$ at the fixed
point~\eqref{fixedpt} is given, as an element of
$K(\mathrm{Rep}(T))$, by
\begin{align}\label{TQMNf}
    T^{\mathrm{vir}}_{(\mathcal{V}^{\bullet}, \mathcal{M}_{k}(-D), s)} QM\vec{N} = &-\sum\limits_{l < r}t_{l} t_{r}^{-1}H^{*}( \mathcal{M}_{l} \otimes \mathcal{M}_{r}^{-1}) - H^{*}( \mathcal{O}_{C}) + \\
    &+\sum\limits_{l = 1}^{n+1} t_{l} t_{k}^{-1} H^{*}( \mathcal{M}_{l} \otimes \mathcal{M}_{k}^{-1} (D)) \otimes \mathcal{O}_{QM\vec{N}}(1). \nonumber
\end{align}
The fiber of $N(k)$ is given by the $z$-positive degree terms
in~\eqref{TQMNf},
\begin{align*}
    N(k) = -\sum\limits_{l < r} H^{*}(\mathcal{M}_{l} \otimes \mathcal{M}_{r}^{-1}) + \sum\limits_{l = 1}^{k-1} H^{*}( \mathcal{M}_{l} \otimes \mathcal{M}_{k}^{-1}(D)) \otimes \mathcal{O}_{QM\vec{N}}(1),
\end{align*}
from which we conclude the formula \eqref{rkNk}.
\end{proof}
An important feature of the decomposition \eqref{attrdecomp} is the smoothness of the attracting orbits $\mathrm{Attr}(k)$ 
\begin{proposition}\label{thm:attr-smooth}
The attracting set $\mathrm{Attr}(k)$ is smooth for \mbox{$k = 1, \ldots, n+1$}.
\end{proposition}
\begin{proof}
To prove smoothness of $\mathrm{Attr}(k)$, it suffices to show that the
$z$-nonnegative degree terms in the obstruction space~\eqref{obstructions}
vanish at every fixed point. At the fixed point~\eqref{fixedpt}, the map
$s$ is given by a tuple
\begin{align*}
    s = (s_{l} = \delta_{lk}  s' )_{l = 1}^{n+1} \in  \bigoplus_{l = 1}^{n+1} H^{0}(\mathcal{M}_{l} \otimes \mathcal{M}_{k}^{-1}(D))
\end{align*}
for some nonzero $s' \in H^{0}(\mathcal{O}_{C}(D))$, and the obstruction
space is dual to the kernel of the map~\eqref{Deltaadj}, which is given by
\begin{align*}
    \Delta_{s}^{\vee}\colon \bigoplus\limits_{l = 1}^{n+1} H^{0}(\mathcal{M}_{k}(-D)&\otimes \mathcal{M}_{l}^{-1} \otimes \mathcal{K}_{C}) \longrightarrow H^{0}\left( \mathcal{K}_{C} \right) \oplus \bigoplus\limits_{l < r} H^{0}\left( \mathcal{M}_{r} \otimes \mathcal{M}_{l}^{-1} \otimes \mathcal{K}_{C} \right)\\
    (\alpha_{l})_{l = 1}^{n+1} & \mapsto \left( - \sum\limits_{l = 1}^{n+1}\alpha_{l} s_{l}, (s_{r}\alpha_{l})_{l<r} \right) = \left( -\alpha_{k} s', (\delta_{rk} s' \alpha_{l})_{l < r} \right). \nonumber
\end{align*}
Since multiplication by $s'$ is injective, the kernel is cut out by the
equations
\begin{align*}
    \alpha_{l} = 0, \; l = 1, \ldots, k.
\end{align*}
Therefore
\begin{align*}
    \mathbb{H}^{1}(E^{\bullet}) = (\mathrm{Ker}\,\Delta_{s^{\vee}})^{\vee} = \bigoplus_{l > k} H^{0}(\mathcal{M}_{k}(-D)\otimes \mathcal{M}_{l}^{-1}\otimes\mathcal{K}_{C})^{\vee} = \bigoplus_{l > k} H^{1}(\mathcal{M}_{l}\otimes \mathcal{M}_{k}^{-1}(D)).
\end{align*}
The weights of these spaces are $z^{k - l}$ for $l = k+1, \ldots, n+1$,
all of which are negative $z$-powers; hence the attractor has no
obstructions and is smooth at the fixed point.
\end{proof}
 It follows that $\mathrm{Attr}(k)$ is a vector bundle stack over the smooth projective variety $QM\vec{N}(k) \cong S^{\deg\,D}C$, and its cohomology is pure. In particular, this allows us to establish a filtration on the cohomology of the quasimaps.

\begin{proposition}\label{pr:birula-filtration}
For $k = 1, \ldots, n+1$ there is a short exact sequence
\begin{align}\label{birulaqmn}
    0 \rightarrow \hat{H}^{*}_{c}(QM\vec{N}(k))[[- 2 \mathrm{rk}\, N(k)]] \rightarrow \hat{H}^{*}_{c}\left(\bigcup\limits_{l = 1}^{k}\mathrm{Attr}(l)\right) \rightarrow \hat{H}^{*}_{c}\left(\bigcup\limits_{l = 1}^{k-1}\mathrm{Attr}(l)\right) \rightarrow 0,
\end{align}
where for all the spaces the shift $\hat{H}_{c}^{*}$ is the same as in~\eqref{Hcent}. Consequently,
$\hat{H}^{*}_{c}(QM\vec{N})$ carries a filtration with graded pieces
$\hat{H}^{*}_{c}(QM\vec{N}(k))[[-2 \mathrm{rk}\, N(k)]]$.
\end{proposition}

\begin{proof}
For each $k = 1, \ldots, n+1$ the pair of closed and open immersions
\begin{align*}
    \bigcup\limits_{l = 1}^{k-1}\mathrm{Attr}(l) \hookrightarrow \bigcup\limits_{l = 1}^{k}\mathrm{Attr}(l), \;\;\; \mathrm{Attr}(k) \hookrightarrow \bigcup\limits_{l = 1}^{k}\mathrm{Attr}(l)
\end{align*}
(with $\bigcup_{l = 1}^{n+1}\mathrm{Attr}(l) = QM\vec{N}$) induces a long
exact sequence
\begin{align*}
    \hat{H}^{*}_{c}(\mathrm{Attr}(k)) \rightarrow \hat{H}^{*}_{c}\left(\bigcup\limits_{l = 1}^{k}\mathrm{Attr}(l)\right) \rightarrow \hat{H}^{*}_{c}\left(\bigcup\limits_{l = 1}^{k-1}\mathrm{Attr}(l)\right) \overset{\delta}{\rightarrow} \hat{H}^{*}_{c}(\mathrm{Attr}(k))[1],
\end{align*}
whose first map is the pushforward and second map is the pullback. Since
$\mathrm{Attr}(k)$ is an affine bundle over $QM\vec{N}(k)$,
\begin{align*}
    \hat{H}^{*}_{c}(QM\vec{N}(k))[[- 2 \, \mathrm{rk}N(k)]] \overset{\cong}{\rightarrow} \hat{H}^{*}_{c}(\mathrm{Attr}(k)),
\end{align*}
The cohomology of $\mathrm{Attr}(k)$ is pure, so by induction
in $k$ the cohomology of $\bigcup_{l=1}^{k}\mathrm{Attr}(l)$ is pure and
$\delta$ vanishes. The resulting short exact
sequence is~\eqref{birulaqmn}.
\end{proof}

\subsection{Correspondences on fixed loci}

The modification correspondence $\iota$ preserves the fixed loci, hence acts on their cohomology. In this subsection, analogously to Section~\ref{sec:mod-operators}, 
we record pull-push identities, define the associated operators
$e^{(k)}, f^{(k)}$, and derive the commutator and Casimir relations they
satisfy.

There are $n+1$ fixed loci $QM\vec{N}(k)$, $k = 1, \ldots, n+1$, of the
torus action~\eqref{1torus} on $QM\vec{N}$. The correspondence $\iota$
preserves each of them, and we consider the closed immersion
\begin{align*}
    \iota^{(k)}\colon C \times QM\vec{N}_{d}(k) \rightarrow C \times QM\vec{N}_{d+1}(k),
\end{align*}
the restriction of $\iota$ to the corresponding fixed locus. Each component $QM\vec{N}_{d}(k)$ is a symmetric product of the curve $C$, and $\iota^{(k)}$ ``adds a point''.

It defines pullbacks and pushforwards (analogously to Lemma~\ref{lem:iota-push}) on cohomology:
\begin{align*}
   & \iota^{(k)*}\colon \hat{H}^{i}_{c}(C \times QM\vec{N}_{d+1}(k), \chi^{-n-1}\boxtimes \chi) \rightarrow \hat{H}^{i + n + 1}_{c}(C \times QM\vec{N}_{d}(k), \overline{\mathbb{Q}}_{l}\boxtimes \chi) \\
    & \iota^{(k)*}\colon \hat{H}^{i}_{c}(C \times QM\vec{N}_{d+1}(k), \overline{\mathbb{Q}}_{l}\boxtimes \chi) \rightarrow \hat{H}^{i + n + 1}_{c}(C \times QM\vec{N}_{d}(k), \chi^{n+1}\boxtimes \chi) \nonumber
\end{align*}
\begin{align*}  
    &\iota^{(k)}_{*}\colon \hat{H}^{i}_{c}(C \times QM\vec{N}_{d}(k), \chi^{n+1}\boxtimes \chi) \rightarrow \hat{H}^{i-n+1}_{c}(C \times QM\vec{N}_{d+1}(k), \overline{\mathbb{Q}}_{l}\boxtimes \chi)\\
     &\iota^{(k)}_{*}\colon \hat{H}^{i}_{c}(C \times QM\vec{N}_{d}(k), \overline{\mathbb{Q}}_{l}\boxtimes \chi) \rightarrow \hat{H}^{i-n+1}_{c}(C \times QM\vec{N}_{d+1}(k), \chi^{-n-1}\boxtimes \chi) \nonumber
\end{align*}
\begin{remark}\label{cohshiftiotak}
    By Proposition~\ref{pr:birula-filtration} the cohomology of $QM\vec{N}(k)$ enters the filtration on the cohomology of $QM\vec{N}$ with the shift $[[-2\mathrm{rk}\,N(k)]]$. Therefore the cohomological and Frobenius degree of $\iota^{(k)*}$ and $\iota_{*}^{(k)}$, compatible with that of $\iota^{*}$ and $\iota_{*}$, should be shifted by $[[2\mathrm{rk}\,N_{d+1}(k) - 2\mathrm{rk}\,N_{d}(k)]] = [[2(k-1)]]$ and $[[2\mathrm{rk}\,N_{d}(k) - 2\mathrm{rk}\,N_{d+1}(k)]] = [[2(1-k)]]$ respectively.
\end{remark}
Consider a vector bundle
\begin{align*}
    \mathcal{N}^{(k)} = \mathcal{V}^{(k)}\otimes\mathcal{L}^{-1}\otimes \mathcal{O}_{QM\vec{N}}(1)
\end{align*}
on $C \times QM\vec{N}(k)$. The quotient bundle
\begin{align*}
    \mathcal{N}^{(k)}/\mathcal{N}^{(k-1)} \cong \mathcal{M}_{k} \otimes\mathcal{L}^{-1}\otimes \mathcal{O}_{QM\vec{N}}(1)
\end{align*}
carries a natural section $\bar{s}^{(k)}(x)$ obtained by restricting
$s(x)$ to $C \times QM\vec{N}(k)$, and the image of $\iota^{(k)}$ is
cut out by the equation $\bar{s}^{(k)}(x) = 0$. Hence
\begin{align}\label{chernNi}
\iota^{(k)}_{*}\iota^{(k)*} = c_{1}(\mathcal{N}^{(k)}/\mathcal{N}^{(k-1)}) \; \cup,
\end{align}
where $c_{1}(\mathcal{N}^{(k)}/\mathcal{N}^{(k-1)}) \in H^{2}(QM\vec{N}(k))$
is the Euler class of the line bundle
$\mathcal{N}^{(k)}/\mathcal{N}^{(k-1)}$. Analogously
to~\eqref{eq:iotanorm}, the virtual normal bundle $N_{\iota^{(k)}}$ to the
image of $\iota^{(k)}$ is
\begin{align*}
    N_{\iota^{(k)}} = (\iota^{(k)*} - \mathrm{id})T^{\mathrm{vir}}(C \times QM\vec{N}(k)) = \iota^{(k)*}\mathcal{N}^{(k)}/\mathcal{N}^{(k-1)} 
\end{align*}
which implies
\begin{align}\label{cherniNi}
\iota^{(k)*}\iota^{(k)}_{*} = c_{1}(\iota^{(k)*}\mathcal{N}^{(k)}/\mathcal{N}^{(k-1)}) \; \cup = c_{1}(\mathcal{N}^{(k)}/\mathcal{N}^{(k-1)} \otimes TC ) \; \cup.
\end{align}
Using~\eqref{VviaM}, the top Chern class of $\mathcal{N}^{(k)}/\mathcal{N}^{(k-1)}$ is
\begin{align*}
    c_{1}(\mathcal{N}^{(k)}/\mathcal{N}^{(k-1)}) = c_{1}(\mathcal{M}_{k} \otimes \mathcal{L}^{-1} \otimes \mathcal{O}_{QM\vec{N}}(1)) = \mu_{k}\otimes 1 - \lambda + 1 \otimes \eta,
\end{align*}
and
\begin{align*}
    c_{1}(\iota^{*}\mathcal{N}^{(k)}/\mathcal{N}^{(k-1)}) = c_{1}(\mathcal{N}^{(k)}/\mathcal{N}^{(k-1)} \otimes TC) =  \tau \otimes 1 + \mu_{k}\otimes 1 - \lambda + 1 \otimes \eta.
\end{align*}

\begin{definition}\label{defn:ekfk}
In analogy with~\eqref{efdef}, define
\begin{align}\label{ekfkdefinition}
    &e^{(k)} = \int\limits_{C} \iota^{(k)}_{*},\;\;\; f^{(k)} = -\int\limits_{C} \iota^{(k)*}.
\end{align}
\end{definition}

The identities~\eqref{chernNi} and~\eqref{cherniNi} yield the commutation
relation
\begin{align*}
    e^{(k)}f_{2}^{(k)} - f^{(k)}e_{1}^{(k)} = \int\limits_{C} \Delta_{C}^{*}.
\end{align*}
This implies the supercommutator identities of descendants, which we collect in Appendix~\ref{relations}.

From~\eqref{chernNi} and~\eqref{cherniNi} we also obtain a Casimir identity
\begin{align*}
    \frac{1}{2}(e^{(k)}f_{2}^{(k)} + f^{(k)}e_{1}^{(k)})\Delta_{C*} = - \int\limits_{C} (1 \otimes \eta + \mathbf{p}\otimes 1 (- \deg\,\mathcal{L} + m_{k} + 1 - g) - \sum\limits_{i = 1}^{2g} \gamma_{i} \otimes \gamma_{i}^{\vee}).
\end{align*}

Analogously to Definition~\ref{corralg} we define an algebra of correspondences acting on the fixed loci.
\begin{definition}\label{corralgk}
    The algebra of correspondences 
\begin{align*}
    \mathcal{A}_{\chi}^{(k)} \subset \mathrm{End}\left(\hat{H}^{*}_{c}(QM\vec{N}(k), \chi)[[-2\, \mathrm{rk} N(k)]]\right)
\end{align*}    
is the unital algebra generated by descendants $e^{(k)}\langle \alpha \rangle$ for $\alpha \in H^{*}_{c}(C, \chi^{n+1})$,
$f^{(k)}\langle \beta \rangle$ for $\beta \in H^{*}_{c}(C, \chi^{-n-1})$,
$\eta$, and $\gamma_{i}$ for $i = 1, \ldots, 2g$; when
$\chi_{0}^{n+1} = 1$ we write $\mathcal{A}^{(k)}$ in place of
$\mathcal{A}^{(k)}_{\chi}$. We denote by $\mathcal{A}^{(k)2}_{\chi}$ the subalgebra generated by $\eta$ and descendants of degree two, and by $\mathcal{A}^{(k)\geq 1}_{\chi}$ the subalgebra generated by $\gamma_{i}$ and descendants of degrees one and two.
\end{definition}
Let us use the notation for the descendants
\begin{align*}
    x_{k} = e^{(k)}\langle \mathbf{p} \rangle, \;\;\; y_{k} = -f^{(k)}\langle \mathbf{p} \rangle, \;\;\; 
\end{align*}
\begin{align*}
    \theta^{x_k}_{i} = e^{(k)}\langle \gamma_{i} \rangle, \;\;\; \theta^{y_k}_{i} = - f^{(k)}\langle \gamma_{i} \rangle, \;\;\; i = 1, \ldots, 2g,
\end{align*}
\begin{align}\label{eq:thetaveek}
    \theta_{i}^{\vee x_k} = e^{(k)}\langle \gamma_{i}^{\vee} \rangle,\;\;\; \theta_{i}^{\vee y_k} = - f^{(k)}\langle \gamma_{i}^{\vee} \rangle, \;\;\; i = 1, \ldots, 2g,
\end{align}
\begin{align*}
    \theta^{x_k}_{\chi, i} = e^{(k)} \langle \gamma_{i}^{\chi} \rangle, \qquad \theta^{y_k}_{\chi, i} = - f^{(k)} \langle \gamma_{i}^{\chi^{-1}} \rangle, \qquad i = 1, ..., 2g-2
\end{align*}

\subsection{Algebras of correspondences on the fixed loci}

We now show that the algebra of descendants acting on
$H^{*}_{c}(QM\vec{N}(k), \chi)$ is a \emph{Weyl-Clifford superalgebra}. Its commutative part is the coordinate
ring of $\pi_{1}^{g}$-fixed points of the affine chart $U_{k} \cong \mathbb{A}^{2}$, and the generators of degree one and two assemble into a Clifford algebra.

We briefly review the Weyl-Clifford algebra. Suppose $V = V^{\mathrm{even}} \oplus V^{\mathrm{odd}}$ is a finite-dimensional super vector space over $\overline{\mathbb{Q}}_{l}$, equipped with a nondegenerate bilinear form $\langle \cdot, \cdot \rangle$, satisfying
\begin{equation*}
    \langle u , v \rangle = (-1)^{\mathrm{deg}\,u \, \mathrm{deg}\,v}\langle v , u\rangle
\end{equation*}
for homogeneous $u, v \in V$.
The actual example of interest for us is 
\begin{equation}\label{V=HH}
    V = H^{*}_{c}(C, \chi^{n+1}) \oplus H^{*}_{c}(C, \chi^{n+1})^{*} \cong H^{*}_{c}(C, \chi^{n+1}) \oplus H^{*}_{c}(C, \chi^{-n-1}),
\end{equation}
where $H_{c}^{*}(C, \chi^{-n-1})$ is identified with $H_{c}^{*}(C, \chi^{n+1})^{*}$ by the Poincar\'e pairing. The $\mathbb{Z}/2\mathbb{Z}$-grading is provided by the cohomological degree, and the natural pairing has the form
\begin{align}
    \langle \alpha_{1} \oplus \beta_{1}, \alpha_{2} \oplus \beta_{2} \rangle = \beta_{1}(\alpha_2 ) + (-1)^{\mathrm{deg}\,\alpha_1 \, \mathrm{deg}\,\beta_2}\beta_{2}(\alpha_1)
\end{align}
for homogeneous $\alpha_{1} \oplus \beta_{1}$ and $\alpha_{2} \oplus \beta_{2} $.
\begin{definition}\label{WeylCliffalg}
    The Weyl-Clifford algebra $WCl_{\overline{\mathbb{Q}}_{l}}(V)$ is the quotient of the free algebra
\begin{equation*}
    WCl_{\overline{\mathbb{Q}}_{l}}(V) = \overline{\mathbb{Q}}_{l}\langle V \rangle / \left( u v -(-1)^{\mathrm{deg}\,u\,\mathrm{deg}\,v} vu - \langle u , v\rangle \cdot 1 \right).
\end{equation*}
\end{definition}
The supercommutator relations of Subsection~\ref{sub:supcominAk} imply that the algebra $WCl_{\overline{\mathbb{Q}}_{l}}(V)$ acts on the space
$\hat{H}_{c}^{*}(QM\vec{N}(k), \chi )[[-2\, \mathrm{rk}N(k)]]$ by
\begin{align}\label{WeylCliffact}
    &WCl_{\overline{\mathbb{Q}}_{l}}\left(V\right) \rightarrow \mathcal{A}^{(k)}_{\chi} \\
    &\alpha \mapsto e \langle \alpha \rangle ,\;\; \alpha \in H_{c}^{*}(C, \chi^{n+1}) \nonumber \\
    &\beta \mapsto f \langle \beta \rangle, \;\; \beta \in H_{c}^{*}(C, \chi^{n+1})^{*}. \nonumber
\end{align}
We use the following result~\cite{MACDONALD1962319, KazOko23}, which is a one-dimensional analog of Nakajima's action on the Hilbert schemes of points of surfaces~\cite{LecturesNakajima}.
\begin{theorem}\label{thm:Ak2-triv}
Under the action~\eqref{WeylCliffact} $H_{c}^{*}(QM\vec{N}(k))[[- 2 \, \mathrm{rk}N(k)]]$ is a Verma module for $WCl_{\overline{\mathbb{Q}}_{l}}\left( V \right)$.
\end{theorem}
Theorem~\ref{thm:Ak2-triv} implies injectivity of~\eqref{WeylCliffact}, and surjectivity follows from the Casimir relations in Subsection~\ref{sub:CasrelinAk}, so~\eqref{WeylCliffact} is an isomorphism. In particular,
\begin{equation*}
    \mathcal{A}^{(k)2} \cong \overline{\mathbb{Q}}_{l}[x_k , y_k ].
\end{equation*}
As an $\mathcal{A}^{(k)2}$-module, $WCl_{\overline{\mathbb{Q}}_{l}}$ is a product of symmetric and exterior algebras:
\begin{align*}
    WCl_{\overline{\mathbb{Q}}_{l}} \cong \mathrm{Sym}^{\bullet}_{\mathcal{A}^{(k)2}} \langle e\langle 1 \rangle , f\langle 1 \rangle \rangle \otimes \bigwedge\nolimits^{\bullet}_{\mathcal{A}^{(k)2}} \langle \theta_{i}^{x_k } , \theta_{i}^{y_k }, \theta_i^{\vee x_k },
    \theta_i^{\vee y_k }, \; i = 1, ..., g \rangle.
\end{align*}

We denote the spectrum of $\mathcal{A}^{(k)2}$ by $U_{k}$:
\begin{align}
    U_{k} = \mathrm{Spec}(\mathcal{A}^{(k)2}) \cong \mathbb{A}^{2}_{\overline{\mathbb{Q}}_{l}}.
\end{align}
The $\mathrm{Fr}\times A^{\vee}$-action on cohomology of the fixed loci induces the action on $\mathcal{A}^{(k)2}$ by:
\begin{align*}
     x_{k} \mapsto q^{-\frac{n+1}{2}+k}a^{-n-1}x_{k}, \qquad
    y_{k} \mapsto q^{\frac{n+1}{2}+1-k}a^{n+1}y_{k}. \nonumber
\end{align*}
The smooth variety $U_{k}$ has a symplectic structure given by an $\mathrm{Fr}\times A^{\vee}$-equivariant map
\begin{align}\label{Poissonbraketk}
    \{\cdot , \cdot \} : q^{-1} \otimes \bigwedge\nolimits^{2} \Omega_{\mathcal{A}^{(k)2}} \rightarrow \mathcal{A}^{(k)2} \qquad \{x_{k}, y_{k} \} = 1.
\end{align}

When $\chi_{0}^{n+1} \neq 1$, the isomorphism~\eqref{WeylCliffact}
implies
\begin{align*}
    \mathcal{A}^{(k)2}_{\chi} = \overline{\mathbb{Q}}_{l}[\eta]/(\eta) \cong \overline{\mathbb{Q}}_{l}.
\end{align*}
\begin{remark}
    Let $\pi_{1}^{g}$ act on $\mathcal{A}^{(k)2}$ by
\begin{align}\label{pi1Aonk}
    x_{k} \mapsto \chi_{0}^{n+1} x_{k}, \qquad y_{k} \mapsto \chi^{-n-1}_{0} y_{k}.
\end{align}
Then $\mathcal{A}_{\chi}^{(k)2}$ is identified with the algebra of coinvariants:
\begin{align}\label{A2chik}
    \mathcal{A}^{(k)2}_{\chi} \cong  \mathcal{A}^{(k)2}/(x_{k} - \chi_{0}^{n+1}x_{k}, \; y_{k} - \chi_{0}^{-n-1}y_{k}), \nonumber
\end{align}
Accordingly, its spectrum is the scheme-theoretic fixed locus
\begin{align*}
    \mathrm{Spec}(\mathcal{A}^{(k)2}_{\chi}) = (U_{k})^{\pi_1^{g}}
\end{align*}
of the $\pi_{1}^{g}$-action.
\end{remark} 
Analogously to~\eqref{eq:RmoduleE} consider an $\mathcal{A}^{(k)2}$-module
\begin{align*}
    \mathcal{E}^{(k)} = q^{-1}\Omega_{\mathcal{A}^{(k)2}}\otimes H^{1}_{c}(C).
\end{align*}
It is equipped with the symmetric pairing $\langle \cdot, \cdot \rangle$, provided by anticommutators~\eqref{comk11}. 
The isomorphism~\eqref{WeylCliffact} implies that $\mathcal{A}_{\chi}^{(k)\geq 1}$ is identified with
\begin{align}
    \mathcal{A}^{(k)\geq 1} \cong \mathrm{Cliff}_{\mathcal{A}^{(k)2}}(\mathcal{E}^{(k)}, \langle \cdot , \cdot \rangle).
\end{align}

In Appendix~\ref{Degreezerodescendants} we describe the entire $\mathcal{A}^{(k)}$  as an algebra of differential operators. 

In the nontrivial case $\chi_{0}^{n+1} \neq 1$, we define an
$\mathcal{A}^{(k)2}_{\chi}$-module
\begin{align*}
    \mathcal{E}_{\chi}^{(k)} = H^{1}\left(\pi_{1}^{g}, q^{-1}\Omega_{\mathcal{A}^{(k)2}} \otimes \mathcal{A}_{\chi}^{(k)2}\right),
\end{align*}
where the $\pi_{1}^{g}\times A^{\vee}$-action on $\Omega^{1}_{\mathcal{A}^{(k)2}}$ and
its restriction to the fixed locus are induced from the
action~\eqref{pi1Aonk} on $U_{k}$. There is an analogous pairing
\begin{align}\label{eq:pairingEk}
    \langle \cdot, \cdot \rangle \colon S^{2}\left( \mathcal{E}_{\chi}^{(k)} \right) \rightarrow \mathcal{A}^{(k)2}_{\chi}.
\end{align}
The isomorphism~\eqref{WeylCliffact} for $\chi_{0}^{n+1} \neq 1$ implies
\begin{align*}
    \mathcal{A}^{(k)\geq 1}_{\chi} = \mathrm{Cliff}(\mathcal{E}^{(k)}_{\chi}, \langle \cdot , \cdot \rangle).
\end{align*}

\subsection{Restricting to attractor strata}

We now relate the correspondences on the full stack $QM\vec{N}$ to those
on the fixed loci. The excess intersection theorem gives explicit formulas for the action of the operators
$e\langle\alpha\rangle, f\langle\beta\rangle$ on $\hat{H}^{*}_{c}(QM\vec{N}(k))$
in terms of their fixed-locus analogues.

Consider the closed immersions
\begin{align*}
    j^{(k)}\colon C \times \bigcup\limits_{l = 1}^{k}\mathrm{Attr}(l) \hookrightarrow C \times QM\vec{N},
\end{align*}
\begin{align*}
    i^{(k)}\colon C \times QM\vec{N}(k) \hookrightarrow C \times \bigcup\limits_{l = 1}^{k}\mathrm{Attr}(l).
\end{align*}
The correspondence $\iota$ preserves $C \times \bigcup_{l = 1}^{k}\mathrm{Attr}(l)$
and $C \times QM\vec{N}(k)$, and its restrictions
\begin{align*}
    \bar{\iota}^{(k)}\colon C \times \bigcup\limits_{l = 1}^{k}\mathrm{Attr}(l) \rightarrow C \times \bigcup\limits_{l = 1}^{k}\mathrm{Attr}(l),
\end{align*}
\begin{align*}
    \iota^{(k)}\colon C \times QM\vec{N}(k) \rightarrow C \times QM\vec{N}(k)
\end{align*}
are the base changes of $\iota$ along $j^{(k)}$ and $j^{(k)} \circ i^{(k)}$,
respectively.

Via the pullback map
\begin{align*}
    j^{(k)*}\colon \hat{H}^{*}_{c}(C \times QM\vec{N}, \psi) \rightarrow \hat{H}^{*}_{c}(C \times \bigcup\limits_{l = 1}^{k}\mathrm{Attr}(l), j^{(k)*}\psi)
\end{align*}
the operators $\iota_{*}, \iota^{*}$ induce operators on
$\hat{H}^{*}_{c}(C \times \bigcup_{l=1}^{k}\mathrm{Attr}(l), j^{(k)*}\psi)$,
which we denote by the same symbols (here $\psi$ is a local system on
$C \times QM\vec{N}$). Since pullbacks commute,
\begin{align}\label{pbk}
    \iota^{*} = \bar{\iota}^{(k)*},
\end{align}
while for the pushforward, by~\cite{Khan19},
\begin{align}\label{iotaiotabark}
    \iota_{*} = \bar{\iota}^{(k)}_{*}c_{n+1-k}(j^{(k)*}N_{\iota} - N_{\bar{\iota}^{(k)}}).
\end{align}
Similarly, via the pushforward
\begin{align*}
    i^{(k)}_{*}\colon \hat{H}^{*}_{c}(C \times QM\vec{N}(k), \phi ) \rightarrow \hat{H}^{*}_{c}(C \times \bigcup\limits_{l = 1}^{k}\mathrm{Attr}(l), i_{!}^{(k)}\phi)
\end{align*}
the operators $\bar{\iota}^{(k)}_{*}, \bar{\iota}^{(k)*}$ induce
operators on $\hat{H}^{*}_{c}(C \times QM\vec{N}(k), \phi)$, which we again
denote by the same symbols. Since pushforwards commute,
\begin{align}\label{iotakiotabark}
    \bar{\iota}^{(k)}_{*} = \iota^{(k)}_{*},
\end{align}
while the excess intersection formula (Fulton,
\emph{Intersection Theory}, Theorem~6.3; for the derived $l$-adic version see
\cite{Khan19}) gives
\begin{align*}
    \bar{\iota}^{(k)*} = c_{k-1}(i^{(k)*}N_{\bar{\iota}^{(k)}} - N_{\iota^{(k)}}) \iota^{(k)*}.
\end{align*}

We next simplify the excess normal bundle:
\begin{align*}
    i^{(k)*}N_{\bar{\iota}^{(k)}} - N_{\iota^{(k)}} = i^{(k)*}&\left(\bar{\iota}^{(k)*}T^{\mathrm{vir}}\bigcup\limits_{l = 1}^{k}\mathrm{Attr}(l) -T^{\mathrm{vir}}\bigcup\limits_{l = 1}^{k}\mathrm{Attr}(l)\right) \\
    &-\left(\iota^{(k)*}T^{\mathrm{vir}}QM\vec{N}(k) - T^{\mathrm{vir}}QM\vec{N}(k)\right). \nonumber
\end{align*}
Using
\begin{align*}
i^{(k)*}\bar{\iota}^{(k)*}T^{\mathrm{vir}}\bigcup\limits_{l = 1}^{k}\mathrm{Attr}(l) = \iota^{(k)*}i^{(k)*}T^{\mathrm{vir}}\bigcup\limits_{l = 1}^{k}\mathrm{Attr}(l),
\end{align*}
this rearranges to
\begin{align*}
 i^{(k)*}N_{\bar{\iota}^{(k)}} - N_{\iota^{(k)}} &=\iota^{(k)*} \big(i^{(k)*}T^{\mathrm{vir}}\bigcup\limits_{l = 1}^{k}\mathrm{Attr}(l) - T^{\mathrm{vir}}QM\vec{N}(k)\big) -\\
 &-\big( i^{(k)*}T^{\mathrm{vir}}\bigcup\limits_{l = 1}^{k}\mathrm{Attr}(l) - T^{\mathrm{vir}}QM\vec{N}(k) \big)= \nonumber\\
    &=\iota^{(k)*} N_{\mathrm{Attr}(k)/QM\vec{N}(k)} - N_{\mathrm{Attr}(k)/QM\vec{N}(k)} = \iota^{(k)*} N(k) - N(k). \nonumber
\end{align*}
By the same computation as in~\eqref{eq:iotanorm},
\begin{align*}
    &\iota^{(k)*} N(k) - N(k)=(\iota^{(k)*} - \mathrm{id} )\left[H^{*}(C, \mathcal{V}^{(k-1)}\otimes \mathcal{L}^{-1})\otimes\mathcal{O}_{QM\vec{N}}(1)\right]  \\
    &=\iota^{(k)*}\mathcal{V}^{(k-1)}\otimes \mathcal{L}^{-1}\otimes\mathcal{O}_{QM\vec{N}}(1) = \iota^{(k)*} \mathcal{N}^{(k-1)} = \mathcal{N}^{(k-1)} \otimes TC, \nonumber
\end{align*}
which finally yields
\begin{align}\label{iotapr}
    \bar{\iota}^{(k)*} = c_{k-1}(\mathcal{N}^{(k-1)} \otimes TC) \iota^{(k)*}.
\end{align}

Combining~\eqref{iotaiotabark} and~\eqref{iotakiotabark},
\begin{align}\label{pushiotainchart}
    \iota_{*} = \iota^{(k)}_{*}c_{n+1-k}((j^{(k)} \circ i^{(k)})^{*}N_{\iota} - N_{\iota^{(k)}}) = \iota^{(k)}_{*}c_{n + 1 -k}\left( (\mathcal{N}/\mathcal{N}^{(k)}) \otimes TC \right),
\end{align}
and combining~\eqref{pbk} and~\eqref{iotapr},
\begin{align}\label{pulliotainchart}
    \iota^{*} = c_{k-1}(\mathcal{N}^{(k-1)} \otimes TC) \iota^{(k)*}.
\end{align}
From~\eqref{pushiotainchart} and~\eqref{pulliotainchart} we obtain the
following formulas for the action of descendants on
$\hat{H}^{*}_{c}(QM\vec{N}(k))$.
\begin{align}
    &f \langle \mathbf{p} \rangle = -\int\limits_{C} \iota^{*} \mathbf{p} \otimes \cdot = 
    -\int\limits_{C} c_{k-1}\left( \mathcal{N}^{(k-1)}\otimes TC \right) \iota^{(k)*}\mathbf{p}\otimes\cdot = f^{(k)} \langle \mathbf{p} \rangle \eta^{k-1}, \nonumber \\
   &f \langle \gamma_i \rangle =  -\int\limits_{C} c_{k-1}\left( \mathcal{N}^{(k-1)}\otimes TC \right) \iota^{(k)*}\gamma_i \otimes\cdot =  \eta^{k-1} f^{(k)}\langle \gamma_i \rangle  + (k-1)f^{(k)} \langle \mathbf{p} \rangle \eta^{k-2} \gamma_i , \label{fgammatrans}\\
    &e \langle \mathbf{p} \rangle = \int\limits_{C} \iota_{*} \mathbf{p} \otimes \cdot =
    \int\limits_{C} c_{n+1 -k}\left( \mathcal{N}/\mathcal{N}^{(k)} \right) \iota^{(k)}_{*}\mathbf{p}\otimes\cdot =e^{(k)} \langle \mathbf{p} \rangle \eta^{n+1-k}, \nonumber \\
   &e \langle \gamma_i \rangle =  \int\limits_{C} c_{n+1-k}\left( \mathcal{N}/\mathcal{N}^{(k)} \right) \iota^{(k)}_{*}\gamma_i \otimes\cdot =  \eta^{n+1-k} e^{(k)}\langle \gamma_i \rangle  + (n+1-k)e^{(k)} \langle \mathbf{p} \rangle \eta^{n-k} \gamma_i . \label{egammatrans}
\end{align}
We also derive analogous formulas for the degree-zero generators in Appendix~\ref{Degreezerodescendants}.

For the nontrivial character case $\chi_{0}^{n+1} \neq 1$, the action of descendants on
$\hat{H}^{*}_{c}(QM\vec{N}(k), \chi)$ is given by
\begin{align*}
     f \langle \gamma_{i}^{\chi^{-1}} \rangle &=  \eta^{k-1}f^{(k)}\langle \gamma_{i}^{\chi^{-1}} \rangle = \delta_{k, 1} f^{(k)}\langle \gamma_{i}^{\chi^{-1}} \rangle,\\
    e \langle \gamma_{i}^{\chi} \rangle &=  \eta^{n+1-k}e^{(k)}\langle \gamma_{i}^{\chi} \rangle = \delta_{k, n+1} e^{(k)}\langle \gamma_{i}^{\chi} \rangle. \nonumber
\end{align*}

We thus have compatible homomorphisms
\begin{equation}\label{eq:AtoAk}
    \varphi^{(k)}_{\chi}: \mathcal{A}_{\chi} \rightarrow \mathcal{A}_{\chi}^{(k)}, \qquad  \varphi^{(k)2}_{\chi}: \mathcal{A}_{\chi}^{2} \rightarrow \mathcal{A}_{\chi}^{(k)2}, \qquad  \varphi^{(k)\geq 1}_{\chi}: \mathcal{A}_{\chi}^{\geq 1} \rightarrow \mathcal{A}_{\chi}^{(k)\geq 1}, \qquad k = 1, ..., n+1.
\end{equation}
We omit the subscript $\chi$ when $\chi_{0}^{n+1} = 1$.
\subsection{Symplectic resolution}\label{sub:sympres}

In this subsection we use~\eqref{eq:AtoAk} to prove the isomorphism $R \cong \mathcal{A}^{2}$ and to define a smooth symplectic resolution $X^{\vee}\to
X_{0}^{\vee}$. Then we prove the isomorphism $Cl \cong \mathcal{A}^{\geq 1}$. The morphisms of algebras $\mathcal{A}^{\geq 1} \rightarrow \mathcal{A}^{(k)\geq 1}$ will then correspond to a morphism of sheaves of Clifford algebras $\mathcal{A}^{\geq 1} \rightarrow \mathrm{Cliff}$ on $X^{\vee}$.

\begin{proposition}\label{iso2}
    The map~\eqref{RtoAchi}
\begin{align}
    \Phi^{2}: R \rightarrow \mathcal{A}^{(2)}
\end{align}
is an isomorphism.
\end{proposition}
\begin{proof}
   We need to prove that $\Phi^2$ is injective. Consider the composition
\begin{align}
    \varphi^{(1)2} \circ \Phi^2 \colon R &\rightarrow \overline{\mathbb{Q}}_{l}[x_1 , y_1 ], \qquad x \mapsto x_1^{n+1}y_1^{n}, \qquad y \mapsto y_1 , \qquad \eta \mapsto x_{1}y_{1}.
\end{align}
After inverting $\eta$ it becomes an isomorphism, and since $R$ is an integral domain, the composition, and hence $\Phi^{2}$, is injective.
\end{proof}
From now on we make no distinction between $R$ and $\mathcal{A}^{2}$. The homomorphisms $\varphi^{(k)2}$ correspond to birational maps
\begin{align}\label{projectk}
    \pi_{k}\colon\; U_{k} \rightarrow X^{\vee}_{0} \qquad x = x_{k}^{n+2 - k}y_{k}^{n+1-k}, \qquad  y = x_{k}^{k-1}y_{k}^{k},
\end{align}
which preserve the Poisson structure and the
$\pi_{1}^{g}\times \mathrm{Fr} \times A^{\vee}$-action. In particular, for $\chi_{0}^{n+1} \neq 1$ the map $\pi_{k}$ restricts to a morphism of $\pi_{1}^{g}$-fixed loci
\begin{align*}
    \mathrm{Spec}\, \mathcal{A}^{(k)2}_{\chi}  = (U_{k})^{\pi_1^{g}} &\rightarrow (X^{\vee}_{0})^{\pi_{1}^{g}} = \mathrm{Spec}\, R_{\mathrm{fix}}.
\end{align*}
It is easy to see that the corresponding map of algebras is exactly
\begin{equation*}
    \varphi^{(k)2}_{\chi} \circ \Phi^{2}_{\chi} \colon R_{\mathrm{fix}} \rightarrow \mathcal{A}^{(k)2}_{\chi}  \qquad \eta \mapsto 0.
\end{equation*}
Proposition~\ref{iso2} extends to the generators of degree $\geq 1$.
\begin{proposition}\label{ClcongA}
    The map~\eqref{CltoA} of $R$-algebras
\begin{equation*}
    \Phi^{\geq 1}: Cl \rightarrow \mathcal{A}^{\geq 1}
\end{equation*}
is an isomorphism.
\end{proposition}
\begin{proof}
 We need to prove the injectivity. Consider the composition
 \begin{align}
     \varphi^{(1)\geq 1} \circ \Phi^{\geq1}\colon Cl \rightarrow \mathrm{Cliff}_{\mathcal{A}^{(1)2}}(\mathcal{E}^{(1)}, \langle \cdot , \cdot \rangle ).
 \end{align}
 As a map of $R$-modules, it is
 \begin{align}\label{eq:EtoEk}
     \bigwedge\nolimits^{\bullet} \mathcal{E} \rightarrow \bigwedge\nolimits^{\bullet} \mathcal{E}^{(1)}, \;\;  \theta_{i}^{x} \mapsto (n+1)x_1^{n}y_1^{n}\theta^{x_{1}}_{i} + n x_1^{n+1}y_1^{n-1}\theta_{i}^{y_1}, \qquad \theta_{i}^{y} \mapsto \theta^{y_{1}}_{i}, \;\; \gamma_{i} \mapsto x_{1}\theta_i^{y_1} + y_1 \theta_i^{x_1}.
 \end{align}
The $R$-module $\Omega_{R}$ is torsion free, and so, too, is $\bigwedge\nolimits^{\bullet}\mathcal{E}$. Since after inverting $\eta$ the composition~\eqref{eq:EtoEk} becomes an isomorphism, the map $\Phi^{\geq 1}$ is injective.
\end{proof}
From now on we make no distinction between $Cl$ and $\mathcal{A}^{\geq 1}$.

\begin{remark}\label{phik1aspullback}
The $R$-algebra map
\begin{align}\label{Ageq1toAkgeq1}
   \varphi^{(k)\geq 1} \colon  \mathcal{A}^{\geq 1} &\rightarrow \mathcal{A}^{(k) \geq 1}\\
    \theta_{i}^{x} &\mapsto (n+2-k)(x_k y_k )^{n+1-k}\theta^{x_{k}}_{i} + (n+1-k)x_{k}^{n+2-k}y_k^{n-k}\theta_{i}^{y_{k}}, \nonumber\\
    \theta_{i}^{y} &\mapsto k(x_k y_k )^{k-1}\theta^{y_{k}}_{i} + (k-1)x_k^{k-2}y_{k}^{k}\theta^{x_k}_{i}, \nonumber
\end{align}
coincides with the map induced by pullback $\pi_{k}^{*}\colon \Omega_{R} \rightarrow \Omega_{\mathcal{A}^{(k)2}}$ of differential forms under~\eqref{projectk}.
\end{remark}
The map $\varphi^{(k)}$ preserves the
pairing $\langle \cdot , \cdot \rangle$ and the $\pi_{1}^{g} \times \mathrm{Fr} \times A^{\vee}$-action.
In particular, for $\chi_{0}^{n+1} \neq 1$ it induces the $\mathcal{A}^{2}_{\chi}$-algebra homomorphism
\begin{align*}
\mathrm{Cliff}\left( H^{1}\left(\pi_{1}^{g}, q^{-1}\Omega_{R}\otimes R_{\mathrm{fix}}\right) , \langle \cdot , \cdot \rangle \right) &\rightarrow \mathrm{Cliff}\left( H^{1}\left(\pi_{1}^{g}, q^{-1}\Omega_{\mathcal{A}^{(k)2}}\otimes \mathcal{A}^{(k)2}_{\chi} \right) , \langle \cdot , \cdot \rangle \right) .
\end{align*}
It is easy to see that it coincides with
\begin{align}\label{chiAgeq1toAkgeq1}
    \varphi_{\chi}^{(k)\geq 1} \circ \Phi_{\chi}^{\geq 1} \colon Cl_{\chi} \rightarrow \mathcal{A}^{(k) \geq 1}_{\chi} \qquad \theta^{x}_{\chi, i} \mapsto \delta_{k, n+1} \theta^{x_k}_{\chi, i}, \qquad
   \theta^{y}_{\chi, i} \mapsto \delta_{k, 1} \theta^{y_k}_{\chi, i}, \qquad
    \gamma_{i} &\mapsto 0.
\end{align}

\begin{definition}\label{defn:Xvee}
Affine spaces $(U_{k})_{k = 1}^{n+1}$ can be glued along the birational maps
\begin{equation*}
    x_{l} = x_{k}^{l-k+1}y_k^{l-k}, \qquad y_{l} = x_k^{k-l}y_k^{k-l+1},
\end{equation*}
uniquely defined by $\pi_k$. The result of gluing is a smooth variety
$X^{\vee}$ equipped with a symplectic resolution
\begin{align*}
    \rho \colon X^{\vee} \rightarrow X^{\vee}_{0}.
\end{align*}
\end{definition}

The $\pi_{1}^{g} \times \mathrm{Fr} \times A^{\vee}$-actions on the
$U_{k}$ are compatible, and hence induce a
$\pi_{1}^{g} \times \mathrm{Fr} \times A^{\vee}$-action on $X^{\vee}$.
The $\pi_{1}^{g}$-fixed points of $X^{\vee}$ are identified with the
disjoint union of the chart-wise fixed loci,
\begin{align*}
    (X^{\vee})^{\pi_{1}^{g}} = \bigsqcup_{k = 1}^{n+1} \mathrm{Spec}\, \mathcal{A}_{\chi}^{(k)2}.
\end{align*}

The $\mathcal{A}^{(k)2}$-modules $\mathcal{E}^{(k)}$ glue to the vector bundle
$q^{-1} \Omega^{1}_{X^{\vee}} \otimes  H^{1}_{c}(C)$ on $X^{\vee}$, which
inherits the symmetric pairing $\langle \cdot, \cdot \rangle$. The
associated Clifford algebra is a locally free sheaf of
$\mathcal{O}_{X^{\vee}}$-algebras
\begin{align*}
    \mathrm{Cliff} = \mathrm{Cliff}( q^{-1}\Omega_{X^{\vee}} \otimes H^{1}_{c}(C), \langle \cdot , \cdot \rangle ),
\end{align*}
and the homomorphisms~\eqref{Ageq1toAkgeq1} glue to an inclusion of sheaves
of algebras
\begin{align*}
    \rho^{*}Cl \rightarrow \mathrm{Cliff}.
\end{align*}
Note that this inclusion is not surjective on global sections. Indeed, $\Omega_{R}$ has three generators $dx, dy, d\eta$, but global sections of $\Omega_{X^{\vee}}$ are generated by $dx, dy, d\eta$, and $\eta(\frac{dx}{x} - \frac{dy}{y})$.
\begin{remark}
    In Appendix~\ref{Degreezerodescendants} we construct a similar map for the entire algebra $\mathcal{A}$.
\end{remark}
Analogously, for $\chi_{0}^{n+1} \neq 1$ the sheaves $\mathcal{E}^{(k)}_{\chi}$ glue to the locally
free sheaf of $\mathcal{O}_{(X^{\vee})^{\pi_{1}^{g}}}$-modules
\begin{align*}
      H^{1}\left(\pi_{1}^{g}, q^{-1}\Omega_{X^{\vee}}\big|_{(X^{\vee})^{\pi_{1}^{g}}} \otimes H^{1}_{c}(C)\right) = \bigoplus\limits_{k = 1}^{n+1}\mathcal{E}_{\chi}^{(k)},
\end{align*}
which inherits $\langle \cdot, \cdot \rangle$. The associated Clifford
algebra is a locally free sheaf of
$\mathcal{O}_{(X^{\vee})^{\pi_{1}^{g}}}$-algebras
\begin{align*}
    \mathrm{Cliff}_{\chi} = \mathrm{Cliff}\left(\bigoplus\limits_{k = 1}^{n+1}\mathcal{E}^{(k)}_{\chi}, \langle \cdot, \cdot \rangle \right) = \mathop{\mathlarger{\mathlarger{\times}}}\limits_{k = 1}^{n+1} \mathrm{Cliff}\left(\mathcal{E}^{(k)}_{\chi}, \langle \cdot, \cdot \rangle \right),
\end{align*}
and the homomorphisms~\eqref{chiAgeq1toAkgeq1} glue to a morphism of
sheaves of algebras
\begin{align*}
\rho^{*}\mathcal{A}^{\geq1}_{\chi} \rightarrow \mathrm{Cliff}_{\chi}.
\end{align*}
\section{Cohomology of the resolution}\label{sec:cohomology-Xvee}

We want to describe the cohomology of quasimaps as the cohomology of a natural sheaf on the resolution $X^{\vee}$. In this section we
construct this sheaf as the twisted tensor product
$\hat{\mathcal{O}}_{\mathrm{vir}} \otimes \mathcal{O}(\vec{\mathcal{M}})$
of a natural spin bundle and a line bundle depending on $\vec{\mathcal{M}} \in \mathrm{Pic}(C)^{\times (n+1)}$, and compute the local cohomology $H^{*}_{L}$ with support in the attracting
subvariety $L \subset X^{\vee}$.

\subsection{The symmetrised virtual structure sheaf}

We begin by constructing a locally free $\mathrm{Cliff}$-module
$\mathrm{Spin}$. Twisting it by a square root of the canonical bundle produces
the ``symmetrised virtual structure sheaf'' $\hat{\mathcal{O}}_{\mathrm{vir}}$.

Let $\mathrm{Cliff}$ be a sheaf of Clifford algebras. A \emph{spin module} $\mathrm{Spin}$ is an
$\mathcal{O}_{X^{\vee}}$-locally free $\mathrm{Cliff}$-module satisfying
\begin{align}\label{HAHB}
    \mathrm{Cliff} &\cong \mathrm{Spin}^{\vee} \mathbin{\mathop{\otimes}\limits_{\;\mathcal{O}_{X^{\vee}}}} \mathrm{Spin},\\
    \mathrm{Spin}^{\vee} &\cong \mathrm{Spin}.
\end{align}

Let us give a construction of $\mathrm{Spin}$ in the case of the Clifford algebra sheaf defined in Subsection~\ref{sub:sympres}. Consider a decomposition of $H^{1}_{c}(C)$ into two
Lagrangian subspaces
\begin{align*}
    H^{1}_{c}(C) = H_{A} \oplus H_{B},
\end{align*}
where
\begin{align*}
    H_{A} = \langle \gamma_{1}, \ldots, \gamma_{g} \rangle_{\overline{\mathbb{Q}}_{l}}, \;\;\; H_{B} = \langle \gamma_{g + 1}, \ldots, \gamma_{2g} \rangle_{\overline{\mathbb{Q}}_{l}}.
\end{align*}
Recall that the $\gamma_{i}$ were chosen to be $\mathrm{Fr}$-eigenvectors
with eigenvalues $\alpha_{i} = \pm q^{1/2}$; since $\gamma_{g + i} = \gamma_{i}^{\vee}$
for $i = 1, \ldots, g$, we have
\begin{align*}
   \alpha_{i+g} = q/\alpha_{i} = \alpha_i , \; i = 1, \ldots, g.
\end{align*}

\begin{definition}\label{defn:spin-expl}
$\mathrm{Spin}$ is a $\mathrm{Cliff}$-module
\begin{align}\label{spindef}
    \mathrm{Spin} = \left( \mathrm{Cliff}/(q^{-1}\Omega^{1}_{X^{\vee}} \otimes H_{B})\right) \otimes \left( \bigwedge\nolimits^{2g} q^{-1}\Omega^{1}_{X^{\vee}} \otimes H_{B} \right)^{1/2},
\end{align}
where $(q^{-1}\Omega^{1}_{X^{\vee}} \otimes H_{B})$ denotes the right ideal
generated by $q^{-1}\Omega^{1}_{X^{\vee}} \otimes H_{B}$.
\end{definition}

As an $\mathcal{O}_{X^{\vee}}$-module the quotient is the exterior algebra
\begin{align*}
\mathrm{Cliff}/(q^{-1}\Omega^{1}_{X^{\vee}} \otimes H_{B}) \cong \bigwedge q^{-1}\Omega^{1}_{X^{\vee}} \otimes H_{A},
\end{align*}
and the square root in~\eqref{spindef} can be chosen as
\begin{align*}
    \left( \bigwedge\nolimits^{2g} q^{-1}\Omega^{1}_{X^{\vee}} \otimes H_{B} \right)^{1/2} = \prod\limits_{i = 1}^{g}\frac{q^{1/2}}{\alpha_{i}} \; \mathcal{O}_{X^{\vee}},
\end{align*}
the numerical factor representing the $\mathrm{Fr}$-twist on the fibers of
the trivial bundle. $\mathrm{Spin}$ can be viewed as the
structure sheaf of the supervariety $\mathcal{X}^{\vee} = H_{A} \otimes T X^{\vee}$,
with $H_{A}$ placed in odd degree.

We define a square root of the canonical bundle
\begin{align*}
    (\mathcal{K}_{X^{\vee}})^{1/2} = (\Omega^{2}_{X^{\vee}})^{1/2} = q^{1/2}\mathcal{O}_{X^{\vee}}.
\end{align*}

\begin{definition}\label{defn:Ovir}
The \emph{symmetrised virtual structure sheaf} is
\begin{align*}
\hat{\mathcal{O}}_{\mathrm{vir}} = \mathrm{Spin} \otimes (\mathcal{K}_{X^{\vee}})^{1/2}.
\end{align*}
\end{definition}

As an $\mathcal{O}_{X^{\vee}}$-module it is the twisted exterior algebra
\begin{align*}
   \hat{\mathcal{O}}_{\mathrm{vir}} \cong q^{1/2}\prod\limits_{i = 1}^{g}\frac{q^{1/2}}{\alpha_{i}} \left(\bigwedge q^{-1}\Omega^{1}_{X^{\vee}} \otimes H_{A}\right).
\end{align*}

In the nontrivial case $\chi_{0}^{n+1} \neq 1$, set
\begin{align}\label{spinchidef}
    \mathrm{Spin}_{\chi} = \bigoplus\limits_{k = 1}^{n+1} \left( \mathrm{Cliff}_{k}(\mathcal{E}^{(k)}_{\chi}, \langle \cdot, \cdot \rangle)/\left(V_{k}\right) \right)\otimes \left( \bigwedge\nolimits^{2g-2} V_{k} \right)^{1/2},
\end{align}
where
\begin{align*}
    V_{k} = q^{-1}(T^{*}_{p_{k}}X^{\vee})^{f}\otimes H^{1}_{c}(C, \chi^{-n-1} ),
\end{align*}
and $(T^{*}_{p_{k}}X^{\vee})^{f}$ is the subspace of $T^{*}_{p_{k}}X^{\vee}$
on which $A^{\vee}$ acts by a positive power of $a$:
\begin{align*}
   q^{-1}(T^{*}_{p_{k}}X^{\vee})^{f}\otimes H^{1}_{c}(C, \chi^{-n-1} ) \cong \langle f^{(k)}\langle \gamma_{1}^{\chi^{-1}} \rangle , \ldots, f^{(k)}\langle \gamma_{2g-2}^{\chi^{-1}} \rangle \rangle_{\overline{\mathbb{Q}}_{l}}.
\end{align*}
The square root in~\eqref{spinchidef} can be chosen as
\begin{align*}
    \left( \bigwedge\nolimits^{2g-2} V_{k} \right)^{1/2} \cong q^{\frac{1}{2}(1-g) + (k-1)(1-g)}(q^{1/2}a)^{-(n+1)(1-g)}\overline{\mathbb{Q}}_{l}.
\end{align*}

\begin{definition}\label{defn:Ovir-chi}
The \emph{symmetrised virtual structure sheaf} in the nontrivial case is
\begin{align*}
    \hat{\mathcal{O}}_{\mathrm{vir}, \chi} = \mathrm{Spin}_{\chi} \otimes (\mathcal{K}_{(X^{\vee})^{\pi_{1}^{g}}})^{1/2} \cong \mathrm{Spin}_{\chi}.
\end{align*}
\end{definition}

As a $\overline{\mathbb{Q}}_{l}$-vector space it is isomorphic to a direct sum
of exterior algebras:
\begin{align*}
    \hat{\mathcal{O}}_{\mathrm{vir}, \chi} \cong \bigoplus\limits_{k = 1}^{n+1} \left( \bigwedge\nolimits_{\overline{\mathbb{Q}}_{l}} \langle e^{(k)}\langle \gamma_{1}^{\chi} \rangle, \ldots, e^{(k)}\langle \gamma_{2g-2}^{\chi}\rangle \rangle_{\overline{\mathbb{Q}}_{l}} \right) \otimes \left( \bigwedge\nolimits^{2g-2} V_{k} \right)^{1/2}.
\end{align*} 
\subsection{The line bundle $\mathcal{O}(\vec{\mathcal{M}})$}

The entire construction starting from $QM\vec{N}$ depends on the tuple
$\vec{\mathcal{M}} \in \mathrm{Bun}_{T}(C)$. In this subsection we encode
this dependence on the $X^{\vee}$-side as a
$\mathrm{Fr} \times A^{\vee}$-equivariant line bundle
$\mathcal{O}(\vec{\mathcal{M}})$ on $X^{\vee}$, which will ultimately
twist $\hat{\mathcal{O}}_{\mathrm{vir}}$ in the main comparison.

Let $\vec{\mathcal{M}} = (\mathcal{M}_{1}, \ldots, \mathcal{M}_{n+1}) \in \mathrm{Bun}_{T}(C)$
with $\deg(\mathcal{M}_{k}) = m_k$. To this tuple we associate a
$\mathrm{Fr} \times A^{\vee}$-equivariant line bundle
$\mathcal{O}(\vec{\mathcal{M}})$ on $X^{\vee}$, uniquely determined up to
isomorphism by the following conditions:
\begin{align}
&\mathrm{wt}_{\mathrm{Fr}\times A^{\vee}} \left( \mathcal{O}(\vec{\mathcal{M}})\big|_{p_1 } \right) = q^{-\frac{1}{2}\sum\limits_{l < r}(m_{l} - m_{r}+1-g)}  (q^{1/2}a)^{-\sum\limits_{l=1}^{n+1}m_{l} - m_{1}}, \label{linearization} \\
    &\mathcal{O}(\vec{\mathcal{M}})\big|_{\bar{L}_{k}} \cong \mathcal{O}_{\mathbb{P}^{1}}(m_k - m_{k+1} ), \;\;\; k = 1, ..., n. \nonumber
\end{align}
See Appendix~\ref{Degreezerodescendants} for the natural origin of this bundle.
It follows that the $\mathrm{Fr}\times A^{\vee}$-weights of the fibers of
$\mathcal{O}(\vec{\mathcal{M}})$ at the fixed points are
\begin{align*}
   \mathrm{wt}_{\mathrm{Fr}\times A^{\vee}} \left( \mathcal{O}(\vec{\mathcal{M}})\big|_{p_k } \right) = q^{-\frac{1}{2}\sum\limits_{l < r}(m_{l} - m_{r} + 1 - g) + \sum\limits_{l = 1}^{k-1}m_{l} - m_{k}} (q^{1/2}a)^{-\sum\limits_{l=1}^{n+1}m_{l} - m_{k}}.
\end{align*}
Since $\mathcal{O}(\vec{\mathcal{M}})$ is trivialized over each $U_{k}$,
\begin{align*}
    \mathcal{O}(\vec{\mathcal{M}})\big|_{U_{k} } \cong q^{-\frac{1}{2}\sum\limits_{l < r}(m_{l} - m_{r} + 1 - g) + \sum\limits_{l = 1}^{k-1}m_{l} - m_{k}} (q^{1/2}a)^{-\sum\limits_{l=1}^{n+1}m_{l} - m_{k}} \mathcal{O}_{X^{\vee}}.
\end{align*}
The assignment
\begin{align*}
   \mathrm{Bun}_{T}(C) &\rightarrow \mathrm{Pic}_{\mathrm{Fr}\times A^{\vee}}(X^{\vee}),\\
   \vec{\mathcal{M}} &\mapsto \mathcal{O}(\vec{\mathcal{M}}) \nonumber
\end{align*}
is evidently a group homomorphism. 
\subsection{Local cohomology of $X^{\vee}$}

We introduce the local cohomology $H^{*}_{L}$ of $X^{\vee}$ with
support in the attracting set $L$ of the $A^{\vee}$-action. The
stratification of $L$ into attracting cells $L_{k}$ induces a filtration on
$H^{*}_{L}$.

\begin{definition}\label{defn:L}
Define $L \subset X^{\vee}$ to be the attracting set of the
$A^{\vee}$-action,
\begin{align*}
    L = \{ x \in X^{\vee}: \exists \lim\limits_{a \rightarrow \infty} a.x \}.
\end{align*}
Denote by $p_1, \ldots, p_{n+1}$ the $n+1$ fixed points of the action, and
decompose $L$ into attracting manifolds
\begin{align*}
    L = \bigsqcup\limits_{k = 1}^{n+1}L_{k},\;\;\; L_{k} = \{x \in X^{\vee}: \lim\limits_{a \rightarrow \infty} a.x = p_{k} \}.
\end{align*}
\end{definition}

Each closure $\bar{L}_{k}$ is a smooth Lagrangian subvariety, and
$\bar{L}_{k}$ meets $\bar{L}_{k+1}$ transversely at $p_{k+1}$, for
$k = 1, \ldots, n$. For $k = 1, \ldots, n+1$ define
\begin{align*}
    L^{k} = \bigcup\limits_{l = 1}^{k} L_{l},
\end{align*}
so that $L^{n+1} = L$, and observe the decomposition into closed and open
subsets
\begin{align}\label{Lkdec}
    L^{k} = L_{k} \bigsqcup L^{k-1}, \;\;\; k = 2, \ldots, n + 1.
\end{align}

The functor $\Gamma_{L}$ of sections with support in $L$ is defined by
\begin{align*}
    \Gamma_{L}\colon \mathrm{Qcoh}_{X^{\vee}} &\longrightarrow \Gamma(X^{\vee}, \mathcal{O}_{X^{\vee}})\text{-}\mathrm{mod}\\
    \mathcal{F} &\mapsto \{s \in \Gamma(X^{\vee}, \mathcal{F}): \mathcal{I}^{N}_{L}s = 0, N \gg 0\}, \nonumber
\end{align*}
where $\mathcal{I}_{L}$ is the ideal sheaf of $L$ generated by $y$:
\begin{align*}
    \mathcal{I}_{L} = y \mathcal{O}_{X^{\vee}} \subset \mathcal{O}_{X^{\vee}}.
\end{align*}
The local cohomology groups are defined as the right derived functors
\begin{align*}
    H^{i}_{L}(X^{\vee}, \mathcal{F}) = R^{i}\Gamma_{L}(\mathcal{F}).
\end{align*}
The decomposition~\eqref{Lkdec} yields the long exact sequence of a pair
\begin{align}\label{Hfilt}
H^{*}_{L_{k}}(X^{\vee}, \mathcal{F}) \rightarrow H^{*}_{L^{k}}(X^{\vee}, \mathcal{F}) \rightarrow H^{*}_{L^{k-1}}(X^{\vee}, \mathcal{F}) \rightarrow H^{*}_{L_{k}}(X^{\vee}, \mathcal{F})[1], \;\; k = 2, \ldots, n + 1.
\end{align}
In the following we compute
$H^{*}_{L}(X^{\vee}, \hat{\mathcal{O}}_{\mathrm{vir}}\otimes \mathcal{O}(\vec{\mathcal{M}}))$.
We will see that only the first cohomology is nontrivial, so the long
exact sequence~\eqref{Hfilt} becomes short and equips
$H^{*}_{L}(X^{\vee}, \hat{\mathcal{O}}_{\mathrm{vir}}\otimes \mathcal{O}(\vec{\mathcal{M}}))$
with a filtration whose associated graded pieces are the
$H^{*}_{L_{k}}(X^{\vee}, \hat{\mathcal{O}}_{\mathrm{vir}}\otimes \mathcal{O}(\vec{\mathcal{M}}))$.
\subsection{Cohomology with support in attracting subvarieties}

Since $L_{k} \subset U_{k}$, the local cohomology can be computed in the
chart $U_{k}$:
\begin{align*}
    H^{*}_{L_{k}} (X^{\vee}, \hat{\mathcal{O}}_{\mathrm{vir}}\otimes \mathcal{O}(\vec{\mathcal{M}})) = H^{*}_{L_{k}} (U_{k},\hat{\mathcal{O}}_{\mathrm{vir}}\otimes \mathcal{O}(\vec{\mathcal{M}})).
\end{align*}
The subvariety $L_{k}$ is closed in $U_{k}$, with ideal sheaf
$\mathcal{I}_{L_{k}} = y_k \mathcal{O}_{X^{\vee}}$. Applying the exact
sequence of a pair to $L_{k} \subset U_{k}$, we obtain the long exact
sequence
\begin{align*}
    0 &\rightarrow H^{0}_{L_{k}}(U_{k}, \hat{\mathcal{O}}_{\mathrm{vir}}\otimes \mathcal{O}(\vec{\mathcal{M}})) \rightarrow H^{0}(U_{k}, \hat{\mathcal{O}}_{\mathrm{vir}}\otimes \mathcal{O}(\vec{\mathcal{M}})) \rightarrow H^{0}(U_{k}\setminus L_{k}, \hat{\mathcal{O}}_{\mathrm{vir}}\otimes \mathcal{O}(\vec{\mathcal{M}})) \rightarrow\\
    &\rightarrow H^{1}_{L_{k}}(U_{k}, \hat{\mathcal{O}}_{\mathrm{vir}}\otimes \mathcal{O}(\vec{\mathcal{M}})) \rightarrow H^{1}(U_{k}, \hat{\mathcal{O}}_{\mathrm{vir}}\otimes \mathcal{O}(\vec{\mathcal{M}})) \rightarrow ... \nonumber
\end{align*}
The group
$H^{0}_{L_{k}}(U_{k}, \hat{\mathcal{O}}_{\mathrm{vir}}\otimes \mathcal{O}(\vec{\mathcal{M}}))$
vanishes because the sheaf
$\hat{\mathcal{O}}_{\mathrm{vir}}\otimes \mathcal{O}(\vec{\mathcal{M}})$
is locally free. Since $U_{k}$ and $U_{k} \setminus L_{k}$ are affine,
$\hat{H}^{i}(U_{k}, \hat{\mathcal{O}}_{\mathrm{vir}}\otimes \mathcal{O}(\vec{\mathcal{M}})) = 0$
for $i > 1$. Consequently
\begin{align*}
    H^{1}_{L_{k}}(U_k , \hat{\mathcal{O}}_{\mathrm{vir}}\otimes \mathcal{O}(\vec{\mathcal{M}})) = \frac{H^{0}(U_{k}\setminus L_{k}, \hat{\mathcal{O}}_{\mathrm{vir}}\otimes \mathcal{O}(\vec{\mathcal{M}}))}{H^{0}(U_{k}, \hat{\mathcal{O}}_{\mathrm{vir}}\otimes \mathcal{O}(\vec{\mathcal{M}}))} = \frac{\overline{\mathbb{Q}}_{l}[y_{k}^{\pm 1}]}{\overline{\mathbb{Q}}_{l}[y_{k}]} \otimes_{\overline{\mathbb{Q}}_{l}[y_{k}]} H^{0}(U_{k}, \hat{\mathcal{O}}_{\mathrm{vir}}\otimes \mathcal{O}(\vec{\mathcal{M}}))
\end{align*}
as an $\mathcal{A}^{(k)}$-module, and this is the only nonzero local
cohomology group. Expanding further,
\begin{align}
  &\frac{\overline{\mathbb{Q}}_{l}[y_{k}^{\pm 1}]}{\overline{\mathbb{Q}}_{l}[y_{k}]} \otimes_{\overline{\mathbb{Q}}_{l}[y_{k}]}H^{0}(U_{k}, \hat{\mathcal{O}}_{\mathrm{vir}}\otimes \mathcal{O}(\vec{\mathcal{M}})) =\\
  &q^{-\frac{1}{2}\sum\limits_{l < r}(m_{l} - m_{r} + 1 - g) + \sum\limits_{l = 1}^{k-1}m_{l} - m_{k}} (q^{1/2}a)^{-\sum\limits_{l=1}^{n+1}m_{l} - m_{k}} q^{1/2}\left(\prod\limits_{i = 1}^{g}\frac{q^{1/2}}{\alpha_{i}} \right)\frac{\overline{\mathbb{Q}}_{l}[y_{k}^{\pm 1}]}{\overline{\mathbb{Q}}_{l}[y_{k}]} \otimes_{\overline{\mathbb{Q}}_{l}[y_{k}]}\bigwedge \left(q^{-1}\Omega^{1}_{U_{k}} \otimes V_{A}\right).\label{spinkaka}
\end{align}
Explicitly, in the coordinates $x_k, y_k, \theta^{x_{k}}_{i}, \theta^{y_{k}}_{i}$,
\begin{align*}
    \frac{\overline{\mathbb{Q}}_{l}[y_{k}^{\pm 1}]}{\overline{\mathbb{Q}}_{l}[y_{k}]} \otimes_{\overline{\mathbb{Q}}_{l}[y_{k}]}\bigwedge \left(q^{-1}\Omega^{1}_{U_{k}} \otimes V_{A}\right) = y_{k}^{-1}\overline{\mathbb{Q}}_{l}[x_{k}, y_{k}^{-1}] \otimes \bigwedge\nolimits_{\overline{\mathbb{Q}}_{l}} \langle \theta^{x_{k}}_{1}, \ldots, \theta^{x_{k}}_{g}, \theta^{y_{k}}_{1}, \ldots, \theta^{y_{k}}_{g} \rangle_{\overline{\mathbb{Q}}_{l}},
\end{align*}
with the action of $\mathcal{A}^{(k)}$ given by
\begin{align*}
    &x_k = x_k \cdot \;\;\; D_{x_k } = - \partial_{y_{k}}\\
    &y_k = y_k \cdot \;\;\; D_{y_k } = \partial_{x_{k}} \nonumber\\
    &\theta_{i}^{x_k} = \theta_{i}^{x_k} \cdot \;\;\; \theta_{i}^{y_k} = \theta_{i}^{y_k} \cdot, \;\;\; i = 1, \ldots, g \nonumber\\
    &\theta^{\vee y_k}_{i} = -\partial_{\theta^{x_k}_{i}}, \;\;\; \theta^{\vee x_k}_{i} = \partial_{\theta^{y_k}_{i}}, \;\;\; i = 1, \ldots, g. \nonumber
\end{align*}
We see that $x_{k}, D_{x_{k}}, \theta^{x_{k}}_{i}$ for
$i = 1, \ldots, 2g$ freely generate this module from the cyclic vector
\begin{align*}
    y_{k}^{-1} \theta_{1}^{y_{k}} \wedge \cdots \wedge \theta^{y_{k}}_{g},
\end{align*}
which has $\mathrm{Fr} \times A^{\vee}$-weight
$q^{k-1}\bigl(\prod_{i=1}^{g} q^{-k}\alpha_{i}\bigr)(q^{1/2}a)^{-(n+1)(1-g)}$. Keeping track of the twists
in~\eqref{spinkaka}, we conclude that, as an $\mathcal{A}^{(k)}$-module,
$H^{1}_{L_{k}}(U_k , \hat{\mathcal{O}}_{\mathrm{vir}}\otimes \mathcal{O}(\vec{\mathcal{M}}))$
is isomorphic to
\begin{equation}\label{eq:oldisohk}
    q^{\frac{1}{2}\kappa_{k}(\vec{m})} a^{\rho_{k}(\vec{m})} \otimes \overline{\mathbb{Q}}_{l}[x_{k}, D_{x_{k}}]\otimes \bigwedge\nolimits_{\overline{\mathbb{Q}}_{l}} \langle \theta^{x_{k}}_{1}, \ldots, \theta^{x_{k}}_{2g} \rangle_{\overline{\mathbb{Q}}_{l}},
\end{equation}
where
\begin{align*}
    &\kappa_{k}(\vec{m}) = 1 - g -\sum\limits_{l < r}(m_{l} - m_{r} + 1 - g) + \sum\limits_{l = 1}^{k-1}(m_{l} - m_{k} + 1 - g) -\sum\limits_{l = k}^{n+1}(m_{l} - m_{k} + 1 - g),\\
    &\rho_{k}(\vec{m}) = -\sum\limits_{l=1}^{n+1}(m_{l} - m_{k} + 1 - g).
\end{align*}
In the case $\chi_{0}^{n+1} \neq 1$,
\begin{align*}
    H^{*}_{L}((X^{\vee})^{\pi_{1}^{g}}, \hat{\mathcal{O}}_{\mathrm{vir},\chi}\otimes\mathcal{O}(\vec{\mathcal{M}})) = \bigoplus_{k = 1}^{n+1}H_{L_{k}}^{0}((X^{\vee})^{\pi_{1}^{g}}, \hat{\mathcal{O}}_{\mathrm{vir},\chi}\otimes \mathcal{O}(\vec{\mathcal{M}})),
\end{align*}
where $H_{L_{k}}^{0}((X^{\vee})^{\pi_{1}^{g}}, \hat{\mathcal{O}}_{\mathrm{vir},\chi}\otimes \mathcal{O}(\vec{\mathcal{M}}))$ is isomorphic to
\begin{align}\label{eq:chioldisohk}
 q^{\frac{1}{2}\kappa_{k}(\vec{m})} a^{\rho_{k}(\vec{m})}  \otimes \bigwedge\nolimits_{\overline{\mathbb{Q}}_{l}} \langle e^{(k)}\langle \gamma_{1}^{\chi} \rangle, \ldots, e^{(k)}\langle \gamma_{2g-2}^{\chi}\rangle \rangle_{\overline{\mathbb{Q}}_{l}}.
\end{align}
\section{Comparison of the cohomologies}\label{sec:cohomology-QMN}

In this section we compare the two filtrations constructed in the
previous sections: on the $QM\vec{N}$-side, the Bia{\l}ynicki-Birula
filtration of Proposition~\ref{pr:birula-filtration}; on the
$X^{\vee}$-side, the filtration given by the long exact
sequence~\eqref{Hfilt}. We first compute the cohomology of each fixed
locus $QM\vec{N}(k)$ as an $\mathcal{A}^{(k)}$-module, match the graded
pieces of the two filtrations, and then prove that the matched
filtrations assemble into an isomorphism of $\mathcal{A}$-modules (in
all but an exceptional configuration).

Recall that the short exact sequence~\eqref{birulaqmn},
\begin{align*}
    0 \rightarrow  \hat{H}^{*}_{c}(QM\vec{N}(k))[[-2\mathrm{rk}\,N(k)]] \rightarrow  \hat{H}^{*}_{c}\left( \bigcup\limits_{l = 1}^{k}\mathrm{Attr}(l)\right) \rightarrow \hat{H}^{*}_{c}\left( \bigcup\limits_{l = 1}^{k-1}\mathrm{Attr}(l)\right) \rightarrow 0, \;\;\; k = 2, \ldots, n+1,
\end{align*}
induces a filtration on $H^{*}_{c}(QM\vec{N})$ with graded pieces
$H^{*}_{c}(QM\vec{N}(k))[[-2\mathrm{rk}\,N(k)]]$. We first compute the associated graded module.

\subsection{The associated graded module}

We explicitly compute the action of $\mathcal{A}^{(k)}$ on $\hat{H}^{*}_{c}(QM\vec{N}(k), \chi)$ and identify it with the local cohomology
$H^{*}_{L_{k}}(X^{\vee}, \hat{\mathcal{O}}_{\mathrm{vir}}\otimes \mathcal{O}(\vec{\mathcal{M}}))$
computed in Section~\ref{sec:cohomology-Xvee}.

Theorem~\ref{thm:Ak2-triv} identifies
$\mathcal{A}^{(k)}$ with a Weyl-Clifford algebra, acting on the Verma module
\begin{align*}
    H^{*}_{c}(QM\vec{N}(k)) \cong H^{*}_{c}(S^{\bullet}C) \cong S^{\bullet}H^{\text{even}}(C) \otimes \bigwedge H^{\text{odd}}(C) \cong \overline{\mathbb{Q}}_{l}[\mathbf{p}, \mathbf{u}]\otimes \bigwedge \langle \gamma_1, \ldots, \gamma_{2g} \rangle_{\overline{\mathbb{Q}}_{l}}.
\end{align*}
Explicitly, the action of $\mathcal{A}^{(k)}$ is given by
\begin{align}\label{Akactcoord}
    &x_{k} = \mathbf{p}\cdot, \;\;\; \theta^{x_{k}}_{i} =  \gamma_{i} \cdot, \;\;\; D_{x_k } = \mathbf{u} \cdot,\\
    &y_k = \partial_{\mathbf{u}}, \;\;\; \theta^{\vee y_{k}}_{i} = -\partial_{\gamma_i }, \;\;\; D_{y_k } = \partial_{\mathbf{p}}, \nonumber
\end{align}
where $\mathbf{u} \in H^{0}(C)$ is the unit. The operators
$x_k, D_{x_k}, \theta_{i}^{x_{k}}$, $i = 1, \ldots, 2g$, freely generate
this module from the cyclic vector
\begin{align*}
    1 \in H_{c}^{0}(S^{0}C) \cong H^{0}_{c}(QM\vec{N}_{-m_{k}}(k)),
\end{align*}
which has $\mathrm{Fr} \times A^{\vee}$-weight
$\prod_{l = 1}^{n+1} a^{-\deg\left(\mathcal{M}_{l} \otimes \mathcal{M}_{k}^{-1}\right)}$.
Keeping track of all the shifts, we get
\begin{align}\label{hattiso}
       &\hat{H}^{*}_{c}(QM\vec{N}(k))[[-2\mathrm{rk}N(k)]] \cong  \overline{\mathbb{Q}}_{l}[x_{k}, e \langle 1 \rangle ]\otimes \bigwedge\nolimits_{\overline{\mathbb{Q}}_{l}} \langle \theta^{x_{k}}_{1}, \ldots, \theta^{x_{k}}_{2g} \rangle_{\overline{\mathbb{Q}}_{l}} [[-\kappa_{k}(\vec{m})]] a^{\rho_{k}(\vec{m})}.
\end{align}
Analogously, in the nontrivial local system case, Theorem~\ref{thm:Ak2-triv} gives
\begin{align}\label{hqmnkchi}
     H^{*}_{c}(QM\vec{N}(k), \chi) \cong \bigoplus\limits_{m = 0}^{\infty} H^{*}_{c}\left( S^{m}C, (\chi^{n+1})^{(m)}\right) \cong \bigwedge H^{1}(C, \chi^{n+1}) \cong  \bigwedge \langle \gamma_1^{\chi}, \ldots, \gamma_{2g-2}^{\chi} \rangle_{\overline{\mathbb{Q}}_{l}},
\end{align}
where $(\chi^{n+1})^{(m)}$ is the natural local system on $S^{m}C$ whose
stalk at $q_{1} + \ldots + q_{m} \in S^{m}C$ is
$\chi^{n+1}_{q_{1}} \otimes \cdots \otimes \chi^{n+1}_{q_{m}}$. The first
isomorphism in~\eqref{hqmnkchi} follows from the definition of the local
system $\chi$ on $QM\vec{N}$ as the pullback of $\chi$ from
$\mathrm{Pic}(C)$ under~\eqref{Ns}. The action of $\mathcal{A}_{\chi}^{(k)}$
is given by
\begin{align*}
    e\langle \gamma_{i}^{\chi} \rangle &= \gamma_{i}^{\chi} \cdot,\\
     f\langle \gamma_{i}^{\chi^{-1}} \rangle &= - \partial_{\gamma_{i}^{\chi}}. \nonumber
\end{align*}
The operators $e\langle \gamma_{i}^{\chi} \rangle$ for
$i = 1, \ldots, 2g-2$ freely generate this module from the one-dimensional
cyclic vector
\begin{align*}
    1 \in H_{c}^{0}(S^{0}C) \cong H^{0}_{c}(QM\vec{N}_{-m_{k}}(k), \chi).
\end{align*}
Applying the shifts, we obtain 
\begin{align}\label{chihattiso}
    \hat{H}^{*}_{c}(QM\vec{N}(k), \chi)[[-2\mathrm{rk}N(k)]] \cong \bigwedge\nolimits_{\overline{\mathbb{Q}}_{l}} \langle e^{(k)}\langle \gamma_{1}^{\chi} \rangle, \ldots, e^{(k)}\langle \gamma_{2g-2}^{\chi} \rangle \rangle_{\overline{\mathbb{Q}}_{l}} [[-\kappa_{k}(\vec{m})]] a^{\rho_{k}}. 
\end{align}

The associated graded modules of the two filtrations coincide.
\begin{proposition}\label{pr:graded-agree}
For all $k$ and any $\chi_{0}$, there is an isomorphism
\begin{align}\label{isok}
    \hat{H}^{*}_{c}(QM\vec{N}(k), \chi)[[-2\mathrm{rk}\,N(k)]] \cong H^{*}_{L_{k}}(X^{\vee}, \hat{\mathcal{O}}_{\mathrm{vir}}\otimes \mathcal{O}(\vec{\mathcal{M}}))
\end{align}
of $\mathcal{A}_{\chi}^{(k)}$-modules.
\end{proposition}

\begin{proof}
Compare formulas~\eqref{eq:oldisohk} and~\eqref{eq:chioldisohk}
with~\eqref{hattiso} and~\eqref{chihattiso}: both sides are isomorphic
as $\mathcal{A}^{(k)}$-modules, and since the
$\mathcal{A}$-action factors through $\mathcal{A}^{(k)}$ on each side, the
isomorphism is $\mathcal{A}$-linear.
\end{proof}
\subsection{Intermediate algebras}
We introduce an auxiliary family of algebras that will be useful to analyze the filtration on $\hat{H}_{c}^{*}(QM\vec{N})$. Consider the
restriction of $\iota$ to
$C \times \bigcup_{l = k}^{s}\mathrm{Attr}(l)$:
\begin{align}\label{iotaks}
    \iota^{(k, s)}\colon C \times \bigcup\limits_{l = k}^{s}\mathrm{Attr}(l) \rightarrow C \times \bigcup\limits_{l = k}^{s}\mathrm{Attr}(l).
\end{align}
In analogy with~\eqref{ekfkdefinition}, we define operators
$e^{(k, s)}\langle \gamma_{i}^{\chi} \rangle, f^{(k, s)}\langle \gamma_{i}^{\chi^{-1}} \rangle$
on $H_{c}^{*}(\bigcup_{l=k}^{s}\mathrm{Attr}(l), \chi)$ using~\eqref{iotaks}.

\begin{definition}\label{defn:Aks}
For $1 \leq k \leq s \leq n+1$, let
\begin{align*}
    \mathcal{A}^{(k,s)}_{\chi} \subset \mathrm{End} \left( H_{c}^{*}\left (\bigcup\limits_{l = k}^{s}\mathrm{Attr}(l), \chi \right) \right)
\end{align*}
be the algebra generated by $\eta$, $\gamma_{i}$, and descendants of degree one.
Observe that $\mathcal{A}^{(k,k)}_{\chi} = \mathcal{A}^{(k)}_{\chi}$ and
$\mathcal{A}^{(1,n+1)}_{\chi} = \mathcal{A}_{\chi}$.
\end{definition}

For $\chi_{0}^{n+1} \neq 1$, the generators satisfy the following relations:
\begin{align}
    &[e^{(k, s)}\langle \gamma_{i}^{\chi} \rangle, f^{(k, s)}\langle \gamma_{j}^{\chi^{-1}} \rangle] = \delta_{ij}(s - k + 1)\eta^{s-k}, \label{commutefsk}\\
    &[e^{(k, s)}\langle \gamma_{i}^{\chi} \rangle, \gamma_{j}] = [f^{(k, s)} \langle \gamma_{i}^{\chi} \rangle, \gamma_{j}] = 0, \;\;\; \eta^{s - k + 1} = 0, \;\;\; \eta^{s - k}\gamma_{i} = 0, \nonumber\\
    &\sum\limits_{i = 1}^{2g-2} f^{(k, s)}\langle \gamma_{i}^{\chi^{-1}} \rangle e^{(k, s)}\langle \gamma_{i}^{\chi} \rangle + \eta^{s-k}\chi_C(\mathcal{V}^{(s)}/\mathcal{V}^{(k-1)} \otimes \mathcal{K}_{C}^{-1/2} \otimes \mathcal{L}^{-1}) - (s - k)(s - k + 1)\eta^{s - k - 1}\Theta = 0. \nonumber
\end{align}
From the identities
\begin{align*}
    \iota_{*} &= \iota_{*}^{(k, s)}c_{n+1-s}\left( (\mathcal{N}/\mathcal{N}^{(s)}) \otimes TC \right),\\
    \iota^{*} &= c_{k-1} \left(\mathcal{N}^{(k-1)} \otimes TC \right) \iota^{(k, s)*}, \nonumber
\end{align*}
we obtain a natural homomorphism
\begin{align*}
    \mathcal{A}_{\chi} &\rightarrow \mathcal{A}_{\chi}^{(k, s)},\\
     e \langle \gamma_{i}^{\chi} \rangle &=  \eta^{n+1-s} e^{(k,s)}\langle \gamma_{i}^{\chi} \rangle, \nonumber\\
     f \langle \gamma_{i}^{\chi^{-1}} \rangle &=  \eta^{k-1} f^{(k, s)}\langle \gamma_{i}^{\chi^{-1}} \rangle, \; i = 1, ..., 2g-2. \nonumber
\end{align*}
The action of $\mathcal{A}_{\chi}^{(k, s)}$ on
$\hat{H}_{c}^{*}(QM\vec{N}(l), \chi)[[-2\mathrm{rk}\,N(l)]]$ for
$k \leq l \leq s$ is given by
\begin{align}\label{AsktoAl}
\mathcal{A}^{(k, s)}_{\chi} &\rightarrow \mathcal{A}^{(l)}_{\chi},\\
 e^{(k, s)} \langle \gamma_{i}^{\chi} \rangle &=  \delta_{l, s} e^{(l)}\langle \gamma_{i}^{\chi} \rangle, \label{elek}\\
     f^{(k, s)} \langle \gamma_{i}^{\chi^{-1}} \rangle & = \delta_{l, k}f^{(l)}\langle \gamma_{i}^{\chi^{-1}} \rangle, \; i = 1, ..., 2g-2. \label{flfk}
\end{align}
More generally, for any $\chi$ and $k \leq k' \leq s' \leq s$ there is a similar map
\begin{equation}\label{eq:restscal}
    \mathcal{A}^{(k, s)}_{\chi} \rightarrow \mathcal{A}^{(k', s')}_{\chi}.
\end{equation}
\subsection{The trivial-character case}
In this subsection we prove the isomorphism between the cohomology of the quasimap space and that of the symplectic dual variety for the trivial local system.
\begin{proposition}\label{pr:comparison-triv}
Assume $\chi_{0}^{n+1}=1$. Then we have an isomorphism of $\mathcal{A}$-modules
\begin{align*}
    \hat{H}^{*}_{c}(QM\vec{N}) \cong H^{*}_{L}(X^{\vee}, \hat{\mathcal{O}}_{\mathrm{vir}}\otimes \mathcal{O}(\vec{\mathcal{M}})).
\end{align*}
\end{proposition}

\begin{proof}
We prove by induction on $k$ that
\begin{align}\label{comparison-k}
     H^{*}_{L^{k}}(X^{\vee}, \hat{\mathcal{O}}_{\mathrm{vir}}\otimes \mathcal{O}(\vec{\mathcal{M}})) \cong \hat{H}^{*}_{c}\left(\bigcup\limits_{l = 1}^{k}\mathrm{Attr}(l)\right)
\end{align}
as $\mathcal{A}^{(1, k)}$-modules; then taking $k = n+1$ and using $L^{n+1} = L$ together with
$\bigcup_{l = 1}^{n+1}\mathrm{Attr}(l) = QM\vec{N}$ yields the statement. The base case $k=1$
follows from~\eqref{isok}. For the inductive step, consider the two short exact sequences of
$\mathcal{A}^{(1, k)}$-modules:
\begin{align}\label{qmnfiltr}
   0 \rightarrow  \hat{H}^{*}_{c}(QM\vec{N}(k))[[-2\mathrm{rk}N(k)]] \rightarrow \hat{H}^{*}_{c}\left(\bigcup\limits_{l = 1}^{k}\mathrm{Attr}(l)\right) \rightarrow \hat{H}^{*}_{c}\left( \bigcup\limits_{l = 1}^{k-1}\mathrm{Attr}(l) \right) \rightarrow 0
\end{align}
\begin{align*}
   0 \rightarrow  H^{*}_{L_{k}}(X^{\vee}, \hat{\mathcal{O}}_{\mathrm{vir}}\otimes \mathcal{O}(\vec{\mathcal{M}})) \rightarrow H^{*}_{L^{k}}(X^{\vee}, \hat{\mathcal{O}}_{\mathrm{vir}}\otimes \mathcal{O}(\vec{\mathcal{M}})) \rightarrow H^{*}_{L^{k-1}}(X^{\vee}, \hat{\mathcal{O}}_{\mathrm{vir}}\otimes \mathcal{O}(\vec{\mathcal{M}})) \rightarrow 0,
\end{align*}
where $\mathcal{A}^{(1, k)}$ acts on the submodules and the quotients by means of~\eqref{eq:restscal}.
The submodules are isomorphic by~\eqref{isok}, and the quotients are
isomorphic by the induction hypothesis. Using the fact that the submodule is free over
$\overline{\mathbb{Q}}_{l}[x_{k}]$, one can show (see~\cite{KazOko23}, Section 6.4.3) that the extensions coincide as well.

\end{proof}
\subsection{The nontrivial-character extension problem}
When \(\chi_0^{n+1}\neq1\), the fixed locus of the resolved surface
\(X^\vee\) is reduced: it is the disjoint union of the \(n+1\) fixed
points. By contrast, the scheme-theoretic fixed locus of the singular
affinization \(X_0^\vee\) is the fat point
\[
\operatorname{Spec}\overline{\mathbb Q}_l[\eta]/(\eta^{n+1}).
\]
The difference between these two models is exactly the possible
\(\eta\)-extension appearing in the \(QM\vec N\)-filtration. Thus the
associated graded comparison is naturally expressed on the resolved fixed
locus, while the filtered comparison can remember the nonreduced structure. In this subsection we show that the filtration splits precisely when
$\eta = 0$, and establish the range of $g$ and $\vec{\mathcal{M}}$ in which $\eta$ acts nontrivially.

We have the decreasing filtration on the $\mathcal{A}$-module
$M = \hat{H}_{c}^{*}(QM\vec{N}, \chi)$ defined by
\begin{align*}
    F^{k}M = H^{*}_{c}\left(QM\vec{N} \setminus \bigcup\limits_{l = 1}^{k}\mathrm{Attr}(l) \right), \;\; k = 1, \ldots, n+1.
\end{align*}
By~\eqref{qmnfiltr}, the associated graded is
\begin{align}\label{grFM}
\mathrm{gr}\, F^{\bullet}M = \bigoplus\limits_{k = 1}^{n+1} \hat{H}_{c}^{*}(QM\vec{N}(k), \chi)[[-2\mathrm{rk}\,N(k)]].
\end{align}
Since
\begin{align*}
    H^{*}_{L}((X^{\vee})^{\pi_{1}^{g}}, \hat{\mathcal{O}}_{\mathrm{vir}, \chi}\otimes \mathcal{O}(\vec{\mathcal{M}})) \cong \bigoplus\limits_{k = 1}^{n+1} H^{*}_{L^{k}}(X^{\vee}, \hat{\mathcal{O}}_{\mathrm{vir}, \chi}\otimes \mathcal{O}(\vec{\mathcal{M}})),
\end{align*}
the module $M$ is isomorphic to
$H^{*}_{L}((X^{\vee})^{\pi_{1}^{g}}, \hat{\mathcal{O}}_{\mathrm{vir}, \chi}\otimes \mathcal{O}(\vec{\mathcal{M}}))$
if and only if $M \cong \mathrm{gr}\, F^{\bullet}M$. Let us give a useful characterization of when the filtration on $M$ splits.

\begin{proposition}\label{pr:splitting} For $M$ as above
\begin{align*}
    M \cong \mathrm{gr}\, F^{\bullet}M \Longleftrightarrow \eta = 0 \text{ on } M.
\end{align*}
\end{proposition}

\begin{proof}
The implication $(\Rightarrow)$ follows from~\eqref{eta0onk}. For
$(\Leftarrow)$, assume that $\eta$ acts by zero on $M$.

\emph{Step 1: Reduction to finite intervals.}
It suffices to prove that for
each $k = 1, \ldots, n+1$, the short exact sequence of
$\mathcal{A}_{\chi}$-modules
\begin{align*}
    0 \rightarrow F^{k}M \rightarrow F^{k-1}M \rightarrow \operatorname{gr}_kM \rightarrow 0
\end{align*}
splits. We prove by induction on $s = k, \ldots, n+1$ that
\begin{align*}
    0 \rightarrow F^{k}M/F^{s}M \rightarrow F^{k-1}M/F^{s}M \rightarrow \operatorname{gr}_kM \rightarrow 0
\end{align*}
splits as a short exact sequence of $\mathcal{A}^{(k, s)}_{\chi}$-modules
(note that $F^{k-1}M/F^{s}M = H_{c}^{*}(\bigcup_{l=k}^{s}\mathrm{Attr}(l), \chi)$
carries a natural $\mathcal{A}_{\chi}^{(k, s)}$-action). The required
statement then follows from $s = n+1$ and restriction of scalars
$\mathcal{A}_{\chi} \rightarrow \mathcal{A}^{(k, n+1)}_{\chi}$.

\emph{Step 2: Induction on \(s\).}
 
The base case $s = k$ follows from~\eqref{grFM}. Assume the statement for \(s-1\). Consider the short exact sequence
\begin{align*}
    0 \rightarrow \mathrm{gr}_{s}M \rightarrow F^{k-1}M/F^{s}M \overset{\pi_{k}}{\longrightarrow} F^{k-1}M/F^{s-1}M \rightarrow 0.
\end{align*}
By the induction hypothesis, there is a splitting of
$\mathcal{A}_{\chi}^{(k, s-1)}$-modules
\begin{align*}
    \mathrm{gr}_{k}M \hookrightarrow F^{k-1}M/F^{s-1}M,
\end{align*}
which, by restriction of scalars
$\mathcal{A}^{(k, s)}_{\chi} \to \mathcal{A}_{\chi}^{(k, s-1)}$, is also
a splitting of $\mathcal{A}^{(k, s)}_{\chi}$-modules. Let \(M_{k,s}\) be the inverse image of \(\operatorname{gr}_kM\) in
\(F^{k-1}M/F^sM\)
\begin{align*}
    M_{k, s} = \pi_{k}^{-1}\left(\mathrm{gr}_{k}M\right) \subset F^{k-1}M/F^{s}M.
\end{align*}
Then the $\mathcal{A}_{\chi}^{(k, s)}$-module $M_{k, s}$ fits into
\begin{align}\label{Mksseq}
     0 \rightarrow \operatorname{gr}_sM \rightarrow M_{k, s} \rightarrow \operatorname{gr}_kM \rightarrow 0.
\end{align}

\emph{Step 3: Clifford generators force splitting.}

By~\eqref{AsktoAl}, the operators
$f^{(k, s)} \langle \gamma_{i}^{\chi^{-1}} \rangle$ freely generate
$\hat{H}_{c}^{*}(QM\vec{N}(k), \chi)[[-2\mathrm{rk}\,N(k)]]$ while
$e^{(k, s)} \langle \gamma_{i}^{\chi} \rangle$ vanishes on it; the
situation is reversed on
$\hat{H}_{c}^{*}(QM\vec{N}(s), \chi)[[-2\mathrm{rk}\,N(s)]]$. Since the
$f^{(k, s)} \langle \gamma_{i}^{\chi^{-1}} \rangle$ commute with the
$e^{(k, s)} \langle \gamma_{i}^{\chi} \rangle$ by~\eqref{commutefsk}
(using $\eta = 0$ and $s > k$), the sequence~\eqref{Mksseq} splits as
$\mathcal{A}^{(k, s)}_{\chi}$-modules, providing the required splitting
\begin{align*}
 \hat{H}_{c}^{*}(QM\vec{N}(k), \chi)[[-2\mathrm{rk}\,N(k)]] \hookrightarrow M_{k, s} \hookrightarrow F^{k-1}M/F^{s}M.
\end{align*}
This completes the induction.
\end{proof}

\begin{proposition}
\label{prop:eta-necessary}
Assume \(\chi_0^{n+1}\neq1\). If multiplication by \(\eta\) acts
nontrivially on \(M=\hat H_c^*(QM\vec N,\chi)\), then there exist
\(1\leq k<s\leq n+1\) such that
\[
\deg\mathcal M_k=\deg\mathcal M_s,
\qquad
\chi_C\!\left((\mathcal V^{(s)}/\mathcal V^{(k)})
\otimes\mathcal M_k^{-1}\otimes\mathcal K_C\right)=1.
\]
\end{proposition}
\begin{proof}
 Suppose $\eta$ is nonzero. By~\eqref{commutinAchin}, $\eta$ commutes with $\mathcal{A}_{\chi}$, hence multiplication by $\eta$ is an $\mathcal{A}_{\chi}$-module map. Let
$k$ be the largest index in $2, \ldots, n+1$ for which
$\eta\big|_{F^{k-1}M}$ is nontrivial. By the choice of $k$ together
with~\eqref{eta0onk}, this gives a nonzero $\mathcal{A}_{\chi}$-module map
\begin{align}\label{etak-1k}
    \eta: F^{k-1}M/F^{k}M \rightarrow F^{k}M.
\end{align}
Let $s$ be the smallest index in $k+1, \ldots, n + 2$ for which the
composition of~\eqref{etak-1k} with the natural projection
$F^{k}M \rightarrow F^{k}M/F^{s}M$ is nonzero. By the choice of $s$,
this composition factors through a nonzero map
\begin{align*}
   \eta\colon F^{k-1}M/F^{k}M \rightarrow F^{s-1}M/F^{s}M,
\end{align*}
which is a morphism of $\mathcal{A}^{(k, s)}_{\chi}$-modules.
By~\eqref{flfk}, the operators $f^{(k, s)} \langle \gamma_{i}^{\chi^{-1}} \rangle$ act trivially on
\[
\hat{H}_{c}^{*}(QM\vec{N}(s), \chi)[[-2\,\mathrm{rk}\,N(s)]],
\]
so any vector in the image of some $f^{(k, s)} \langle \gamma_{i}^{\chi^{-1}} \rangle$ lies in the kernel of $\eta$. Since
\[
\hat{H}_{c}^{*}(QM\vec{N}(k), \chi)[[-2\,\mathrm{rk}\,N(k)]]
\]
is generated by $f^{(k, s)} \langle \gamma_{i}^{\chi^{-1}} \rangle$, the map $\eta$ can only be nonzero if it does not vanish on the one-dimensional generating subspace
\begin{align}\label{wtgenk}
q^{\frac{1}{2}\kappa_{k}(\vec{m})} a^{\rho_{k}(\vec{m})}  
    \overline{\mathbb{Q}}_{l} e^{(k)} \langle \gamma_{1}^{\chi} \rangle \wedge \cdots \wedge e^{(k)}\langle \gamma_{2g-2}^{\chi} \rangle.
\end{align}
Similarly, by~\eqref{elek}, the operators $e^{(k, s)} \langle \gamma_{i}^{\chi} \rangle$ act trivially on
\[
\hat{H}_{c}^{*}(QM\vec{N}(k), \chi)[[-2\,\mathrm{rk}\,N(k)]],
\]
so the image of $\eta$ must lie in the intersection of the kernels of the $e^{(k, s)} \langle \gamma_{i}^{\chi} \rangle$, $i = 1, \ldots, 2g-2$. Since
\[
\hat{H}_{c}^{*}(QM\vec{N}(s), \chi)[[-2\, \mathrm{rk}\,N(s)]]
\]
is freely generated by these $e^{(k, s)} \langle \gamma_{i}^{\chi} \rangle$,
the intersection is the one-dimensional cogenerating subspace

\begin{align}\label{wtgens}
q^{\frac{1}{2}\kappa_{s}(\vec{m})} a^{\rho_{s}(\vec{m})}  
    \overline{\mathbb{Q}}_{l} e^{(k)} \langle \gamma_{1}^{\chi} \rangle \wedge \cdots \wedge e^{(k)}\langle \gamma_{2g-2}^{\chi} \rangle.
\end{align}
Comparing the $\mathrm{Fr} \times A^{\vee}$-weights of~\eqref{wtgens}
and $\eta \times \eqref{wtgenk}$ yields the constraints
\begin{align}
    &\deg(\mathcal{M}_{s}) =  \deg(\mathcal{M}_{k}), \label{degkdegs}\\
    &\chi_C\left( (\mathcal{V}^{(s)}/\mathcal{V}^{(k)}) \otimes \mathcal{M}_{k}^{-1} \otimes \mathcal{K}_{C} \right) = 1. \label{chi(Vs/Vk)}
\end{align}
\end{proof}
\begin{remark}
    The condition~\eqref{chi(Vs/Vk)} is equivalent to
    \begin{align}\label{eqdims}
        \dim \mathrm{Attr}(s) = \dim \mathrm{Attr}(k) + 1.
    \end{align}
For $s = k+1$, this forces $g = 2$.
\end{remark}
\begin{remark}\label{rem:gammavanish}
    The same argument shows that if multiplication by $\gamma_{i}$ on $H_{c}^{*}(QM\vec{N}, \chi^{n+1})$ were nonzero, then for some $k < s$ we would have $\chi_C\left( (\mathcal{V}^{(s)}/\mathcal{V}^{(k)}) \otimes \mathcal{M}_{k}^{-1} \otimes \mathcal{K}_{C} \right) = 1/2$. Since the left-hand side is an integer, $\gamma_{i}$ must vanish.
\end{remark}
For what follows, it is convenient to pass to the dual spaces, namely, Borel--Moore homology.

The dual to the space~\eqref{wtgenk} is
\begin{align}\label{dualwtgenk}
     &\hat{H}^{BM}_{2g-2}(QM\vec{N}(k)_{2g-2 - m_{k}}, \chi^{-1})[[-2\mathrm{rk}\,N_{2g-2 - m_{k}}(k)]] =\\
     &=\hat{H}_{2g-2}^{BM}(S^{2g-2}C, (\chi^{-n-1})^{(m)})[[-2\mathrm{rk}\,N_{2g-2 - m_{k}}(k)]],
\end{align}
which is supported on the linear system
$\mathbb{P}(H^{0}(\mathcal{K}_{C})) \cong \mathbb{P}^{g-1} \overset{i}{\hookrightarrow} S^{2g-2}C$ of the divisor
class of $\mathcal{K}_{C}$, i.e.\
\begin{align*}
    i_{*}: H_{2g-2}^{BM}(\mathbb{P}^{g-1}) \otimes \bigwedge\nolimits^{2g-2}H_{1}^{BM}(C, \chi^{-n-1}) \overset{\cong}{\rightarrow} H_{2g-2}^{BM}(S^{2g-2}C, (\chi^{-n-1})^{(m)}).
\end{align*}
Hence the class in~\eqref{dualwtgenk} is supported in the fiber of the projection~\eqref{piLmap} over $\mathcal{L} = \mathcal{M}_{k} \otimes \mathcal{K}_{C}^{-1}$. Analogously, the class in
\begin{align*}
    \hat{H}^{BM}_{2g-2}(QM\vec{N}(s)_{2g-2 - m_{s}}, \chi^{-1})[[-2\mathrm{rk}\,N_{2g-2 - m_{s}}(s)]]
\end{align*}
is supported in the fiber over $\mathcal{L} = \mathcal{M}_{s} \otimes \mathcal{K}_{C}^{-1}$. Therefore $\eta$ can only be nonzero if
\begin{align}\label{mkmsiso}
    \mathcal{M}_{k} \cong \mathcal{M}_{s},
\end{align}
which refines the condition~\eqref{degkdegs}. We use the notation
\begin{align*}
QM\vec{N}_{\mathcal{M}_{k}^{-1} \otimes \mathcal{K}_{C}} = \pi^{-1}_{\mathcal{L}}(\mathcal{M}_{k} \otimes \mathcal{K}_{C}^{-1}).
\end{align*}
Consider the projection to the space of iterated framed extensions
\begin{align*}
    \pi_{\mathcal{V}^{\bullet}}\colon QM\vec{N}_{\mathcal{M}_{k}^{-1} \otimes \mathcal{K}_{C}} \rightarrow \{\mathcal{V}^{\bullet}\}.
\end{align*}
A fiber of this map over some $\mathcal{V}^{\bullet}$ is a quotient stack
\begin{align*}
    \mathbb{P}\left( H^{0}(\mathcal{V} \otimes \mathcal{L}^{-1}) \right) / \mathrm{End}(\mathcal{V}^{\bullet}).
\end{align*}
Let us consider a scheme $\widetilde{QM\vec{N}}_{\mathcal{M}_{k}^{-1} \otimes \mathcal{K}_{C}}$, which is mapped to the affine space of iterated extensions
\begin{align}\label{pifromrqmntoext}
     \pi_{\mathcal{V}^{\bullet}}\colon \widetilde{QM\vec{N}}_{\mathcal{M}_{k}^{-1} \otimes \mathcal{K}_{C}} \rightarrow \{\mathcal{V}^{\bullet}\}
\end{align}
with fibers $\mathbb{P}\left( H^{0}(\mathcal{V} \otimes \mathcal{L}^{-1}) \right)$. The scheme $\widetilde{QM\vec{N}}_{\mathcal{M}_{k}^{-1} \otimes \mathcal{K}_{C}}$ has a natural derived structure with virtual tangent bundle (compare with~\eqref{TvirQMN})
\begin{align*}
T^{\mathrm{vir}} \widetilde{QM\vec{N}}_{\mathcal{M}_{k}^{-1} \otimes \mathcal{K}_{C}}
    = \sum\limits_{l = 1}^{n+1} \mathrm{Ext}^{\bullet}(\mathcal{L},
    \mathcal{M}_{l})\otimes \mathcal{O}_{QM\vec{N}}(1)
    -\mathrm{Hom}(\mathcal{L}, \mathcal{L})
    + \sum\limits_{l < r}\mathrm{Ext}^{1}(\mathcal{M}_{r},
    \mathcal{M}_{l}).
\end{align*}

The pullback under the natural projection
\begin{align}\label{rig}   \widetilde{QM\vec{N}}_{\mathcal{M}_{k}^{-1} \otimes \mathcal{K}_{C}} \rightarrow QM\vec{N}_{\mathcal{M}_{k}^{-1} \otimes \mathcal{K}_{C}}
\end{align}
identifies Borel--Moore homology with shift $[[2\sum\limits_{l < r} h^{0}(\mathcal{M}_{l}\otimes \mathcal{M}_{r})]]$, and the action of $\eta$ lifts to the homology of the rigidification.
Let $\mathrm{Attr}'(l)$ denote the preimage under~\eqref{rig} of the intersection of $\mathrm{Attr}(l)$ and $\pi^{-1}_{\mathcal{L}}\left( \mathcal{M}_{k} \otimes \mathcal{K}_{C}^{-1}\right)$.
\begin{remark}
The condition~\eqref{eqdims} is equivalent to
\begin{align}\label{eqprimedims}
    \mathrm{dim}\, \mathrm{Attr}'(s) = \mathrm{dim}\, \mathrm{Attr}'(k) + 1.
\end{align}
\end{remark}
We will need the following lemma.
\begin{lemma}\label{imageofcl}
    For any $l = 1, \ldots, n+1$ the images $\pi_{\mathcal{V}^{\bullet}} \left( \mathrm{Attr}'(l) \right)$ and $\pi_{\mathcal{V}^{\bullet}} \left( \overline{\mathrm{Attr}'(l)} \right)$ under the map~\eqref{pifromrqmntoext} coincide.
\end{lemma}
\begin{proof}
    Consider a long exact sequence
\begin{align*}
  \cdots \rightarrow H^{0}(\mathcal{V}^{(l)}\otimes \mathcal{M}_{k}^{-1} \otimes \mathcal{K}_{C}) \overset{\mathrm{pr}_{l}}{\longrightarrow} H^{0}(\mathcal{M}_{l}\otimes \mathcal{M}_{k}^{-1} \otimes \mathcal{K}_{C}) \overset{\beta_{l}}{\rightarrow} H^{1}(\mathcal{V}^{(l-1)}\otimes \mathcal{M}_{k}^{-1} \otimes \mathcal{K}_{C}) \rightarrow \cdots,
\end{align*}    
where $\beta_{l}$ is multiplication by the corresponding extension class in $\mathrm{Ext}^{1}(\mathcal{M}_{l}, \mathcal{V}^{(l-1)})$. An iterated extension $\mathcal{V}^{\bullet}$ is in the image of $\mathrm{Attr}'(l)$ if and only if $\mathrm{Im}(\mathrm{pr}_{l})$ is nonzero, i.e.\ the set $\pi_{\mathcal{V}^{\bullet}}(\mathrm{Attr}'(l))$ is cut out in $\{\mathcal{V}^{\bullet}\}$ by the closed condition
\begin{align*}
    \pi_{\mathcal{V}^{\bullet}}(\mathrm{Attr}'(l)) = \{\mathrm{Ker}(\beta_{l}) \neq 0\} \subset \{ \mathcal{V}^{\bullet}\}.
\end{align*}
Since $\pi_{\mathcal{V}^{\bullet}}(\mathrm{Attr}'(l))$ is closed, it coincides with the image of the closure. 
\end{proof}
After dualizing~\eqref{etak-1k}, passing to the rigidification, and reducing common shifts, we obtain a map
\begin{align}\label{etaBMh}
    \eta\colon H^{BM}_{\mathrm{top}}(\mathrm{Attr}'(s)) = \overline{\mathbb{Q}}_{l}[\mathrm{Attr}'(s)] \rightarrow \overline{\mathbb{Q}}_{l}[\mathrm{Attr}'(k)] = H^{BM}_{\mathrm{top}}(\mathrm{Attr}'(k)).
\end{align}
\begin{lemma}\label{lem:eqproject}
    If the map~\eqref{etaBMh} is nonzero, then $\pi_{\mathcal{V}^{\bullet}} \left( \mathrm{Attr}'(k) \right) = \pi_{\mathcal{V}^{\bullet}} \left( \mathrm{Attr}'(s) \right)$.
\end{lemma}
\begin{proof}
    First, we prove the inclusion $\subset$. If this is not the case, then
\begin{align*}
    \pi_{\mathcal{V}^{\bullet}} \left( \mathrm{Attr}'(k) \right) \centernot\subset \pi_{\mathcal{V}^{\bullet}} \left( \overline{\mathrm{Attr}'(s)} \right),
\end{align*}
and consequently
\begin{align*}
    Z \coloneqq \pi_{\mathcal{V}^{\bullet}} \left( \mathrm{Attr}'(k) \right) \cap \pi_{\mathcal{V}^{\bullet}} \left( \overline{\mathrm{Attr}'(s)} \right) \varsubsetneq  \pi_{\mathcal{V}^{\bullet}} \left( \mathrm{Attr}'(k) \right)
\end{align*}
is a proper closed subset. Thus,
\begin{align*}
    U = \mathrm{Attr}'(k) \setminus \pi_{\mathcal{V}^{\bullet}}^{-1}(Z)
\end{align*}
is an open dense subset of $\mathrm{Attr}'(k)$.

To understand how $\eta$ acts in~\eqref{etaBMh}, we need to take a representative in $H^{BM}_{*}\left(\bigcup\limits_{l = 1}^{s}\mathrm{Attr}'(l)\right) \subset H^{BM}_{*}\left(\widetilde{QM\vec{N}}_{\mathcal{M}_{k}^{-1} \otimes \mathcal{K}_{C}} \right)$, which under the pullback 
\begin{align*}
    j^{*}_{s}\colon H_{*}^{BM}\left( \bigcup\limits_{l = 1}^{s}\mathrm{Attr}'(l) \right) \twoheadrightarrow H^{BM}_{*}\left( \mathrm{Attr}'(s)\right)
\end{align*}
by open immersion
\begin{align*}
    j_{s}\colon \mathrm{Attr}'(s) \hookrightarrow \bigcup\limits_{l = 1}^{s}\mathrm{Attr}'(l)
\end{align*}
restricts to $[\mathrm{Attr}'(s)]$. There is a canonical representative given by $[\overline{\mathrm{Attr}'(s)}]$. Since the map~\eqref{etaBMh} is well defined, there is a class
\begin{align*}
    \alpha \in H^{BM}_{*}\left( \bigcup\limits_{l=1}^{k}\mathrm{Attr}'(l)\right)
\end{align*}
which under the pushforward 
\begin{align*}
    i_{*}\colon H_{*}^{BM}\left( \bigcup\limits_{l = 1}^{k}\mathrm{Attr}'(l) \right) \hookrightarrow H_{*}^{BM}\left( \bigcup\limits_{l = 1}^{s}\mathrm{Attr}'(l) \right)
\end{align*}
by the closed immersion
\begin{align*}
    i\colon \bigcup\limits_{l = 1}^{k}\mathrm{Attr}'(l) \hookrightarrow \bigcup\limits_{l = 1}^{s}\mathrm{Attr}'(l)
\end{align*}
is mapped to the cap product:
\begin{align}\label{ipushalpha}
    i_{*}(\alpha) = \eta \cap [\overline{\mathrm{Attr}'(s)}],
\end{align}
and $\eta$ applied to $[\mathrm{Attr}'(s)]$ is given by
\begin{align}\label{jpullalpha}
    \eta [\mathrm{Attr}'(s)] = j_{k}^{*}(\alpha)
\end{align}
for
\begin{align*}
    j_{k}\colon \mathrm{Attr}'(k) \hookrightarrow \bigcup\limits_{l = 1}^{k}\mathrm{Attr}'(l).
\end{align*}
From~\eqref{ipushalpha} and~\eqref{jpullalpha} we see that the support of $\eta [\mathrm{Attr}'(s)]$ is contained in
\begin{align*}
    \mathrm{Attr}'(k) \cap \overline{\mathrm{Attr}'(s)} \subset \pi_{\mathcal{V}^{\bullet}}^{-1}(Z),
\end{align*}
so the restriction of $\eta [\mathrm{Attr}'(s)]$ to $U$ is zero, but the restriction of $[\mathrm{Attr}'(k)]$ to $U$ is equal to $[U] \neq 0$, so $\eta$ must vanish.

Now we prove the inclusion $\supset$. If this is not the case, then by Lemma~\ref{imageofcl} and the inclusion $\subset$ proved above,
\begin{align*}
\pi_{\mathcal{V}^{\bullet}}\left(\overline{\mathrm{Attr}'(k)}\right) \varsubsetneq  \pi_{\mathcal{V}^{\bullet}} \left( \overline{\mathrm{Attr}'(s)} \right)
\end{align*}
is a proper closed subset. Then 
\begin{align}\label{inc1}
    \pi_{\mathcal{V}^{\bullet}}^{-1} \left( \pi_{\mathcal{V}^{\bullet}}\left(\overline{\mathrm{Attr}'(k)}\right) \right) \cap \overline{\mathrm{Attr}'(s)} \varsubsetneq \overline{\mathrm{Attr}'(s)}
\end{align}
is a proper closed subset. Also, it follows from the inclusion $\subset$ that for $\mathcal{V}^{\bullet} \in \pi_{\mathcal{V}^{\bullet}}\left( \overline{\mathrm{Attr}'(k)}\right) = \pi_{\mathcal{V}^{\bullet}}\left( \mathrm{Attr}'(k)\right)$
\begin{align*}
    H^{0}(\mathcal{V}^{(k)}\otimes \mathcal{M}_{k}^{-1}\otimes\mathcal{K}_{C}) \varsubsetneq  H^{0}(\mathcal{V}^{(s)}\otimes \mathcal{M}_{k}^{-1}\otimes\mathcal{K}_{C}),
\end{align*}
hence
\begin{align}\label{inc2}
    \overline{\mathrm{Attr}'(k)} \varsubsetneq \pi_{\mathcal{V}^{\bullet}}^{-1} \left( \pi_{\mathcal{V}^{\bullet}}\left(\overline{\mathrm{Attr}'(k)}\right) \right) \cap \overline{\mathrm{Attr}'(s)}
\end{align}
is a proper closed subset. Comparing~\eqref{inc1} and~\eqref{inc2}, we see that
\begin{align*}
  \dim\mathrm{Attr}'(s) \geq \dim\mathrm{Attr}'(k) + 2,
\end{align*}
which contradicts~\eqref{eqprimedims}.
\end{proof}
\begin{corollary}\label{vanishingh1h0}
    If the map~\eqref{etaBMh} is nonzero, then $h^{1}(\mathcal{M}_{l} \otimes \mathcal{M}_{k}^{-1}\otimes\mathcal{K}_{C}) = 0$ for $l < k$, and \\
    $h^{0}(\mathcal{M}_{l} \otimes \mathcal{M}_{k}^{-1}\otimes\mathcal{K}_{C}) = 0$ for $k < l < s$. 
\end{corollary}
\begin{proof}
Consider a vector bundle
\begin{align*}
    \tilde{\mathcal{V}} = \mathcal{M}_{k} \oplus \cdots \oplus \mathcal{M}_{s-1}.
\end{align*}
The rank of zero cohomology is given by
\begin{align*}
    h^{0}(\tilde{\mathcal{V}}\otimes \mathcal{M}_{k}^{-1}\otimes \mathcal{K}_{C}) = h^{0}(\mathcal{K}_{C}) + \sum\limits_{l = k+1}^{s-1} h^{0}(\mathcal{M}_{l} \otimes \mathcal{M}_{k}^{-1}\otimes\mathcal{K}_{C})= g + \sum\limits_{l = k+1}^{s-1} h^{0}(\mathcal{M}_{l} \otimes \mathcal{M}_{k}^{-1}\otimes\mathcal{K}_{C}).
\end{align*}
Combined with~\eqref{chi(Vs/Vk)}, this gives the rank of the first cohomology
\begin{align*}
    &h^{1}(\tilde{\mathcal{V}} \otimes \mathcal{M}_{k}^{-1}\otimes\mathcal{K}_{C}) = h^{0}(\tilde{\mathcal{V}}\otimes \mathcal{M}_{k}^{-1}\otimes \mathcal{K}_{C}) - \chi_C(\tilde{\mathcal{V}}) = g - 1 + \sum\limits_{l = k+1}^{s-1} h^{0}(\mathcal{M}_{l} \otimes \mathcal{M}_{k}^{-1}\otimes\mathcal{K}_{C})
\end{align*}
Then for a vector bundle
\begin{align*}
    \tilde{\tilde{\mathcal{V}}} = \mathcal{M}_{1} \oplus \cdots \oplus \mathcal{M}_{s-1}
\end{align*}
we have the following expression for the rank of the first cohomology:
\begin{align*}
   h^{1}(\tilde{\tilde{\mathcal{V}}} \otimes \mathcal{M}_{k}^{-1}\otimes\mathcal{K}_{C}) &= h^{1}(\tilde{\mathcal{V}} \otimes \mathcal{M}_{k}^{-1}\otimes\mathcal{K}_{C}) + \sum\limits_{l = 1}^{k-1}h^{1}(\mathcal{M}_{l} \otimes \mathcal{M}_{k}^{-1}\otimes\mathcal{K}_{C}) =\\
   &=g - 1 + \sum\limits_{l = 1}^{k-1}h^{1}(\mathcal{M}_{l} \otimes \mathcal{M}_{k}^{-1}\otimes\mathcal{K}_{C}) + \sum\limits_{l = k+1}^{s-1} h^{0}(\mathcal{M}_{l} \otimes \mathcal{M}_{k}^{-1}\otimes\mathcal{K}_{C}).
\end{align*}
For every $\beta \in \mathrm{Ext}^{1}(\mathcal{M}_{s}, \tilde{\tilde{\mathcal{V}}})$ the boundary map
\begin{align*}
    \beta\colon H^{0}(\mathcal{K}_{C}) = H^{0}(\mathcal{M}_{s}\otimes\mathcal{M}_{k}^{-1} \otimes \mathcal{K}_{C}) \rightarrow H^{1}(\tilde{\tilde{\mathcal{V}}}\otimes\mathcal{M}_{k}^{-1}\otimes\mathcal{K}_{C})
\end{align*}
has dimension of the domain $h^{0}(\mathcal{K}_{C}) = g$. If any of the $h^{0}$'s or $h^{1}$'s in the sum is nonzero, then the dimension of the codomain of $\beta$ is at least $g$, so for generic $\beta$ this map will be injective. Such a $\beta$ corresponds to an extension
\begin{align*}
    0 \rightarrow \tilde{\tilde{\mathcal{V}}} \rightarrow \mathcal{V}^{(s)} \rightarrow \mathcal{M}_{s} \rightarrow 0,
\end{align*}
which fits in an obvious way into an iterated filtration $\mathcal{V}^{\bullet} \in \pi_{\mathcal{V}^{\bullet}}(\mathrm{Attr}'(k)) \setminus \pi_{\mathcal{V}^{\bullet}}(\mathrm{Attr}'(s))$, which contradicts Lemma~\ref{lem:eqproject}.
\end{proof}
\begin{remark}
    The conditions from Corollary~\ref{vanishingh1h0} and~\eqref{chi(Vs/Vk)} give
\begin{align*}
    \sum\limits_{k < l < s} h^{1}\left( \mathcal{M}_{l} \otimes \mathcal{M}_{k}^{-1} \otimes\mathcal{K}_{C}\right) = g-2.
\end{align*}
In particular, this implies $g \geq 2$.

\end{remark}
\begin{corollary}\label{h1Vs-1}
    For any iterated extension $\mathcal{V}^{\bullet}$
\begin{align*}
    h^{1}(\mathcal{V}^{(s-1)}\otimes \mathcal{M}_{k}^{-1}\otimes\mathcal{K}_{C}) = g-1.
\end{align*}
\end{corollary}
\begin{proof}
Use a long exact sequence
\begin{align*}
    \rightarrow H^{1}(\mathcal{V}^{(k-1)}\otimes \mathcal{M}_{k}^{-1}\otimes\mathcal{K}_{C}) \rightarrow H^{1}(\mathcal{V}^{(s-1)}\otimes \mathcal{M}_{k}^{-1}\otimes\mathcal{K}_{C}) \rightarrow H^{1}(\mathcal{V}^{(s-1)}/\mathcal{V}^{(k-1)}\otimes \mathcal{M}_{k}^{-1}\otimes\mathcal{K}_{C}) \rightarrow 0
\end{align*}
Since $H^{1}(\mathcal{M}_{l}\otimes \mathcal{M}_{k}^{-1}\otimes\mathcal{K}_{C}) = 0$ for $l < k$, the first term vanishes, and we have
\begin{align*}
    h^{1}(\mathcal{V}^{(s-1)}\otimes \mathcal{M}_{k}^{-1}\otimes\mathcal{K}_{C}) = h^{1}(\mathcal{V}^{(s-1)}/\mathcal{V}^{(k-1)}\otimes \mathcal{M}_{k}^{-1}\otimes\mathcal{K}_{C}).
\end{align*}
The latter can be computed with the help of~\eqref{chi(Vs/Vk)}:
\begin{align*}
   h^{1}(\mathcal{V}^{(s-1)}/\mathcal{V}^{(k-1)}\otimes \mathcal{M}_{k}^{-1}\otimes\mathcal{K}_{C}) = h^{0}(\mathcal{V}^{(s-1)}/\mathcal{V}^{(k-1)}\otimes \mathcal{M}_{k}^{-1}\otimes\mathcal{K}_{C}) - 1 = g-1.
\end{align*}
The last equality comes from the vanishing of $H^{0}(\mathcal{M}_{l}\otimes \mathcal{M}_{k}^{-1}\otimes\mathcal{K}_{C})$ for $k < l < s$.
\end{proof}
From Corollary~\ref{h1Vs-1} we conclude that the boundary map
\begin{align*}
    H^{0}(\mathcal{M}_{s} \otimes \mathcal{M}_{k}^{-1}\otimes\mathcal{K}_{C}) \rightarrow H^{1}(\mathcal{V}^{(s-1)}\otimes \mathcal{M}_{k}^{-1}\otimes\mathcal{K}_{C})
\end{align*}
is surjective for any $\mathcal{V}^{\bullet}$ (so $\pi_{\mathcal{V}^{\bullet}}(\mathrm{Attr}'(s)) = \{\mathcal{V}^{\bullet}\}$) and generically has one-dimensional kernel, hence for generic $\mathcal{V}^{\bullet}$ we have
\begin{align*}
    h^{0}(\mathcal{V}^{(s)} \otimes \mathcal{M}_{k}^{-1}\otimes\mathcal{K}_{C}) = h^{0}(\mathcal{V}^{(s-1)} \otimes \mathcal{M}_{k}^{-1}\otimes\mathcal{K}_{C}) + 1 = h^{0}(\mathcal{V}^{(k)} \otimes \mathcal{M}_{k}^{-1}\otimes\mathcal{K}_{C}) + 1.
\end{align*}

We can now prove the following.
\begin{proposition}\label{pr:eta-sufficient}
    The conditions from Corollary~\ref{vanishingh1h0} with~\eqref{chi(Vs/Vk)} and~\eqref{mkmsiso} are sufficient for the map~\eqref{etaBMh} to be non-vanishing.
\end{proposition}

\begin{proposition}\label{pr:eta-cap}
Assume \((k,s)\) is exceptional. Then
\[
\eta\cap[\operatorname{Attr}'(s)]
=
[\operatorname{Attr}'(k)]
\]
in Borel--Moore homology.
\end{proposition}

\begin{proof}
See Appendix~\ref{app:BM-eta}.
\end{proof}

We can now prove the main result, Theorem~\ref{thm:main}.

\begin{proof}[Proof of Theorem~\ref{thm:main}]
For $\chi_{0}^{n+1}=1$, the isomorphism is Proposition~\ref{pr:comparison-triv}.
For $\chi_{0}^{n+1}\neq 1$, by Proposition~\ref{pr:splitting} and the
decomposition $H^{*}_{L}(X^{\vee}, \hat{\mathcal{O}}_{\mathrm{vir}}\otimes \mathcal{O}(\vec{\mathcal{M}})) \cong \bigoplus_{k} H^{*}_{L^{k}}$,
it suffices to show that $\eta = 0$ on $M$. The preceding analysis
exhibits~\eqref{exceptionalcond} as the only possible
obstruction.
\end{proof}

\appendix

\section{Relations in the algebra of correspondences}\label{relations}

\subsection{Supercommutator relations in $\mathcal{A}_{\chi}$}\label{sub:supcominA}
Substituting $\alpha, \beta \in \{\mathbf{p}, \gamma_{i}, 1\}$ in~\eqref{compsup} and
using the explicit formula for $\psi$ of Proposition~\ref{pr:psi}
gives the relations in $\mathcal{A}_{\chi}$ for the trivial $\chi_{0}^{n+1}$:
\begin{align}
    &[e\langle \mathbf{p} \rangle, f \langle \mathbf{p} \rangle] = 0, \label{commutef0}\\
    &[e\langle \mathbf{p} \rangle, f \langle \gamma_{i} \rangle] = 0,\\
    &[e\langle \gamma_{i} \rangle, f \langle \mathbf{p} \rangle] = 0,\\
    &[e\langle \gamma_{i} \rangle, f \langle \gamma_{j}^{\vee} \rangle]
    = \delta_{ij} (n+1)\eta^{n},\label{[egfg]}\\
    &[e\langle 1 \rangle, f \langle 1 \rangle]
    = n\,\eta^{n-1}\chi_C(\mathcal{V}\otimes \mathcal{L}^{-1})
    - (n+1)n(n-1)\eta^{n-2}\Theta,\label{[e1f1]}\\
    &[e\langle 1 \rangle, f \langle \mathbf{p} \rangle]
    = -[f\langle 1 \rangle, e \langle \mathbf{p} \rangle]
    = (n+1)\eta^{n},\label{[e1fp]}\\
    &[e\langle 1 \rangle, f \langle \gamma_{i} \rangle]
    = -[f\langle 1 \rangle, e \langle \gamma_{i} \rangle]
    = n(n+1)\eta^{n-1}\gamma_{i}.\label{[e1fg]}
\end{align}

By Lemma~\ref{lem:pullbacks}, the operators on $\hat{H}^{*}_{c}(QM\vec{N})$ of multiplication by $\eta, \gamma_{i} \in H^{*}(QM\vec{N})$ have the following commutators:
\begin{align}
     [e\langle \mathbf{p} \rangle, \eta]
    &= 0,
   &[f\langle \mathbf{p} \rangle, \eta]
    &= 0,\label{[epeta]}\\
       [e\langle \gamma_{i} \rangle, \eta]
    &= 0,
    &[f\langle \gamma_{i} \rangle, \eta]
    &= 0,\label{[egammaeta]}\\
    [e\langle 1 \rangle, \eta]
    &= -e\langle \mathbf{p} \rangle,
    &[f\langle 1 \rangle, \eta]
    &= f\langle \mathbf{p} \rangle,\label{[e1eta]}\\
    [e\langle 1 \rangle, \gamma_{i}]
    &= -e \langle \gamma_{i} \rangle,
    &[f\langle 1 \rangle, \gamma_{i}]
    &= f \langle \gamma_{i} \rangle,\label{[e1g]}\\
    [e\langle \gamma_{i} \rangle, \gamma_{j}^{\vee}]
    &= -\delta_{ij} e\langle \mathbf{p} \rangle,
    &[f\langle \gamma_{i} \rangle, \gamma_{j}^{\vee}]
    &= \delta_{ij} f \langle \mathbf{p} \rangle,\label{[egg]}\\
    [e\langle \mathbf{p} \rangle, \gamma_{i}]
    &= 0,
    &[f\langle \mathbf{p} \rangle, \gamma_{i}]
    &= 0,\label{[epg]}
\end{align}

Commutators with $\chi_C(\mathcal{V}\otimes \mathcal{L}^{-1})$ are given by
\begin{align}
      [e\langle \alpha \rangle, \chi_C(\mathcal{V}\otimes \mathcal{L}^{-1})]
    &= -(n+1)e\langle \alpha \rangle,
   &[f\langle \alpha \rangle, \chi_C(\mathcal{V}\otimes \mathcal{L}^{-1})]
    &= (n+1)f\langle \alpha \rangle, \;\;\; \alpha \in H^{1}_{c}(C, \chi). \label{[epdeg]}
\end{align}
Using the identity
\begin{align*}
    H_{c}^{*}(C, \overline{\mathbb{Q}}_{l}) &\rightarrow H_{c}^{*}(C \times C, \chi_{0}^{n+1} \boxtimes \chi_{0}^{-n-1}),\\
    \Delta_{C*}(1) &= -\sum\limits_{i = 1}^{2g-2}\gamma_{i}^{\chi} \otimes \gamma_{i}^{\chi^{-1}},
\end{align*}
we analogously compute supercommutators in $\mathcal{A}_\chi$ when $\chi_{0}^{n+1} \neq 1$:
\begin{equation}\label{commutinAchin}
    [e\langle \gamma_{i}^{\chi} \rangle, f\langle \gamma_{j}^{\chi^{-1}} \rangle]
    = \delta_{ij}(n+1)\eta^{n},\qquad
    [e\langle \gamma_{i}^{\chi} \rangle, \eta]
    = [f\langle \gamma_{i}^{\chi^{-1}} \rangle, \eta] =[e\langle \gamma_{i}^{\chi} \rangle, \gamma_{j}]
    = [f\langle \gamma_{i}^{\chi^{-1}} \rangle, \gamma_{j}] = 0. 
\end{equation}
\subsection{Casimir relations in $\mathcal{A}_{\chi}$}
By Proposition~\ref{pr:anticomm}, and the identities
\begin{equation}\label{deltapush}
\begin{aligned}
    \Delta_{C*}(\mathbf{p}) &= \mathbf{p} \otimes \mathbf{p},\\
    \Delta_{C*}(\gamma_{i}) &= \mathbf{p} \otimes \gamma_{i} + \gamma_{i} \otimes \mathbf{p},\\
    \Delta_{C*}(1) &= \mathbf{p} \otimes 1 + 1 \otimes \mathbf{p} - \sum_{i=1}^{2g} \gamma_{i} \otimes \gamma_{i}^{\vee},
\end{aligned}
\end{equation}
the operators in $\mathcal{A}_{\chi}$ for the trivial $\chi_{0}^{n+1}$ satisfy
\begin{align}
    &e\langle \mathbf{p} \rangle f \langle \mathbf{p} \rangle + \eta^{n+1} = 0, \label{cas0}\\
    &e\langle \mathbf{p} \rangle f\langle \gamma_{i} \rangle
    + f\langle \mathbf{p} \rangle e\langle \gamma_{i} \rangle
    + (n+1)\eta^{n}\gamma_{i} = 0,\label{cas1}\\
    &\begin{aligned}
    &f\langle \mathbf{p} \rangle e \langle 1 \rangle
    + e \langle \mathbf{p} \rangle f\langle 1 \rangle
    - \sum_{i=1}^{g}\bigl( e \langle \gamma_{i} \rangle f \langle \gamma_{i}^{\vee} \rangle
    + f \langle \gamma_{i} \rangle e \langle \gamma_{i}^{\vee} \rangle \bigr)\\
    &\quad + \eta^{n}\chi_C(\mathcal{V} \otimes \mathcal{L}^{-1})
    - n(n+1) \eta^{n-1} \Theta
    \end{aligned}
    = 0\label{cas2}.
\end{align}
For the nontrivial $\chi_{0}^{n+1}$ we have
\begin{align}
    &\eta^{n+1} = 0, \qquad \eta^{n}\gamma_{i} = 0, \qquad \sum\limits_{i = 1}^{2g-2}f\langle \gamma_{i}^{\chi^{-1}}\rangle e\langle\gamma_{i}^{\chi}\rangle + \eta^{n}\chi_C(\mathcal{V}\otimes\mathcal{K}_{C}^{-1/2}\otimes\mathcal{L}^{-1})-n(n+1)\eta^{n-1}\Theta=0. \nonumber
\end{align}

\subsection{Supercommutator relations in $\mathcal{A}^{(k)}_{\chi}$}\label{sub:supcominAk}
Analogously to Subsection~\ref{sub:supcominA} we get supercommutators for descendants in $\mathcal{A}_{\chi}^{(k)}$ for the trivial $\chi_{0}^{n+1}$:
\begin{align}
    &[e^{(k)}\langle \mathbf{p} \rangle , f^{(k)}\langle \mathbf{p} \rangle ] = 0, \label{comk2}\\
     &[e^{(k)}\langle 1 \rangle , f^{(k)}\langle 1 \rangle ] = 0, \label{comk1}\\
    &[e^{(k)}\langle 1 \rangle , f^{(k)}\langle \mathbf{p} \rangle ] = -[f^{(k)}\langle 1 \rangle , e^{(k)}\langle \mathbf{p} \rangle ] =  1,\label{commk1rest}\\
    &[e^{(k)}\langle 1 \rangle , \eta ] = - e^{(k)}\langle \mathbf{p} \rangle, \nonumber \\
    &[f^{(k)}\langle 1 \rangle , \eta ] =  f^{(k)}\langle \mathbf{p} \rangle, \nonumber\\
    &[e^{(k)}\langle 1 \rangle , \gamma_{i}] = - e^{(k)}\langle \gamma_i \rangle, \nonumber\\
    &[f^{(k)}\langle 1 \rangle , \gamma_{i}] =  f^{(k)}\langle \gamma_i \rangle, \nonumber \\
    &[e^{(k)}\langle \gamma_{i} \rangle , f^{(k)}\langle \gamma_{j}^{\vee} \rangle ] = \delta_{ij}, \label{comk11}\\
    &[e^{(k)}\langle \gamma_{i} \rangle , \gamma_{j}^{\vee} ] = - \delta_{ij} e^{(k)}\langle \mathbf{p} \rangle, \nonumber \\
    &[f^{(k)}\langle \gamma_{i} \rangle , \gamma_{j}^{\vee} ] = \delta_{ij} f^{(k)}\langle \mathbf{p} \rangle. \nonumber
\end{align}
For $\chi_{0}^{n+1} \neq 1$, the only nontrivial supercommutator is
\begin{align}
    [e^{(k)}\langle \gamma_{i}^{\chi} \rangle, f^{(k)}\langle \gamma_{j}^{\chi^{-1}} \rangle] = \delta_{ij}.
\end{align}
\subsection{Casimir relations in $\mathcal{A}^{(k)}_{\chi}$}\label{sub:CasrelinAk}
Using~\eqref{deltapush} we compute the Casimir relations in $\mathcal{A}_{\chi}^{(k)}$ for the trivial~$\chi_{0}^{n+1}$:
\begin{align}
    &e^{(k)}\langle \mathbf{p} \rangle f^{(k)} \langle \mathbf{p} \rangle + \eta = 0 \label{cask2}\\
    &e^{(k)}\langle \mathbf{p} \rangle f^{(k)}\langle \gamma_{i} \rangle + f^{(k)}\langle \mathbf{p} \rangle e^{(k)}\langle \gamma_{i} \rangle + \gamma_{i} = 0 \label{cask1} \\
    &f^{(k)}\langle \mathbf{p} \rangle e^{(k)} \langle 1 \rangle + e^{(k)} \langle \mathbf{p} \rangle f^{(k)}\langle 1 \rangle - \sum\limits_{i = 1}^{g} \left( e^{(k)} \langle \gamma_{i} \rangle f^{(k)} \langle \gamma_{i}^{\vee} \rangle + f^{(k)} \langle \gamma_{i} \rangle e^{(k)} \langle \gamma_{i}^{\vee} \rangle \right) + \nonumber \\
    &+\chi_C(\mathcal{M}_{k} \otimes \mathcal{L}^{-1})  = 0 \label{cask0}
\end{align}
Analogously, for $\chi_{0}^{n+1} \neq 1$ the operators in $\mathcal{A}_{\chi}^{(k)}$ satisfy
\begin{align}
    & \eta = 0, \qquad \gamma_{i} = 0, \qquad \sum_{i = 1}^{2g-2} f^{(k)}\langle \gamma_{i}^{\chi^{-1}} \rangle e^{(k)}\langle \gamma_{i}^{\chi} \rangle + \chi_C(\mathcal{M}_{k} \otimes \mathcal{K}_{C}^{-1/2} \otimes \mathcal{L}^{-1}) = 0. \label{eta0onk}
\end{align}

\section{Degree zero descendants as differential operators}\label{Degreezerodescendants}
In this section we describe the full algebra of correspondences. The degree-zero descendants will be identified with the global vector fields on the resolution $X^{\vee}$ preserving the exceptional divisor. This will produce a natural module for the algebra $\mathcal{A}$.
\subsection{The degree-zero descendants in $\mathcal{A}_{\chi}$.}
Consider the decomposition~\eqref{HAHB} of $H^{1}_{c}(C)$ into two Lagrangian subspaces
\begin{align*}
    H^{1}_{c}(C) = H_{A} \oplus H_{B}.
\end{align*}
This induces a Lagrangian submodule of $\mathcal{E}$,
\begin{align*}
    \mathcal{E}_{A} =  q^{-1}\Omega_{R} \otimes H_{A}.
\end{align*}

Let us introduce a super commutative algebra
\begin{equation*}
    \mathcal{R} = \bigwedge\nolimits^{\bullet}_{R} \mathcal{E}_{A},
\end{equation*}
and denote the corresponding super variety $\mathcal{X}_{0}^{\vee}$.
It has a natural super Poisson structure, extending the Poisson structure~\eqref{Poissonbraket} on $R$: 
\begin{align}
    &\pi \colon q^{-1} \otimes \mathcal{R}\otimes \mathcal{R}   \rightarrow \mathcal{R} \\
    &\{ x , \theta_{i}^{y} \} = - \{ y , \theta_{i}^{x} \} = n(n+1)\eta^{n-1}\gamma_i , \qquad \{x , \gamma_i \} = \theta_{i}^{x} \qquad \{y , \gamma_{i} \} = -\theta_i^{y},\\
    &\{ x , \theta_{i}^{x}\} = \{ y , \theta_{i}^{y}\} = \{ \theta_{i}^{x} , \theta_{j}^{x}\} = \{ \theta_{i}^{x} , \theta_{j}^{x}\} = \{ \theta_{i}^{y} , \theta_{j}^{y}\} = 0, \;\; i, j = 1, ..., g.
\end{align}
For $r \in R$ let 
\begin{align*}
    \xi_{r} = - \{r , \cdot \} \in \mathrm{Der}(R)
\end{align*}
be the corresponding Poisson derivation of $R$. Let us denote by
\begin{align*}
    \mathrm{Der}_{\pi}(R) \subset \mathrm{Der}(R)
\end{align*}
the submodule of Poisson derivations.
The Lie derivative extends $\xi_{r}$ to an even derivation of $\mathcal{R}$:
\begin{equation*}
    \mathrm{Lie}_{\xi_{r}} \in \mathrm{Der}(\mathcal{R}).
\end{equation*}
With the help of the pairing~\eqref{clifpair}, the elements $r \in R$ and $\gamma_{i}^{\vee} \in H_{B},\;  i = 1, ..., g$ produce a map
\begin{equation*}
    \iota_{\xi_{r}} \otimes \langle \gamma_{i}^{\vee}, \cdot \rangle \colon \mathcal{E}_{A} \rightarrow R,
\end{equation*}
which gives rise to an odd derivation
\begin{equation*}
    \theta_{i}^{\vee r} \in \mathrm{Der}(\mathcal{R}).
\end{equation*}
The elements $\mathrm{Lie}_{\xi_{r}}, \theta_{i}^{\vee r}, \;\; i = 1, ..., g, \;\; r = x, y, \eta$ generate a (proper) super Lie subalgebra
\begin{align}
    \mathrm{Der}_{\pi}(\mathcal{R}) \subset \mathrm{Der}(\mathcal{R}).
\end{align}
Let $\mathcal{D}(\mathcal{R})$ denote the algebra of differential operators on $\mathcal{R}$, and let
\begin{align}
  \mathcal{D}_{\pi}(\mathcal{R}) \subset \mathcal{D}(\mathcal{R})
\end{align}
denote the subalgebra generated by $\mathrm{Der}_{\pi}(\mathcal{R})$. A possible construction of $\mathcal{D}_{\pi}$ goes as follows. Let $\mathfrak{a}_{\pi}$ be a super Lie algebroid
\begin{align}\label{Liealgebroid}
    \mathfrak{a}_{\pi} = \mathcal{R} \oplus \mathrm{Der}_{\pi}(\mathcal{R}),
\end{align}
where the bracket between the summands is given by the natural pairing, and is trivial when restricted to $\mathcal{R}$. The natural projection and inclusion define the symbol map and the connection, respectively,
\begin{align}\label{symcon}
    \sigma: \mathfrak{a}_{\pi} \rightarrow \mathrm{Der}_{\pi}(\mathcal{R}) \qquad \nabla: \mathrm{Der}_{\pi}(\mathcal{R}) \rightarrow \mathfrak{a}_{\pi}.
\end{align}

The algebra $\mathcal{D}_{\pi}(\mathcal{R})$ is defined as
\begin{equation}\label{eq:UR}
    \mathcal{D}_{\pi}(\mathcal{R}) = U_{\mathcal{R}}(\mathfrak{a}_{\pi}) \coloneqq U_{\overline{\mathbb{Q}}_{l}}(\mathfrak{a}_{\pi})/\left(1_{\mathcal{R}} - 1_{U_{\overline{\mathbb{Q}}_{l}}(\mathfrak{a}_{\pi})}, r \otimes D - r \cdot D, \;\; r \in \mathcal{R} \subset \mathfrak{a}_{\pi}, D \in \mathfrak{a}_{\pi} \right).
\end{equation}
Let $\mathcal{A}^{\geq 1.5} \subset \mathcal{A}^{\geq1}$ be a subalgebra of $\mathcal{A}$, generated by $\mathcal{A}^{2}$ and $\theta_{i}^{x}$, $\theta_{i}^{y}, \gamma_i \;\; i = 1, ..., g$.
It follows from Proposition~\ref{ClcongA} that $\Phi$ maps $\mathcal{R} \subset \mathcal{D}_{\pi}(\mathcal{R})$ isomorphically onto $\mathcal{A}^{\geq 1.5}$, and in particular, identifies $R \subset \mathcal{D}_{\pi}(\mathcal{R})$ with $\mathcal{A}^{2}$.
The relations from Appendix~\ref{relations} in the algebra $\mathcal{A}$ imply that the $\mathcal{R}$-algebra map
\begin{align}\label{DtoA}
  \Phi \colon  &\mathcal{D}_{\pi}(\mathcal{R}) \rightarrow \mathcal{A} \qquad \mathrm{Lie}_{\xi_{x}} \mapsto e \langle 1 \rangle, \qquad \mathrm{Lie}_{\xi_{y}} \mapsto - f \langle 1 \rangle \qquad \mathrm{Lie}_{\xi_{\eta}} \mapsto \frac{\chi_{C}(\mathcal{V}\otimes\mathcal{L}^{-1})}{n+1}\\
   & \theta_{i}^{\vee x} \mapsto e \langle \gamma_i^{\vee} \rangle  \qquad \theta_{i}^{\vee y}  \mapsto - f \langle \gamma_i^{\vee} \rangle \qquad \theta_{i}^{\vee \eta} \mapsto \gamma_{i}^{\vee}, \;\; i = 1, ..., g. \nonumber
\end{align}
is a homomorphism of algebras.

Analogously, we define an $\mathcal{A}^{(k)2}$-module
\begin{equation*}
     \mathcal{E}_{A}^{(k)} =  q^{-1}\Omega^{1}_{\mathcal{A}^{(k)2}} \otimes H_{A}
\end{equation*}
and consider a super commutative algebra
\begin{equation*}
    \mathcal{R}^{(k)} = \bigwedge\nolimits^{\bullet} \mathcal{E}^{(k)}_{A},
\end{equation*}
which we view as a structure algebra of a super variety $\mathcal{U}_{k}$.
Also there is a natural extension of the symplectic form~\eqref{Poissonbraketk} to a super Poisson structure
\begin{align*}
    &\{ \cdot, \cdot\} \colon q^{-1}\otimes \mathcal{R}^{(k)} \otimes \mathcal{R}^{(k)}  \rightarrow \mathcal{R}^{(k)},
\end{align*}
for which $\theta_{i}^{x_{k}}, \theta_{i}^{y_{k}}, \; i = 1, ..., g$ are Poisson-central elements. The maps of super algebras $\mathcal{R}^{(k)} \rightarrow \mathcal{R}$ can be glued in a natural way to define a super variety with a map
\begin{equation*}
    \mathcal{X}^{\vee} \rightarrow \mathcal{X}_{0}^{\vee}.
\end{equation*}
By Remark~\ref{phik1aspullback} the structure sheaf of $\mathcal{X}^{\vee}$ is given by
\begin{equation*}
\mathcal{O}_{\mathcal{X}^{\vee}} = \bigwedge\nolimits^{\bullet} \Omega_{X^{\vee}} \otimes H_{A}.
\end{equation*}

Using the pairing~\eqref{eq:pairingEk}, for elements $r \in \mathcal{A}^{(k)2}$ and $\gamma_i^{\vee}, \; i = 1, ..., g$, we define a map
\begin{equation*}
        \iota_{\xi_{r}} \otimes \langle \gamma_{i}^{\vee}, \cdot \rangle \colon \mathcal{E}_{A}^{(k)} \rightarrow \mathcal{A}^{(k)2}, 
\end{equation*}
which gives rise to an odd derivation. With a slight abuse of notation~\eqref{eq:thetaveek} we denote this derivation by
\begin{equation*}
    \theta_{i}^{\vee r} \in \mathrm{Der}(\mathcal{R}^{(k)}).
\end{equation*}
The operators $\mathrm{Lie}_{\xi_{x_k}}, \mathrm{Lie}_{\xi_{y_k}}, \theta_{i}^{\vee x_k}, \theta_{i}^{\vee y_k}$ generate a super Lie algebra $\mathrm{Der}(\mathcal{R}^{(k)})$. Analogously to~\eqref{Liealgebroid}, we consider a Lie algebroid
\begin{equation*}
    \mathrm{At}(\mathcal{R}^{(k)}) = \mathcal{R}^{(k)} \oplus \mathrm{Der}(\mathcal{R}^{(k)}),
\end{equation*}
of first-order differential operators on $\mathcal{R}^{(k)}$, also referred to as the \emph{Atiyah algebra} of $\mathcal{R}^{(k)}$.
The algebra of differential operators on $\mathcal{R}^{(k)}$ can be defined as
\begin{equation*}
    \mathcal{D}(\mathcal{R}^{(k)}) = U_{\mathcal{R}^{(k)}}\left(\mathrm{At}(\mathcal{R}^{(k)})\right)/\left( 1_{\mathcal{R}^{(k)}} - 1_{U}\right),
\end{equation*}
where the 'universal enveloping over $\mathcal{R}^{(k)}$' is defined similarly to~\eqref{eq:UR}. This algebra is another description of the Weyl-Clifford algebra in Definition~\ref{WeylCliffalg}. Accordingly, Theorem~\ref{thm:Ak2-triv} can be reformulated as the isomorphism of $\mathcal{R}^{(k)}$-algebras
\begin{align*}
   &\mathcal{D}(\mathcal{R}^{(k)}) \rightarrow \mathcal{A}^{(k)} \qquad \mathrm{Lie}_{\xi_{x_k}} \mapsto e^{(k)} \langle 1 \rangle , \qquad \mathrm{Lie}_{\xi_{y_k}} \mapsto - f^{(k)} \langle 1 \rangle, \qquad \mathrm{Lie}_{ \xi_{\eta} } \mapsto \chi_{C}(\mathcal{M}_{k} \otimes \mathcal{L}^{-1}),\\
   &\theta_{i}^{\vee x_{k}} \mapsto e\langle \gamma_{i}^{\vee} \rangle, \qquad \theta_{i}^{\vee y_{k}} \mapsto - f\langle \gamma_{i}^{\vee} \rangle, \qquad \theta_{i}^{\vee \eta} \mapsto - x_{k} f\langle \gamma_{i}^{\vee} \rangle - y_{k}e\langle \gamma_{i}^{\vee} \rangle, \; i = 1, ..., g.
\end{align*}
We will not make a distinction between these algebras.

By~\eqref{pushiotainchart} and~\eqref{pulliotainchart} the map~\eqref{eq:AtoAk} from $\mathcal{A}$ to $\mathcal{A}^{(k)}$ is defined on degree zero generators by 
\begin{align}
     &e \langle 1 \rangle \mapsto \int\limits_{C} c_{n+1-k} (\mathcal{N}/\mathcal{N}^{(k)} ) \iota^{(k)}_{*} 1 \otimes \cdot =  \nonumber \\
    &\eta^{n+1-k} e^{(k)}\langle 1 \rangle + x_{k} \eta^{n-k}\chi_C(\mathcal{V}/\mathcal{V}^{(k)}\otimes \mathcal{L}^{-1}) - x_{k} (n+1-k)(n-k)\eta^{n-k-1}\Theta  \nonumber \\
    &-(n+1-k)\eta^{n-k} \sum\limits_{i = 1}^{g} \left( e^{(k)}\langle \gamma_i \rangle \gamma_{i}^{\vee} + \gamma_{i} e^{(k)}\langle \gamma_i^{\vee} \rangle \right) \label{e1trans}\\
    &f \langle 1 \rangle \mapsto -\int\limits_{C} c_{k-1} (\mathcal{N}^{(k-1)} \otimes TC) \iota^{(k)*} 1 \otimes \cdot = \nonumber \\
    &\eta^{k-1}f^{(k)}\langle 1 \rangle - y_{k} \eta^{k-2}\chi_C(\mathcal{V}^{(k-1)}\otimes \mathcal{L}^{-1}) + y_{k} (k-1)(k-2)\eta^{k-3}\Theta \nonumber \\
    &-(k-1)\eta^{k-2} \sum\limits_{i = 1}^{g} \left( f^{(k)}\langle \gamma_{i} \rangle \gamma_{i}^{\vee} + \gamma_{i} f^{(k)}\langle \gamma_{i}^{\vee} \rangle \right),  \nonumber \\
    &\chi_{C}(\mathcal{V}\otimes \mathcal{L}^{-1}) \mapsto (n+1)\chi_{C}(\mathcal{M}_{k}\otimes\mathcal{L}^{-1}) + \sum\limits_{l = 1}^{n+1}(m_l - m_k ) \label{f1trans}
\end{align}

This homomorphism allows us to prove the following.
\begin{proposition}\label{pr:DtoA-iso}
    The map~\eqref{DtoA} is an isomorphism.
\end{proposition}
\begin{proof}
We need to prove the injectivity. Consider a composition of $\mathcal{R}$-algebra maps
\begin{align*}
    \varphi^{(1)} \circ \Phi : \mathcal{D}_{\pi}(\mathcal{R}) \rightarrow \mathcal{A}^{(1)},
\end{align*}
where $\varphi^{(1)}$ was defined in~\eqref{eq:AtoAk}. As a map of $R$-modules, it is 
\begin{align*}
\varphi^{(1)} \circ \Phi  \colon &\mathrm{Sym}^{\bullet}\mathrm{Der}_{\pi}(R)\underset{R}{\otimes} \bigwedge\nolimits^{\bullet} \mathcal{E} \rightarrow \mathrm{Sym}^{\bullet}_{\mathcal{A}^{(k)2}} \langle e\langle 1 \rangle , f \langle 1 \rangle \rangle \otimes \bigwedge\nolimits^{\bullet}_{\mathcal{A}^{(k)2}} \langle \theta_{i}^{x_k } , \theta_{i}^{y_k }, \theta_i^{\vee x_k },
    \theta_i^{\vee y_k }, \; i = 1, ..., g \rangle \\
    &\theta_{i}^{x} \mapsto (n+1)x_1^{n}y_1^{n}\theta^{x_{1}}_{i} + n x_1^{n+1}y_1^{n-1}\theta_{i}^{y_1}, \qquad \theta_{i}^{y} \mapsto \theta^{y_{1}}_{i}, \;\; \gamma_{i} \mapsto x_{1}\theta_i^{y_1} + y_1 \theta_i^{x_1}, \; i = 1, ..., g \\
    &\theta_{i}^{\vee x} \mapsto (n+1)x_1^{n}y_1^{n}\theta^{\vee x_{1}}_{i} + n x_1^{n+1}y_1^{n-1}\theta_{i}^{\vee y_1}, \; \; \theta_{i}^{\vee y} \mapsto \theta^{\vee y_{1}}_{i}, \;  \theta_{i}^{\vee \eta} \mapsto x_{1}\theta_i^{y_1} + y_1 \theta_i^{x_1}, \; i = 1, ..., g \\
     &\mathrm{Lie}_{\xi_{x}} \mapsto \mathrm{Lie}_{\xi_{x(x_1 , y_1 )}} + x_1^{n}y_1^{n-1} \sum\limits_{l > 1}(m_l - m_1)\\
    &\mathrm{Lie}_{\xi_{y}} \mapsto \mathrm{Lie}_{\xi_{y_{1}}} \\
    &\mathrm{Lie}_{\xi_{\eta}} \mapsto \mathrm{Lie}_{\xi_{\eta(x_1 , y_1 )}} + \frac{1}{n+1} \sum\limits_{l = 1}^{n+1}(m_l - m_1 ),
\end{align*}
where $\mathrm{Lie}_{\xi_{x(x_1 , y_1 )}}$ and $\mathrm{Lie}_{\xi_{\eta(x_1 , y_1 )}}$ are given by~\eqref{Lietransx} and~\eqref{Lietranseta}.
After localization at $\eta$ it becomes an isomorphism. Since $\mathrm{Der}_{\pi}(R)$ and $\Omega_{R}$ are torsion-free, this implies that the composition, and consequently $\Phi$, is injective.
\end{proof}

From now on, we make no distinction between $\mathcal{D}_{\pi}(\mathcal{R})$ and $\mathcal{A}$. With the help of the formulas~\eqref{fgammatrans},~\eqref{egammatrans},~\eqref{e1trans}, and~\eqref{f1trans}, we can explicitly compute the $\mathcal{R}$-algebra maps
\begin{align}\label{phikmaps}
   \varphi^{(k)} \colon  \mathcal{A} &\rightarrow \mathcal{A}^{(k)} \\
    \theta_{i}^{\vee x} &\mapsto (n+2-k)(x_k y_k )^{n+1-k}\theta^{\vee x_{k}}_{i} + (n+1-k)x_{k}^{n+2-k}y_k^{n-k}\theta^{\vee y_k}_{i}, \; i = 1, ..., g \nonumber \\
    \theta_{i}^{\vee y} &\mapsto k(x_k y_k )^{k-1}\theta^{\vee y_{k}}_{i} + (k-1)x_k^{k-2}y_{k}^{k}\theta^{\vee x_k}_{i}, \;i = 1, ... g \nonumber \\
    \theta_{i}^{\vee \eta} &\mapsto \gamma_{i}^{\vee}, \; i = 1, ..., g \nonumber \\
    \mathrm{Lie}_{\xi_{x}} &\mapsto \mathrm{Lie}_{\xi_{x(x_k , y_k )}} + x_k^{n+1-k}y_k^{n-k} \sum\limits_{l > k}(m_l - m_k) \nonumber \\
    \mathrm{Lie}_{\xi_{y}} &\mapsto \mathrm{Lie}_{\xi_{y(x_k , y_k)}} + x_{k}^{k - 2}y_{k}^{k-1}\sum\limits_{l < k}(m_{l} - m_{k}) \nonumber \\
    \mathrm{Lie}_{\xi_{\eta}} &\mapsto \mathrm{Lie}_{\xi_{\eta(x_k , y_k)}} + \frac{1}{n+1} \sum\limits_{l = 1}^{n+1}(m_l - m_k) \nonumber
\end{align}
Here $\mathrm{Lie}_{\xi_{x(x_k , y_k )}}$, $\mathrm{Lie}_{\xi_{y(x_k , y_k )}}$, and $\mathrm{Lie}_{\xi_{\eta(x_k , y_k)}}$ are the transformations under the change of variables~\eqref{projectk} of Lie derivatives acting on $\mathcal{O}_{X^{\vee}}$. Explicitly, they are given by
\begin{align}
    &\mathrm{Lie}_{\xi_{x(x_k , y_k )}} = (n+2-k)(x_k y_k )^{n+1-k} \mathrm{Lie}_{\xi_{x_k}} + (n+1-k)x_{k}^{2}(x_k y_k )^{n-k}\mathrm{Lie}_{\xi_{y_k}} \nonumber \\
    &- (n+2 - k)(n+1-k)x_k^{n+1-k}y_k^{n-k}\sum\limits_{i = 1}^{g}(\theta_{i}^{y_{k}} \theta_{i}^{\vee x_k } + \theta_{i}^{x_k}\theta_{i}^{\vee y_k }) \label{Lietransx}\\
    &- (n+2-k)(n+1-k)y_{k}(x_{k}y_{k})^{n-k}\sum\limits_{i = 1}^{g}\theta_{i}^{x_k}\theta_{i}^{\vee x_k } - (n+1-k)(n-k)x_{k}^{n+2-k}y_{k}^{n-1-k}\sum\limits_{i = 1}^{g}\theta_{i}^{y_{k}}\theta_{i}^{\vee y_k } \nonumber \\
    &\mathrm{Lie}_{\xi_{y(x_k , y_k )}} = (k-1)x_k^{k-2} y_k^{k} \mathrm{Lie}_{\xi_{x_k}} + k (x_k y_k )^{k-1}\mathrm{Lie}_{\xi_{y_k}}  \nonumber \\
    &-k(k-1)x_k^{k-2}y_k^{k-1}\sum\limits_{i = 1}^{g}(\theta_{i}^{x_k}\theta_{i}^{\vee y_k} + \theta_{i}^{y_k}\theta_{i}^{\vee x_{k}}) \nonumber \\
    &-(k-1)(k-2)x_{k}^{k-1}y_{k}^{k} \sum\limits_{i = 1}^{g} \theta_{i}^{x_k}\theta_{i}^{\vee x_k} - k(k-1)x_k (x_{k}y_{k})^{k-2}\sum\limits_{i = 1}^{g} \theta_{i}^{y_k}\theta_{i}^{\vee y_k} \label{Lietransy}\\
    &\mathrm{Lie}_{\xi_{\eta(x_k , y_k)}} = y_k \mathrm{Lie}_{\xi_{x(x_k , y_k)}} + x_k \mathrm{Lie}_{\xi_{y(x_k , y_k)}} - \sum\limits_{i = 1}^{g}(\theta_{i}^{x_k}\theta_{i}^{\vee y_k} + \theta_{i}^{y_k} \theta_{i}^{\vee x_k}) \label{Lietranseta}
\end{align}
The algebra $\mathcal{A}^{(k)}$ can be viewed as an algebra of differential operators on the supervariety $\mathcal{U}_{k}$. The homomorphisms~\eqref{phikmaps} uniquely define the transition maps
\begin{align*}
     &\mathrm{At}(\mathcal{R}^{(k)})\big|_{\mathcal{U}_{k} \cap \mathcal{U}_{s}} \rightarrow  \mathrm{At}(\mathcal{R}^{(s)})\big|_{\mathcal{U}_{s} \cap \mathcal{U}_{k}}, \qquad k < s\\
     & \left( r , \;\mathrm{Lie}_{\xi_{x_{k}}} \right) \mapsto \left( r  - \frac{1}{y_{k}}\sum\limits_{l = k+1}^{s - 1} (m_{l} - m_{s}) , \; \mathrm{Lie}_{\xi_{x_{k}(x_s , y_s)}} \right)\\
     &\left( r , \; \mathrm{Lie}_{\xi_{y_{k}}} \right) \mapsto \left( r + \frac{1}{x_{k}} \sum\limits_{l = k}^{s - 1} (m_{l} - m_{s}) , \; \mathrm{Lie}_{\xi_{y_{k}(x_{s}, y_{s})}} \right) , \;\; r \in \mathcal{R}^{(k)}.
\end{align*}

Using the commutator relations~\eqref{commk1rest}, we can identify $\xi_{x_{k}}$ and $\xi_{y_{k}}$ with the coordinate basis vectors
\begin{align*}
    \xi_{x_{k}} = -\partial_{y_{k}} \qquad \xi_{y_{k}} = \partial_{x_{k}}.
\end{align*}
Thus, we can rewrite the transition maps in the case $s = k+1$ in the form
\begin{align}\label{transkk+1}
     &\mathrm{At}(\mathcal{R}^{(k)})\big|_{\mathcal{U}_{k} \cap \mathcal{U}_{k+1}} \rightarrow  \mathrm{At}(\mathcal{R}^{(s)})\big|_{\mathcal{U}_{k+1} \cap \mathcal{U}_{k}}  \\
     & \left( r ,\; \mathrm{Lie}_{\partial_{y_{k}}} \right) \mapsto \left( r , \;\mathrm{Lie}_{\partial_{y_{k}}} \right) \nonumber \\
     &\left( r , \mathrm{Lie}_{\partial_{x_{k}}} \right) \mapsto \left( r + \frac{m_k - m_{k+1}}{x_{k}} , \; \mathrm{Lie}_{\partial_{x_{k}}} \right) , \;\; r \in \mathcal{R}^{(k)} \nonumber
\end{align}
Let $\mathcal{O}_{X^{\vee}}(\vec{\mathcal{M}})$ be a $\mathrm{Fr}\times A^{\vee}$-equivariant invertible sheaf of $\mathcal{O}_{X^{\vee}}$-modules, satisfying
\begin{equation}\label{eq:OM}
      \mathcal{O}_{X^{\vee}}(\vec{\mathcal{M}})\big|_{\bar{L}_{k}} \cong \mathcal{O}_{\mathbb{P}^{1}}(m_k - m_{k+1} ), \;\;\; k = 1, ..., n.
\end{equation}
This condition defines $\mathcal{O}_{X^{\vee}}(\vec{\mathcal{M}})$ uniquely up to linearization (i.e. a choice of $\mathrm{Fr}\times A^{\vee}$-weight of a fiber over some fixed point), which we specify in~\eqref{linearization}. Then we define an invertible sheaf
\begin{equation*}
\mathcal{O}_{\mathcal{X}^{\vee}}(\vec{\mathcal{M}}) = \mathcal{O}_{\mathcal{X}^{\vee}}\underset{\mathcal{O}_{X^{\vee}}}{\otimes}  \mathcal{O}_{X^{\vee}}(\vec{\mathcal{M}}) 
\end{equation*}
of $\mathcal{O}_{\mathcal{X}^{\vee}}(\vec{\mathcal{M}})$-modules.
The transition maps~\eqref{transkk+1} glue the algebras $\mathrm{At}(\mathcal{R}^{(k)})$ to a sheaf $\mathrm{At}_{\mathcal{O}_{\mathcal{X}^{\vee}}(\vec{\mathcal{M}})}$
of first-order differential operators on $\mathcal{O}_{\mathcal{X}^{\vee}}(\vec{\mathcal{M}})$. Indeed, the transition maps for $\mathrm{At}_{\mathcal{O}_{\mathcal{X}^{\vee}}(\vec{\mathcal{M}})}$ have the form
\begin{align*}
&\mathrm{At}_{\mathcal{O}_{\mathcal{X}^{\vee}}(\vec{\mathcal{M}})}\big|_{\mathcal{U}_{k}} \rightarrow \mathrm{At}_{\mathcal{O}_{\mathcal{X}^{\vee}}(\vec{\mathcal{M}})}\big|_{\mathcal{U}_{k+1}} \\
&\left( r , \; D \right) \mapsto \left( r + D \log g_{k , k+1}, \; D \right), \;\; r \in \mathcal{R}^{(k)}, \; D \in \mathrm{Der}(\mathcal{R}^{(k)}),
\end{align*}
where $g_{k, k+1}$ is the clutching function
\begin{equation}
\mathcal{O}_{\mathcal{X}^{\vee}}(\vec{\mathcal{M}})\big|_{\mathcal{U}_{k}}\rightarrow \mathcal{O}_{\mathcal{X}^{\vee}}(\vec{\mathcal{M}})\big|_{\mathcal{U}_{k+1}} \qquad s_{k} \mapsto g_{k, k+1}s_{k}.
\end{equation}
The condition~\eqref{eq:OM} forces it to be equal to $g_{k , k+1} = x_{k}^{m_{k} - m_{k+1}}$ up to a multiplicative constant, so we recover~\eqref{transkk+1}.

From the Atiyah algebra sheaf $\mathrm{At}_{\mathcal{O}_{\mathcal{X}^{\vee}}(\vec{\mathcal{M}})}$ we reconstruct the sheaf of differential operators on $\mathcal{O}_{\mathcal{X}^{\vee}}(\vec{\mathcal{M}})$
\begin{equation*}
\mathcal{D}_{\mathcal{O}_{\mathcal{X}^{\vee}}(\vec{\mathcal{M}})} = U_{\mathcal{O}_{\mathcal{X}^{\vee}}}\left(\mathrm{At}_{\mathcal{O}_{\mathcal{X}^{\vee}}(\vec{\mathcal{M}})}\right),
\end{equation*}
where the 'universal enveloping over $\mathcal{O}_{\mathcal{X}^{\vee}}$' is defined similarly to~\eqref{eq:UR}.
The homomorphisms~\eqref{phikmaps} define an inclusion of $\mathcal{A}$ into a space of global differential operators
\begin{equation}\label{eq:AtoGamma}
    \mathcal{A} \rightarrow \Gamma\left(\mathcal{X}^{\vee}, \mathcal{D}_{\mathcal{O}_{\mathcal{X}^{\vee}}(\vec{\mathcal{M}})}\right).
\end{equation}

As an $\mathcal{R}$-algebra,   $\mathcal{A} \cong \mathcal{D}_{\pi}(\mathcal{R})$ is generated by $\mathrm{Lie}_{\xi_{r}}, \; \theta_{i}^{\vee r}, \;\; i = 1, ..., g, \;\; r = x, y, \eta$. One can characterize the image of~\eqref{eq:AtoGamma} by noting that $\xi_{x}, \xi_{y}, \xi_{\eta}$ generate a Lie subalgebra $TX^{\vee}(-\log E)$ of the global vector fields on $X^{\vee}$ preserving the exceptional divisor
\begin{equation*}
    E = \pi^{-1}(0) \subset X^{\vee}.
\end{equation*}
It follows that the connection~\eqref{symcon} gives rise to a flat logarithmic connection
\begin{align*}
   \nabla \colon \mathcal{T}\mathcal{X}^{\vee}(-\log (E)) \rightarrow \mathrm{At}_{\mathcal{O}_{\mathcal{X}^{\vee}}(\vec{\mathcal{M}})}.
\end{align*}

\section{The action on attractor classes via equivariant localization}\label{app:BM-eta}

The following calculation is the only place where the exceptional
conditions are used.
The $T$-action~\eqref{Taction} lifts to the action on $\widetilde{QM\vec{N}}_{\mathcal{M}_{k}^{-1} \otimes \mathcal{K}_{C}}$. We consider a one-dimensional subtorus of $T$, slightly different from~\eqref{1torus}:
\begin{align}\label{newgm}
    &\mathbb{G}_{m} \hookrightarrow T\\
    &t_{l} =
    \begin{cases}
        z^{n+1 - s}, \;k < l \leq s\\
        z^{n+1 - l}, \;\text{otherwise.}
     \end{cases}
\end{align}
Since $h^{0}(\mathcal{M}_{l} \otimes \mathcal{M}_{k}^{-1} \otimes \mathcal{K}_{C}) = 0$ for $k < l < s$, $\mathrm{Attr}'(l) = \varnothing$ for this range of $l$'s, and the $\mathbb{G}_{m}$-action~\eqref{newgm} contracts the space $\{\mathcal{V}^{\bullet}\}$ to the single point $\mathcal{V}_{\mathrm{fix}}^{\bullet}$, namely
$\mathcal{V}_{\mathrm{fix}} = \mathcal{M}_{1} \oplus \cdots \oplus \mathcal{M}_{n+1}$
with the natural filtration. The fiber $F$ over $\mathcal{V}^{\bullet}_{\mathrm{fix}}$ is a projective space
\begin{align}\label{Finclusion}
    F = \mathbb{P}\left( H^{0}(\mathcal{V}_{\mathrm{fix}}\otimes \mathcal{M}_{k}^{-1} \otimes \mathcal{K}_{C})\right) \overset{i_{F}}{\hookrightarrow} \widetilde{QM\vec{N}}_{\mathcal{M}_{k}^{-1} \otimes \mathcal{K}_{C}}.
\end{align}
Because of the vanishing $h^{1}(\mathcal{M}_{l} \otimes \mathcal{M}_{k}^{-1}\otimes \mathcal{K}_{C})$, the attractor $\mathrm{Attr}'(k)$ maps surjectively onto the space $\{\mathcal{V}^{\bullet}\}$ of iterated extensions. The map~\eqref{Finclusion} is quasi-smooth with normal bundle
\begin{align*}
    N_{i_{F}} = \sum\limits_{l < r}\mathrm{Ext}^{1}(\mathcal{M}_{r},
    \mathcal{M}_{l}).
\end{align*}
The fiber is quasi-smooth as well, with virtual tangent bundle
\begin{align*}
    T^{\mathrm{vir}}F = T^{\mathrm{vir}} \widetilde{QM\vec{N}}_{\mathcal{M}_{k}^{-1} \otimes \mathcal{K}_{C}} - N_{i_{F}}
    = \sum\limits_{l = 1}^{n+1} \mathrm{Ext}^{\bullet}(\mathcal{L},
    \mathcal{M}_{l})\otimes \mathcal{O}_{QM\vec{N}}(1)
    -H^{0}(\mathcal{O}_{C}).
\end{align*}
The fiber $F$ has a chain of subspaces
\begin{align*}
    F^{(r)} = \mathrm{Attr}'(r) \cap F = \mathbb{P}\left( H^{0}(\mathcal{V}^{(r)}\otimes \mathcal{M}_{k}^{-1}\otimes \mathcal{K}_{C})\right),\; r = 1, \ldots, n+1,
\end{align*}
which carry a natural derived structure with virtual tangent bundle
\begin{align}\label{TvirFr}
     T^{\mathrm{vir}}F^{(r)} = T^{\mathrm{vir}} \widetilde{QM\vec{N}}_{\mathcal{M}_{k}^{-1} \otimes \mathcal{K}_{C}} - N_{i_{F}}
    = \sum\limits_{l = 1}^{r} \mathrm{Ext}^{\bullet}(\mathcal{L},
    \mathcal{M}_{l})\otimes \mathcal{O}_{QM\vec{N}}(1)
    -H^{0}(\mathcal{O}_{C}).
\end{align}
Since the action is $\pi_{\mathcal{V}^{\bullet}}$-equivariant, the fixed locus 
\begin{align*}
    \mathrm{Fix} \overset{i_{\mathrm{Fix}}}{\hookrightarrow} \widetilde{QM\vec{N}}_{\mathcal{M}_{k}^{-1} \otimes \mathcal{K}_{C}}
\end{align*}
of the $\mathbb{G}_{m}$-action is contained in $F$. Since $F$ is proper, attractors to fixed points cover the space $\widetilde{QM\vec{N}}_{\mathcal{M}_{k}^{-1} \otimes \mathcal{K}_{C}}$, hence the Gysin pullback in $\mathbb{G}_{m}$-equivariant Borel--Moore homology \cite{ARANHA2025110434}
\begin{align*}
    i_{\mathrm{Fix}}^{*}\colon H^{BM, \mathbb{G}_{m}}_{*}(\widetilde{QM\vec{N}}_{\mathcal{M}_{k}^{-1} \otimes \mathcal{K}_{C}}) \rightarrow H^{BM, \mathbb{G}_{m}}_{*}(\mathrm{Fix})
\end{align*}
is injective. Moreover, this map factorizes through the Gysin pullback on the fiber:
\begin{align*}
    i_{F}^{*}\colon H^{BM, \mathbb{G}_{m}}_{*}(\widetilde{QM\vec{N}}_{\mathcal{M}_{k}^{-1} \otimes \mathcal{K}_{C}}) \rightarrow H^{BM, \mathbb{G}_{m}}_{*}(F),
\end{align*}
which implies the injectivity of $i_{F}^{*}$.

Let 
\begin{align*}
    (u_{l} = d_{e}t_{l})_{l = 1}^{n+1}, \;\;\; u = d_{e}z
\end{align*}
be coordinate systems on $\mathrm{Lie}(T)$ and $\mathrm{Lie}(\mathbb{G}_{m})$, respectively.
From~\eqref{TvirFr} we compute
\begin{align*}
    i_{F}^{*} [\mathrm{Attr}'(s)]^{\mathrm{equiv}} &= [F^{(s)}]^{\mathrm{equiv, vir}}=\\
    &=\prod\limits_{l \leq s}(\eta + u_{l} - u_{k})^{h^{1}(\mathcal{M}_{l}\otimes \mathcal{M}_{k}^{-1} \otimes \mathcal{K}_{C})} \prod\limits_{l > s}(\eta + u_{l} - u_{k})^{h^{0}(\mathcal{M}_{l}\otimes \mathcal{M}_{k}^{-1} \otimes \mathcal{K}_{C})} \cap [F]=\\
    &= \eta (\eta + (s-k)u)^{g-1} \prod\limits_{l > s}(\eta + u_{l} - u_{k})^{h^{0}(\mathcal{M}_{l}\otimes \mathcal{M}_{k}^{-1} \otimes \mathcal{K}_{C})} \cap [F],
\end{align*}
\begin{align*}
    i_{F}^{*} [\mathrm{Attr}'(k)]^{\mathrm{equiv}} &= [F^{(k)}]^{\mathrm{equiv, vir}}=\\
    &= \prod\limits_{l \leq k}(\eta + u_{l} - u_{k})^{h^{1}(\mathcal{M}_{l}\otimes \mathcal{M}_{k}^{-1} \otimes \mathcal{K}_{C})} \prod\limits_{l > k}(\eta + u_{l} - u_{k})^{h^{0}(\mathcal{M}_{l}\otimes \mathcal{M}_{k}^{-1} \otimes \mathcal{K}_{C})} \cap [F]=\\
    &= \eta (\eta + (s-k)u)^{g} \prod\limits_{l > s}(\eta + u_{l} - u_{k})^{h^{0}(\mathcal{M}_{l}\otimes \mathcal{M}_{k}^{-1} \otimes \mathcal{K}_{C})} \cap [F].
\end{align*}
We see that
\begin{align*}
    i^{*}_{F}\left( (\eta + (s-k)u ) \cap [\mathrm{Attr}'(s)]^{\mathrm{equiv}} - [\mathrm{Attr}'(k)]^{\mathrm{equiv}}\right)= 0.
\end{align*}
Injectivity of $i^{*}_{F}$ implies
\begin{align*}
    (\eta + (s-k)u ) \cap [\mathrm{Attr}'(s)]^{\mathrm{equiv}} = [\mathrm{Attr}'(k)]^{\mathrm{equiv}}
\end{align*}
in $\mathbb{G}_{m}$-equivariant Borel--Moore homology of $\widetilde{QM\vec{N}}_{\mathcal{M}_{k}^{-1} \otimes \mathcal{K}_{C}}$. The required statement follows by taking $u = 0$.

\bibliographystyle{amsalpha}
\bibliography{bib}

@article{BravermanGaitsgory02,
  author  = {Braverman, Alexander and Gaitsgory, Dennis},
  title   = {Geometric {E}isenstein series},
  journal = {Invent. Math.},
  volume  = {150},
  number  = {2},
  year    = {2002},
  pages   = {287--384},
}

@incollection{Laumon90,
  author    = {G{\'e}rard Laumon},
  title     = {Faisceaux automorphes li{\'e}s aux s{\'e}ries d'Eisenstein},
  booktitle = {Automorphic Forms, Shimura Varieties, and L-functions, Vol. I},
  editor    = {Laurent Clozel and James S. Milne},
  series    = {Perspectives in Mathematics},
  volume    = {10},
  publisher = {Academic Press},
  address   = {Boston, MA},
  year      = {1990},
  pages     = {227--281},
  note      = {Proceedings of the conference held in Ann Arbor, MI, 1988}
}

@misc{Khan19,
  author       = {Khan, Adeel A.},
  title        = {Virtual fundamental classes of derived stacks {I}},
  year         = {2019},
  eprint       = {1909.01332},
  archivePrefix = {arXiv},
  primaryClass = {math.AG},
  url          = {https://arxiv.org/abs/1909.01332}
}

@misc{KazOko23,
  author       = {Kazhdan, David and Okounkov, Andrei},
  title        = {{$L$}-function genera and applications},
  year         = {2023},
  eprint       = {2311.17747},
  archivePrefix = {arXiv},
  primaryClass = {math.NT},
  url          = {https://arxiv.org/abs/2311.17747}
}

@misc{Khan26lec,
      title={Lectures on algebraic stacks}, 
      author={Adeel A. Khan},
      year={2026},
      eprint={2310.12456},
      archivePrefix={arXiv},
      primaryClass={math.AG},
      url={https://arxiv.org/abs/2310.12456}, 
}

@misc{MarianNegut23,
  author       = {Marian, Alina and Negu{\c{t}}, Andrei},
  title        = {The cohomology of the {Q}uot scheme on a smooth curve as a {Y}angian representation},
  year         = {2023},
  eprint       = {2307.13671},
  archivePrefix = {arXiv},
  primaryClass = {math.AG},
  url          = {https://arxiv.org/abs/2307.13671}
}

@article{KamnitzerJoel22,
author = {Kamnitzer, Joel},
title = {Symplectic resolutions, symplectic duality, and Coulomb branches},
journal = {Bulletin of the London Mathematical Society},
volume = {54},
number = {5},
pages = {1515-1551},
doi = {https://doi.org/10.1112/blms.12711},
url = {https://londmathsoc.onlinelibrary.wiley.com/doi/abs/10.1112/blms.12711},
eprint = {https://londmathsoc.onlinelibrary.wiley.com/doi/pdf/10.1112/blms.12711},
year = {2022}
}

@misc{webster20233dimensionalmirrorsymmetry,
      title={3-dimensional mirror symmetry}, 
      author={Ben Webster and Philsang Yoo},
      year={2023},
      eprint={2308.06191},
      archivePrefix={arXiv},
      primaryClass={math-ph},
      url={https://arxiv.org/abs/2308.06191}, 
}

@article{ARANHA2025110434,
title = {Virtual localization revisited},
journal = {Advances in Mathematics},
volume = {479},
pages = {110434},
year = {2025},
issn = {0001-8708},
doi = {https://doi.org/10.1016/j.aim.2025.110434},
url = {https://www.sciencedirect.com/science/article/pii/S0001870825003329},
author = {Dhyan Aranha and Adeel A. Khan and Alexei Latyntsev and Hyeonjun Park and Charanya Ravi}
}

@misc{Kuznetsov1996TheLR,
  author = {Kuznetsov, Alexander},
  title  = {The {L}aumon's resolution of {D}rinfeld's compactification is small},
  year   = {1996},
  note   = {preprint}
}

@article{Ciocan-Fontanine:2011gqf,
    author = "Ciocan-Fontanine, Ionut and Kim, Bumsig and Maulik, Davesh",
    title = "{Stable quasimaps to GIT quotients}",
    eprint = "1106.3724",
    archivePrefix = "arXiv",
    primaryClass = "math.AG",
    doi = "10.1016/j.geomphys.2013.08.019",
    journal = "J. Geom. Phys.",
    volume = "75",
    pages = "17--47",
    year = "2014"
}

@article{CIOCANFONTANINE20103022,
title = {Moduli stacks of stable toric quasimaps},
journal = {Advances in Mathematics},
volume = {225},
number = {6},
pages = {3022-3051},
year = {2010},
issn = {0001-8708},
doi = {https://doi.org/10.1016/j.aim.2010.05.023},
url = {https://www.sciencedirect.com/science/article/pii/S0001870810002173},
author = {Ionuţ Ciocan-Fontanine and Bumsig Kim}
}

@incollection{KimBumsig,
  author    = {Kim, Bumsig},
  title     = {Stable quasimaps to holomorphic symplectic quotients},
  booktitle = {Schubert Calculus---{O}saka 2012},
  series    = {Adv. Stud. Pure Math.},
  volume    = {71},
  publisher = {Math. Soc. Japan, Tokyo},
  year      = {2016},
  pages     = {139--160},
  note      = {\texttt{arXiv:1005.4125}}
}

@article{DrinfeldGaitsgory16,
author = {Drinfeld, Vladimir and Gaitsgory, Dennis},
year = {2016},
month = {10},
pages = {},
title = {Geometric constant term functor(s)},
volume = {22},
journal = {Selecta Mathematica},
doi = {10.1007/s00029-016-0269-3}
}

@article{Bullimore:2016hdc,
    author = "Bullimore, Mathew and Dimofte, Tudor and Gaiotto, Davide and Hilburn, Justin and Kim, Hee-Cheol",
    title = "{Vortices and Vermas}",
    eprint = "1609.04406",
    archivePrefix = "arXiv",
    primaryClass = "hep-th",
    doi = "10.4310/ATMP.2018.v22.n4.a1",
    journal = "Adv. Theor. Math. Phys.",
    volume = "22",
    pages = "803--917",
    year = "2018"
}

@article{Hilburn:2020aau,
    author = "Hilburn, Justin and Kamnitzer, Joel and Weekes, Alex",
    title = "{BFN Springer Theory}",
    eprint = "2004.14998",
    archivePrefix = "arXiv",
    primaryClass = "math.RT",
    doi = "10.1007/s00220-023-04735-4",
    journal = "Commun. Math. Phys.",
    volume = "402",
    number = "1",
    pages = "765--832",
    year = "2023"
}

@article{FFNR11,
	author = {Feigin, Boris and Finkelberg, Michael and Negut, Andrei and Rybnikov, Leonid},
	date = {2011/09/01},
	doi = {10.1007/s00029-011-0059-x},
	id = {Feigin2011},
	isbn = {1420-9020},
	journal = {Selecta Mathematica},
	number = {3},
	pages = {573--607},
	title = {Yangians and cohomology rings of Laumon spaces},
	url = {https://doi.org/10.1007/s00029-011-0059-x},
	volume = {17},
	year = {2011}}

@article{Nakajima15,
author = {Nakajima, Hiraku},
year = {2015},
month = {03},
pages = {},
title = {Towards a mathematical definition of Coulomb branches of $3$-dimensional $\mathcal N=4$ gauge theories, I},
volume = {20},
journal = {Advances in Theoretical and Mathematical Physics},
doi = {10.4310/ATMP.2016.v20.n3.a4}
}

@article{BFN:2016wma,
    author = "Braverman, Alexander and Finkelberg, Michael and Nakajima, Hiraku",
    title = "{Towards a mathematical definition of Coulomb branches of $3$-dimensional $\mathcal{N} = 4$ gauge theories, II}",
    eprint = "1601.03586",
    archivePrefix = "arXiv",
    primaryClass = "math.RT",
    doi = "10.4310/ATMP.2018.v22.n5.a1",
    journal = "Adv. Theor. Math. Phys.",
    volume = "22",
    pages = "1071--1147",
    year = "2018"
}

@article{BFN2017,
author = {Braverman, Alexander and Finkelberg, Michael and Nakajima, Hiraku},
year = {2017},
month = {06},
pages = {},
title = {Ring objects in the equivariant derived Satake category arising from Coulomb branches},
volume = {23},
journal = {Advances in Theoretical and Mathematical Physics},
doi = {10.4310/ATMP.2019.v23.n2.a1}
}

@article{Finkelberg1997GlobalIC,
  title={Global Intersection Cohomology of Quasimaps' Spaces},
  author={Michael Finkelberg and Alexander M. Kuznetsov},
  journal={International Mathematics Research Notices},
  year={1997},
  volume={1997},
  pages={301-328},
  url={https://api.semanticscholar.org/CorpusID:15854670}
}

@article{MACDONALD1962319,
title = {Symmetric products of an algebraic curve},
journal = {Topology},
volume = {1},
number = {4},
pages = {319-343},
year = {1962},
issn = {0040-9383},
doi = {https://doi.org/10.1016/0040-9383(62)90019-8},
url = {https://www.sciencedirect.com/science/article/pii/0040938362900198},
author = {I.G. MacDonald}
}

@book{LecturesNakajima,
  author    = {Nakajima, Hiraku},
  title     = {Lectures on {H}ilbert Schemes of Points on Surfaces},
  series    = {University Lecture Series},
  volume    = {18},
  publisher = {American Mathematical Society},
  address   = {Providence, RI},
  year      = {1999}
}

\end{document}